\documentclass[a4paper,reqno]{amsart}
\usepackage[paperheight=11in,paperwidth=8.5in]{geometry}
\usepackage[utf8]{inputenc}
\usepackage{amsmath,amsfonts,amsthm,amssymb}
\usepackage{dsfont}
\usepackage{graphicx,enumerate}
\usepackage{mathtools,bm,esint}
\usepackage{todonotes}

\usepackage{amsthm,amssymb,amsmath,hyperref}
\usepackage{color, mdframed}
\usepackage{tikz}
\usetikzlibrary{decorations.pathreplacing}
\usepackage{enumitem}

\makeatletter
\def\l@section{\@tocline{1}{12pt plus2pt}{0pt}{}{\bfseries}}
\def\l@subsection{\@tocline{2}{0pt}{2pc}{2pc}{}}
\makeatother
\makeatletter
\def\subsection{\@startsection{subsection}{2}{\z@}%
	{-3.25ex\@plus -1ex \@minus -.2ex}%
	{1.5ex \@plus .2ex}%
	{\normalfont\bfseries\boldmath}}
\def\subsubsection{\@startsection{subsubsection}{3}%
	\z@{.5\linespacing\@plus.7\linespacing}{-.5em}%
	{\normalfont\bfseries\boldmath}}
\renewcommand\paragraph{\@startsection{paragraph}{4}{\z@}%
	{3.25ex \@plus1ex \@minus.2ex}%
	{-1em}%
	{\normalfont\normalsize\bfseries}}
\makeatother

\numberwithin{equation}{section}
\newtheorem{theorem}{Theorem}[section]
\newtheorem*{theorem*}{Theorem}

\newtheorem{lemma}[theorem]{Lemma}
\newtheorem{proposition}[theorem]{Proposition}

\theoremstyle{definition}
\newtheorem{remark}[theorem]{Remark}

\newtheorem{definition}[theorem]{Definition} 

\newcommand{\R}{\mathbb{R}}

\newcommand{\Z}{\mathbb{Z}}
\newcommand{\N}{\mathbb{N}}

\newcommand{\br}[1]{\left(#1\right)}
\newcommand{\Br}[1]{\left[#1\right]}
\newcommand{\bR}[1]{\left(#1\right]}
\newcommand{\BR}[1]{\left\{#1\right\}}
\newcommand{\abs}[1]{\left\vert#1\right\vert}
\newcommand{\nrm}[1]{\left\Vert#1\right\Vert}
\newcommand{\ang}[1]{\left<#1\right>}

\newcommand{\set}[2]{\left\{ #1 \middle.: #2 \right\}}

\newcommand{\1}{\mathds{1}}

\newcommand{\id}{\mathrm{id}}
\newcommand{\str}{\mathrm{str}}
\newcommand{\para}{\mathrm{par}}

\def\beq{\begin{equation}}
	\def\eeq{\end{equation}}

\def\beq{\begin{equation}}
	\def\eeq{\end{equation}}

\def\a{\alpha}

\def\g{\gamma}

\def\F{\mathcal{F}}

\def\l{\lambda}
\def\J{\mathcal{J}}

\def\A{\mathcal{A}}
\def\B{\mathcal{B}}
\def\beq{\begin{equation}}
	\def\eeq{\end{equation}}

\def\beq{\begin{equation}}
	\def\eeq{\end{equation}}

\def\ep{\epsilon}
\def\L{\Lambda}
\def\I{\mathcal{I}}

\DeclareMathOperator{\supp}{supp}
\DeclareMathOperator{\sgn}{sgn}

\DeclareMathOperator{\Mod}{Mod}
\DeclareMathOperator{\Tr}{Tr}
\DeclareMathOperator{\Dil}{Dil}

\DeclareFontFamily{U}{mathx}{}
\DeclareFontShape{U}{mathx}{m}{n}{<-> mathx10}{}
\DeclareSymbolFont{mathx}{U}{mathx}{m}{n}
\DeclareMathAccent{\widehat}{0}{mathx}{"70}
\DeclareMathAccent{\widecheck}{0}{mathx}{"71}

\theoremstyle{plain}
\newtheorem{thm}{Theorem}[section]

\theoremstyle{definition}

\newtheorem{mthm}[thm]{Main Theorem}

\theoremstyle{remark}

\theoremstyle{plain}

\numberwithin{equation}{section}

\theoremstyle{plain} 
\newcommand{\thistheoremname}{}
\newtheorem{genericthm}[thm]{\thistheoremname}

  \newtheorem*{genericthm*}{\thistheoremname}
\newenvironment{namedthm*}[1]
  {\renewcommand{\thistheoremname}{#1}%
   \begin{genericthm*}}
  {\end{genericthm*}}

\def\beq{\begin{equation}}
\def\eeq{\end{equation}}

\def\beq{\begin{equation}}
\def\eeq{\end{equation}}

\def\l{\lambda}

\def\L{\mathcal{L}}

\def\g{\gamma}

\def\a{\alpha}

\def\p{\Phi}

\def\ep{\epsilon}

\def\l{\Lambda}

\def\R{\mathbb{R}}

\def\Z{\mathbb{Z}}

\def\p{\mathcal{P}}

\def\F{\mathcal{F}}
\def\J{\mathcal{J}}
\def\I{\mathcal{I}}

\def\A{\mathcal{A}}
\def\B{\mathcal{B}}

\def\N{\mathbb{N}}

\def\Z{\mathbb{Z}}
\def\beq{\begin{equation}}
\def\eeq{\end{equation}}

\def\beq{\begin{equation}}
\def\eeq{\end{equation}}

\def\g{\gamma}

\def\a{\alpha}

\def\p{\Phi}

\def\ep{\epsilon}

\def\l{\lambda}

\def\R{\mathbb{R}}

\def\Z{\mathbb{Z}}

\def\p{\mathcal{P}}

\def\L{\Lambda}

\def\J{\mathcal{J}}

\def\A{\mathcal{A}}
\def\B{\mathcal{B}}

\def\N{\mathbb{N}}

\def\Z{\mathbb{Z}}
\def\I{\mathcal{I}}
\def\beq{\begin{equation}}
\def\eeq{\end{equation}}

\def\beq{\begin{equation}}
\def\eeq{\end{equation}}

\allowdisplaybreaks
\title[Non-resonant Carleson-Radon Transform]{\Large{On a Carleson-Radon Transform}\\\normalsize{The non-resonant setting}}

\author{Martin Hsu and Victor Lie}
\address{Martin Hsu: Department of Mathematics, Purdue University, 150 N. University St, W. Lafayette, IN 47907, U.S.A.}%
\email{hsu263@purdue.edu}

\address{Victor D. Lie: Department of Mathematics, Purdue University, 150 N. University St, W. Lafayette, IN 47907, U.S.A.  and
The ``Simion Stoilow" Institute of Mathematics of the Romanian Academy, Bucharest, RO 70700, P.O. Box 1-764, Romania}%
\email{vlie@purdue.edu}

\date{\today}

\begin{document}
\maketitle

\begin{abstract}
Given a curve $\vec{\g}=(t^{\a_1}, t^{\a_2}, t^{\a_3})$ with $\vec{\a}=(\a_1,\a_2,\a_3)\in \R_{+}^3$, we define the Carleson-Radon transform along $\vec{\g}$ by the formula
 $$ C_{[\vec{\a}]}f\br{x,y}:=\sup_{a\in\R}\left|\emph{p.v.}\,\int_{\R} f\br{x- t^{\a_1},y-t^{\a_2}}\,e^{i\,a\,t^{\a_3}}\,\frac{dt}{t}\right|\,.$$
We show that in the \emph{non-resonant} case, that is, when the coordinates of $\vec{\a}$ are pairwise disjoint, our operator
$ C_{[\vec{\a}]}$ is $L^p$ bounded for any $1<p<\infty$.

Our proof relies on the (Rank I) LGC-methodology introduced in \cite{LVunif} and employs three key elements:
\begin{itemize}
\item a partition of the time-frequency plane with a linearizing effect on both the argument of the input function and on the phase of the kernel;
\item a sparse-uniform dichotomy analysis of the Gabor coefficients associated with the input/output function;
\item a level set analysis of the time-frequency correlation set.
\end{itemize}
\end{abstract}

\tableofcontents

\section{Introduction}

The main focus of our present paper is the study of singular integral operators that combine the following two distinctive features:
\begin{itemize}
\item a \textsf{Radon-type} behavior, \textit{i.e.}, the integration of the input function is performed over a lower dimensional manifold;

\item a \textsf{Carleson-type} behavior, \textit{i.e.}, the defining object may be expressed as a supremum over a family of linear integral operators with (generalized) modulated singular kernel.
\end{itemize}

For more on the origin and motivation of this topic the reader is invited to consult Section \ref{MotivHistory}. For now, in the desire of keeping this presentation concise, we move directly to

\subsection{Statement of the main result}\label{Mainres}

\begin{mthm} \label{mainresult} \textit{Let $\vec{\g}(t)=(\g_1(t), \g_2(t), \g_3(t))$ be a piecewise differentiable curve in $\R^3$ and define \textsf{the Carleson-Radon transform along $\vec{\g}$} as\footnote{As usual, all the singular integral expression throughout the paper are understood in terms of Cauchy principal value and apriori only defined for Schwartz functions.}
\begin{equation}\label{CRalfa}
        C_{[\vec{\g}]}f\br{x,y}:=\sup_{a\in\R}\left|\int_{\R} f\br{x-\g_1(t),y-\g_2(t)}\,e^{i\,a\,\g_3(t)}\,\frac{dt}{t}\right|\,.
\end{equation}
Taking now\footnote{In the case $\a\notin\Z$ one has to properly define the meaning of $t^{\a}$ as either $|t|^{\a}$ or $\sgn{t}\,|t|^{\a}$.} $\vec{\g}(t)=(t^{\a_1}, t^{\a_2}, t^{\a_3})$ with\footnote{Throughout the paper we denote by $\R_+:=\{x\in\R\,|\,x>0\}$.} $\vec{\a}=(\a_1,\a_2,\a_3)\in \R_{+}^3$ a vector consisting of \emph{pairwise distinct} coordinates we have that
the non-resonant\footnote{\textit{I.e.}, the operator under discussion admits no modulation invariance symmetry.} Carleson-Radon transform $C_{[\vec{\g}]}\equiv^{not}C_{[\vec{\a}]}$ obeys
\begin{equation*}
        \nrm{C_{[\vec{\a}]}f}_{L^p}\lesssim_{\vec{\a},p}\nrm{f}_{L^p}\,,\qquad \forall\:1<p<\infty\,.
\end{equation*}}
\end{mthm}

\begin{remark}\label{Gen} i) The above result can be easily extended to more general non-resonant situations that include the case of  curves of the form $\vec{\g}(t)=(t^{\a_1}, t^{\a_2}, t^{\a_3})$ with $\vec{\a}\in \R_{-}^3\cup \R_{+}^3$ pairwise distinct, or, more generally, the case of curves $\vec{\g}(t)=(\g_1(t), \g_2(t), \g_3(t))$ with $\{\g_i(t)\}_{i=1}^3$ continuous (nondegenerate) homogeneous curves with pairwise distinct degrees.

ii) More importantly, beyond the intrinsic nature of the above result, the main contribution of the present paper lies in the versatility and robustness of the methods developed here: indeed, all of the previous results in the Carleson-Radon realm\footnote{For a comprehensive description of the evolution of this topic we invite the reader to consult the following section.} relied on some ad-hoc arguments that are either extremely sensitive to the analytic/algebraic properties of the phase of the associated multiplier or/and involve kernel regularization techniques that can only work in the higher dimensional versions of \eqref{CRalfa}.\footnote{\emph{I.e.}, requiring $n>1$ in \eqref{CRdn}. For more one may see Remark \ref{Gen1}.} In contrast with all these, building on the LGC-methodology introduced in \cite{LVunif}, our approach here involves a wave-packet analysis having a linearizing effect on both the argument of the input function and on the oscillatory phase of the kernel, thus creating a natural environment for the interaction between the local time-frequency properties of the input/output information. This analysis has the advantage of being (i) adaptable to any dimension, (ii) flexible relative to the properties of $\vec{\g}$ and hence to the induced features of the oscillatory phase of the kernel/multiplier, and, (iii) compatible with the linear wave-packet analysis required when $ C_{[\vec{\g}]}$ enjoys a linear modulation invariance structure. All these arguments are supportive of the fact that our present method offers promising perspectives for obtaining further progress on a number of directions addressing the Carleson-Radon behavior---a topic that will be discussed in the following section.
\end{remark}

\subsection{Motivation and historical background; a hierarchical set of problems}\label{MotivHistory}

We start this section with a brief discussion on the origin and motivation of our problem. This will be delivered as an antithetical story following the evolution on our understanding of the themes revolving around the following two fundamental objects\footnote{For more on the parallelism between the origin and evolution of \eqref{Hilb} and \eqref{Hilbparab} as well as on an extended bibliography we invite the reader to consult the detailed discussion in Section 1.4 of \cite{LVunif}.}:
\begin{itemize}
\item \emph{the classical Hilbert transform}
\begin{equation}\label{Hilb}
        H f (x):= \int_{\R} f(x-t)\,\frac{dt}{t}\,.
\end{equation}

Originally, \eqref{Hilb} took shape in the realm of complex analysis, as an operator that connects the limiting boundary behavior of the real and imaginary part of an analytic function on the upper half plane. This view played an essential role in the first proof of the $L^p,\,1<p<\infty$, boundedness of the Hilbert transform, a celebrated result due to M. Riesz in \cite{RieszHilb}. Later, once that the focus switched to the real variable context, \eqref{Hilb} served as the main prototype for the Calder\'on-Zygmund theory, \cite{CZ1}, \cite{CZ2}, with the latter originating in the study of constant coefficient elliptic differential operators.
\medskip

\item \emph{the Hilbert transform along parabola}
\begin{equation}\label{Hilbparab}
        H_{par} f (x,y):= \int_{\R} f(x-t,y-t^2)\,\frac{dt}{t}\,.
\end{equation}

The origin of \eqref{Hilbparab} is of relative more recent date and can be thought as a story in reverse of \eqref{Hilb}. Indeed, the Hilbert transform along parabola, arisen naturally within the study of constant coefficient parabolic differential operators initiated by F. Jones, E. Fabes and M. Rivi\`ere (\cite{Jon64}, \cite{FR66} and \cite{Fabs67}). The $L^2$ boundedness of \eqref{Hilbparab} is a consequence of Parseval, while the general $L^p$ case, $1<p<\infty$, is a result of the works of A. Nagel, M. Rivi\`ere and S. Wainger, \cite{NRW1},\cite{NRW2}, and later, in a more general context, of E. Stein and S. Wainger, \cite{sw1}.
\end{itemize}

Once at this point, we lift ourselves towards a more conceptual view: \eqref{Hilb} and \eqref{Hilbparab} are both commuting with dilation and translation symmetries. This naturally leads us to the next stage in the complexity hierarchy of our analysis---that of considering objects that have modulation invariance features.

For $a\in\R$, we set now the generalized modulation of order $j\in\N$ by
\begin{equation}\label{modj}
       M_{j,a}f(x):= e^{i a x^j}\,f(x)\,,
\end{equation}
and rermark that if we impose the \emph{linear} modulation invariance structure on \eqref{Hilb}, we then recover the celebrated Carleson operator
\begin{equation}\label{Carl}
       C_1 f(x):= \sup_{a\in\R} |M_{1,a} H M_{1,a}^{*} f (x)|\equiv  \int_{\R} f(x-t)\,\frac{e^{i a(x) t}}{t}\,dt\,,
\end{equation}
where in the latter equality we used the equivalent formulation involving an arbitrary real measurable function $a(\cdot)$ obtained via the classical Kolmogorov--Seliverstov--Plessner linearization procedure. The key motivation for considering \eqref{Carl} goes back to the century-old conjecture of N. Luzin, \cite{Luz}, stating that the sequence of partial Fourier sums of a generic square integrable function is almost everywhere convergent. This conjecture was eventually solved in positive by L. Carleson, \cite{c1}, by providing $L^2-$weak bounds on $C_1$.

Mirroring the linear modulation invariance requirement \eqref{Carl} in \eqref{Hilbparab} we notice\footnote{Throughout the paper we use the notation $M^1_{j,a} f(x,y)=e^{i a x^j}\,f(x,y)$ with the obvious analogue for $M^2_{j,a}$.} that given the extra degree of freedom in $\R^2$ as well as the parabolic structure of the $t-$integrant in \eqref{Hilbparab} the operator
\begin{equation}\label{Carlparab}
       CR_{1} f(x,y):= \sup_{a,b\in\R} \Big|M^1_{1,a}M^2_{1,b} H_{par} (M^1_{1,a})^{*} (M^2_{1,b})^{*} f (x,y)\Big|\equiv  \int_{\R} f(x-t, y-t^2)\,
       \frac{e^{i a(x,y) t+ b(x,y)t^2}}{t}\,dt\,,
\end{equation}
where now $ a(\cdot,\cdot)$ and $b(\cdot, \cdot)$ are real measurable functions on $\R^2$. Thus, we immediately deduce the presence of the linear and quadratic modulation invariance
\begin{equation}\label{Carlparab1}
       CR_{1} \left(M^1_{1,c} M^1_{2,d} M^2_{1,e} f\right)(x,y)= CR_{1} f(x,y)\:\:\:\:\:\:\forall c,\,d,\,e\in\R\,.
\end{equation}

This way, we are naturally led to considering maximal operators with higher order modulation invariance and also in general dimensions, such as
\begin{itemize}
\item the Polynomial Carleson operator
\begin{equation}\label{Cdn}
       C_{d,n}f(x):=\sup_{P\in\mathcal{Q}_{d,n}}\left|\int_{\R^n} f(x-t)\,e^{i P(t)}\,K(t)\,dt\right|,\qquad\,x\in \R^n\,,
\end{equation}

\item the Polynomial Carleson-Radon transform (along the paraboloid)
\begin{equation}\label{CRdn}
       CR_{d,n}f(x,y):=\sup_{P\in\mathcal{Q}_{d,n}}\left|\int_{\R^n} f(x-t,y-|t|^2)\,e^{i P(t)}\,K(t)\,dt\right|,\qquad\,(x,y)\in \R^n\times\R\,,
\end{equation}
\end{itemize}
where in the above $d,n\in\N$ with $\mathcal{Q}_{d,n}$ the set of all real-coefficient polynomials in $n$ variables and of degree less than or equal to $d$ and $K$ a Calder\'on--Zygmund kernel on $\R^n$.

Once at this point, it is worth noticing the presence of several hierarchies:
\begin{enumerate}[label=(H\arabic*)]
\item \label{deduct} firstly, it is not hard to see that the $L^p$ boundedness of \eqref{Cdn} follows from the corresponding bound on \eqref{CRdn} via a standard limiting argument;

\item \label{nr} secondly, unsurprisingly, the resonant case is significantly harder to handle than the corresponding non-resonant case with the former standing for the situation when the operator under analysis encapsulates a modulation invariance structure;

\item \label{dim} lastly and more subtly, a careful, comparative look through the development of the first topic above, suggests\footnote{The philosophy stated in \ref{dim} shapes up on the evidence offered by the Carleson setting but is not yet transparent in the Carleson-Radon case. However, building on the previously gained intuition and on the elements stated in Remark \ref{Gen1}, it is likely that this heuristic also applies to the latter case.} that the one-dimensional case $n=1$, guides and inspires the necessary steps for the general $n-$dimensional case, offering at the conceptual level the road map for the general approach.
\end{enumerate}

In direct relation and consistent with the above enumeration we summarize the evolution of each the themes in \eqref{Cdn} and \eqref{CRdn}.

We start our discussion with the Polynomial Carleson operator:

\begin{itemize}

\item \emph{Historically}, the problem concerning the behavior of \eqref{Cdn} evolved as follows:
\begin{itemize}

\item The $d=1$ case in the one dimensional setting  corresponds to the classical Carleson operator \eqref{Carl} discussed above whose $L^2$-bounds where established by Carleson in \cite{c1}, with the general $L^p$, $1<p<\infty$ case treated by R. Hunt in \cite{hu}. Two other influential proofs of Carleson's result were provided by C. Fefferman, \cite{f}, and M. Lacey and C. Thiele, \cite{lt3}. Finally, consistent with \ref{dim}, the $d=1,\,n\geq 1$ case was treated by P. Sj\"olin in \cite{sj2} by adapting to the $n-$dimensional setting the ideas developed in \cite{c1}.

\item The $d>1$ case was formulated by Stein and originated from two distinct directions: the joint work of E. Stein and S. Wainger on oscillatory integrals/operators in Euclidean setting, \cite{SW70}, and, the joint work of F. Ricci and E. Stein on singular oscillatory integrals on nilpotent groups, see \cite{RS86}, \cite{RS87} and \cite{RS89}. Notice that \eqref{Cdn} may be regarded as a natural extension of the classical Carleson operator \eqref{Carl}.
\end{itemize}
\smallskip
\item Consistent with \ref{nr}, the \emph{non-resonant case} addressing the situation when the supremum in \eqref{Cdn} is taken over polynomials with no linear term was proved for $n=1$ and $d=2$ based on a suitable asymptotic analysis of the kernel $\frac{e^{i y^2}}{y}$, \cite{s2},  and then, for general $n$ and $d$, in \cite{sw} relying on Van der Corput estimates and $TT^{*}$-method. More recently, \cite{LVunif}, a more general non-resonant setting dealing with generalized polynomials/fewnomials with no linear term in the $n=1$ case was treated via the newly introduced LGC-methodology.\footnote{For a more extended bibliography treating particular instances of non-resonant Polynomial Carleson type the reader is invited to consult the detailed Introduction in \cite{LVunif}.} Moreover, as a byproduct of the work in \cite{GALVHybcurves}, one can show that the same methodology also applies to the \emph{linear resonant} one-dimensional case of \eqref{Cdn} where this time the supremum is taken over polynomials/fewnomials with linear but no quadratic term. \smallskip

\item The key insight into the \emph{resonant case} was provided in \cite{LVQuadCarl}, for $n=1$ and $d=2$, and, in \cite{LVPolynCarl}, for $n=1$ and general $d$. This fundamental advance relied on three main ideas: (i) a new perspective on the  higher-order wave-packet theory under the name of \emph{relational time-frequency} analysis---a required \footnote{For a revealing, antithetical discussion on the necessity of this relational time-frequency perspective in the context of higher order modulation invariance as opposed to the standard wave-packet theory applied in the classical (linear) modulation invariance setting, the reader is invited to consult Section 2 in \cite{GALVHybcurves}.} framework for time-frequency problems that involve generalized modulation invariance features, and inspired by the view offered in \cite{Feucp}, (ii) a specialized tile selection algorithm addressing the boundary effect of the time-frequency discretization thus allowing for an exact\footnote{\emph{I.e.}, having an actual equality and not a mere discretized model as was the case with the previous approaches to the Carleson operator.} time frequency decomposition of \eqref{Cdn}---this
    discards the previously employed grid averaging procedures that are not reliable in the relational time-frequency framework due to the poor frequency localization properties of the generalized wave-packets, and, (iii) a novel \emph{local analysis} playing a central role in the time-frequency tile discretization of \eqref{Cdn} and in the concept of a mass of a tile adapted to a spatial region---this allows for the removal of the previously required exceptional set analysis that was directly responsible for global weak-type bounds and hence for the necessity of appealing to interpolation theory.

    Following the philosophy stated in \ref{dim}, one can build on the above ideas in order to obtain various extensions for general $n$ and $d$: see \cite{LVHidimPolyCarl} for the standard translation invariant Calder\'on-Zygmund kernel case, \cite{ZK} for general H\"older continuous Calder\'on–Zygmund kernels and, more recently, \cite{BDJST} for generalized Carleson operators on doubling metric measure spaces.
\end{itemize}

We now move our focus on the Polynomial Carleson-Radon transform:

\begin{itemize}
\smallskip
\item \emph{Historically}, this topic was proposed by L. Pierce and P. Yung in \cite{PY19} in direct relation with the works of Stein and Wainger, \cite{sw}, and Stein and Ricci, \cite{RS87}, and as a natural Radon analogue of the Polynomial Carleson operator discussed above. Taking in account \ref{deduct}, it is not surprising that the progress made in the Carleson-Radon context was moderate and, until very recently, relying on some quite restrictive extra hypothesis. Indeed, in direct relation with the discussion at \ref{nr}, and given the presence of supplementary difficulties in the one dimensional setting that will be addressed in Remark \ref{Gen1} below, our knowledge on this topic developed as follows:
\smallskip
\item The \emph{non-resonant case} evolved first in the $n>1$ setting, \cite{PY19}, and imposed that the supremum in \eqref{CRdn} be taken over a restricted class of polynomials with no linear terms, no quadratic resonance terms and having a uniform homogeneous degree behavior.\footnote{For a precise description of the restrictive hypothesis we invite the reader to consult \cite{PY19}.} In particular, this hypothesis places the operator under analysis in the purely non-zero curvature setting, that is, in the absence of a modulation invariance structure. This elaborate work was later extended in \cite{AMPY24} to the case when the $t$-integration is performed along the manifold $\gamma(t)=(t, Q(t)),\, t\in \R^n$ with $Q(t)=t^{T}\,Q\,t$, $Q\in M_n(\R)$, a generic non-degenerate quadratic form\footnote{Notice that the work in \cite{PY19} addresses precisely the special case $Q(t)=|t|^2$.} but under similar restrictions on the supremum with the ones above. Both these contributions employ kernel regularization, successive application of the $TT^{*}-$method and a refinement of the techniques developed by Stein and Wainger in \cite{sw} for treating the non-resonant Polynomial Carleson operator.

    The one-dimensional correspondent was until very recently widely open with only very particular toy models---confined to one-variable dependence of the linearizing supremum function applied to monomial phases---being addressed in \cite{GPRY17}.
\medskip

\item The \emph{resonant case} has only been attempted in the $n>1$ setting\footnote{We exclude here some simple one-dimensional toy models that use as a black-box the (Polynomial) Carleson operator.}, \cite{Bcarlrad}, again under restrictions on both the structure of the linearizing function(s) and on the set of admissible polynomials. This is the first relevant work in the Carleson-Radon context which requires modulation invariant techniques; it employs a kernel regularization argument together with methods that are inspired from the previous contributions made in the resonant Polynomial Carleson case. For some other, earlier related model problems one may consult \cite{R19} and \cite{B24}.
\end{itemize}

\begin{remark}\label{Gen1}[\textsf{The Carleson-Radon realm: $n=1$ versus $n>1$ case}] $\:\:\:$In what follows, we motivate on two levels why the singular behavior of the Carleson-Radon transform is strictly harder in the $n=1$ case than in the $n>1$ case:
\begin{itemize}
\item[-] at the heuristic level, a higher value of $n$ creates more ``room" for the cancellations induced by the phase of the kernel to manifest, and hence, the ``bad set" where the value of $CR_{d,n}f$ is ``large" has a smaller relative size;\footnote{This heuristic was offered in Section 1.3.4 of the longer, 2021 arxiv version of \cite{LVBilHilbCarl} which can be found at https://arxiv.org/abs/2106.09697. In the same section, the question regarding the boundedness properties of $C_{[\vec{\g}]}$ in the specific monomial curved case $\vec{\g}(t)=(t,t^2,t^3)$ is explicitly raised.}

\item[-] at the factual level, a manifestation of the above heuristic can be visualized in the following dichotomy: the $T T^{*}$ kernel regularization effect that is present in the higher dimensional case---and crucially exploited in all of the above works addressing $CR_{d,n}$ with $n>1$, breaks in the one dimensional setting. Concretely, the integral kernel of the operator $CR_{d,n}CR_{d,n}^{*}$ may be expressed as an integral expression over an $n-1$ dimensional sub-manifold inside $\R^{2n}$, which, in the $n=1$ case presents no smoothing effect.\footnote{This observation has also been noticed in \cite{BGH24}.}
\end{itemize}
\end{remark}

With this overall historical perspective in place, we are now ready to focus on the most recent developments: While the results of the present paper---intended as a first general\footnote{\emph{I.e.}, with no restrictions on the structure of the linearizing function.} treatment of the non-resonant one dimensional monomial case of \eqref{CRdn}---were finalized and in the process of being written, another work, \cite{BGH24}, addressing the particular case $(\a_1,\a_2,\a_3)=(1,2,3)$ of our Main Theorem \ref{mainresult}, was published online. In this latter, surprising work, the authors therein treat the moment curve case via completely different methods, relying on (small cap) decoupling theory, and, in a key fashion, on a complex companion paper, \cite{CGGHIW24}. This latter paper focusses on the study of variable coefficient Schr\"odinger-type operators and involves local smoothing estimates, polynomial partitioning and level set analysis relying in turn fundamentally on works such as \cite{Bo91}, \cite{W05}, \cite{BoGu11}, \cite{Gu16}, \cite{Gu18}, and \cite{KR18}. Now in order for all of the above machinery to be successful, in \cite{BGH24} one has to verify that the phase of the multiplier corresponding to the main operator obeys, besides some suitable H\"ormander-type conditions, also a so-called Nikodym non-compression hypothesis which seems to be quite sensitive to the algebraic/differential properties of the curve $\vec{\g}(t)=(t^{\a_1}, t^{\a_2}, t^{\a_3})$ thus explaining in part why the treatment therein is confined to the moment curve case $(\a_1,\a_2,\a_3)=(1,2,3)$. Also, in the same paper, the authors inquire on the possibility of obtaining a direct\footnote{\emph{I.e.}, not appealing to real interpolation.} $L^2$-smoothing inequality.\footnote{See Section 1.4. in \cite{BGH24}.}

In this short\footnote{The main ingredient, \emph{i.e.}, the control on the high-frequency stationary component is covered completely in Section \ref{mainsect}.} and self-contained paper we obtain a general treatment of the non-resonant one dimensional monomial case--see Main Theorem \ref{mainresult}--in particular solving both issues raised in the above paragraph. More importantly, as already pointed in Remark \ref{Gen}, the ideas developed here are likely to extend to a significantly more general setting motivating us to end this section by laying out the following:
\medskip

\noindent\textsf{A hierarchical set of problems addressing the general Carleson-Radon theme}
$\newline$

In direct relation to \ref{nr}--\ref{dim} and Remark \eqref{Gen1}---which, in particular invite us to a subdivision according to $n=1$ and $n\geq 2$---we propose here, in order of complexity, a natural hierarchy of steps having as a final aim the global understanding of \eqref{CRdn}:
 \begin{itemize}
 \item The \emph{non-resonant} case: In view of the present work, using the more general formulation within Main Theorem \ref{mainresult}, the natural remaining step in the one-dimensional setting is given by the choice $\vec{\g}(t)=(t^{\a_1}, t^{\a_2}, \sum_{j=3}^d a_j t^{\a_j})$, $d\geq 3$, where here we require $\{\a_j\}_{j=1}^d$ pairwise distinct and the supremum in \eqref{CRalfa} be taken over all the coefficients $\{a_j\}_{j=3}^d\subset \R$. The higher dimensional ($n\geq 2$) analogue formulation (with the obvious modifications) is also still open in view of the uniform homogeneous degree restrictive hypothesis in \cite{PY19}, \cite{AMPY24} discussed earlier.

     We believe that both of the above problems are amenable to an approach inspired from the methods developed in the present paper and those in \cite{LVunif}.
\medskip

 \item The \emph{resonant} case: In this situation, choosing to only focus on the one-dimensional case, and, appealing again to the framework in our main result, we list\footnote{This list was mentioned in Section 1.3.4 of the extended, arxiv version of \cite{LVBilHilbCarl} mentioned earlier.} the following natural hierarchy of problems:
 \begin{itemize}
 \item  firstly, is the case $\vec{\g}(t)=(t, t^{2}, t^{2})$ that admits the linear resonance $C_{[(1,2,2)]}(M_{1,a}^2 f)=C_{[(1,2,2)]}(f)$ for any $a\in\R$;

 \item secondly, is the case  $\vec{\g}(t)=(t, t^{2}, t)$ that admits both the linear resonance $C_{[(1,2,1)]}(M_{1,a}^1 f)=C_{[(1,2,1)]}(f)$ for any $a\in\R$ and the joint linear/quadratic resonance $C_{[(1,2,1)]}(M_{2,b}^1 M_{1,b}^2 f)=C_{[(1,2,1)]}(f)$ for any $b\in\R$;

 \item thirdly, is of course the general case $\vec{\g}(t)=(t^{\a_1}, t^{\a_2}, \sum_{j=3}^d a_j t^{\a_j})$, $d\geq 3$, with no restrictions on the set of exponents $\{\a_j\}_{j=1}^d$ thus allowing generalized/higher order resonances.
 \end{itemize}
All the above problems would obviously require (generalized) modulation-invariance techniques, which, while expected to build on the methods in \cite{c1}, \cite{f}, \cite{lt3} and finally \cite{LVPolynCarl}, are likely to require some genuine new ideas.
\end{itemize}

\subsection{Plan of the proof; main ideas}

For simplicity and expository reasons, most of this paper will focus on the treatment of the moment curve case, \textit{i.e.} $\vec{\a}=(1,2,3)$, leaving only for the very end---see Final Remarks, Section \ref{genalfa}---the mild adaptations required for the general non-resonant $C_{[\vec{\a}]}$ case, thus, reinforcing the versatility of our methods.

For convenience and later reference, setting for notational simplicity $C:=C_{[(1,2,3)]}$, we rephrase our Main Theorem \ref{mainresult} in this particular setting:

\begin{theorem}\label{thm_main} With the previous notations, we have
    \begin{equation*}
        \nrm{Cf}_{L^p}\underset{p}{\lesssim}\nrm{f}_{L^p},\quad\forall\: p\in\br{1,\infty}.
    \end{equation*}
  \end{theorem}

Next, a rather standard (non)stationary phase analysis---see Section \ref{Discr}, leads to a natural discretization of the form $C=C_L\,+\,C_{H,NS}\,+\,C_{H,S}$ where
\begin{itemize}
\item $C_L$ stands for the low-frequency component---treated in Section \ref{LF};

\item $C_{H,NS}$ stands for the high-frequency non-stationary component---handled in Section \ref{HNS};

\item $C_{H,S}$  stands for the high-frequency stationary component---this represents the main term in the above decomposition whose treatment is initiated in the remainder of Section \ref{Genconsid} and completed in Section \ref{mainsect}.
\end{itemize}

Indeed, focusing now on the last item above, in Section \ref{HS1}, we exploit the dilation symmetry and further reduce the analysis of $C_{H,S}$ to that of the unit scale operator defined by
\begin{equation}\label{unitsc}
        \mathcal{C}^{\br{a}}f\br{x,y}:=
        \int
            f\br{x-t,y-t^2}
            e\br{a\br{x,y}t^3}
            \rho\br{t}
        dt.
    \end{equation}
where here \(a:\R^2\to\R\) is some arbitrary measurable function and \(\rho\in C^\infty_c\br{\R}\), $\rho\geq 0$ with $\textrm{supp}\,\rho\subseteq (\frac{1}{2},2)$.

Once at this point, Theorem \ref{thm_main} is, via some standard interpolation techniques, a consequence of
\medskip
\begin{theorem}\label{thm_main_loc}
There is a universal constant \(\sigma >0\) such that for any parameter \(\lambda>0\), function \(f\in L^2\br{\R^2}\), measurable set \(E\subset\R^2\) and measurable function \(a:\R^2\to \R\) satisfying the property
\begin{equation}\label{eq_f_a_E_cond}
    \supp \widehat{f}\times a\br{E}\subset
    \Br{-2\lambda,2\lambda}^3 \setminus \Br{-\lambda,\lambda}^3,
\end{equation}
we have the following decay estimate:
    \begin{equation}\label{gl}
        \abs{
            \ang{\mathcal{C}^{\br{a}}f,\1_E}
        }
        \lesssim
        \lambda^{-\sigma}
        \nrm{f}_{L^2}\abs{E}^{\frac{1}{2}}.
    \end{equation}
\end{theorem}
\medskip
This is the central piece of the entire paper. The proof of Theorem \ref{thm_main_loc}---covered in full in Section \ref{mainsect}---involves the three main steps\footnote{For a more detailed discussion please see the introduction in Section \ref{mainsect}.} pertaining to the  LGC-methodology: \emph{(i)} a partition of the time-frequency plane with a linearizing effect on both the argument of the input function and on the phase of the kernel, \emph{(ii)} a Gabor (local Fourier) decomposition of the input, and, \emph{(iii)} a time-frequency correlation analysis.

Applying a duality argument, and as a consequence of \emph{(i)} and \emph{(ii)}, in Section \ref{linsec} we rephrase the main goal \eqref{gl} in Theorem \ref{thm_main_loc} as follows
\begin{equation}\label{keyIntr}
    \Lambda\br{F,G}\lesssim\lambda^{-\sigma}\nrm{f}_{L^2}\abs{E}^{\frac{1}{2}}\,,
\end{equation}
where\footnote{Given a finite index set \(I\), we define the expectation \(\displaystyle{\mathop{\mathbb{E}}_{i\in I}a_i:=\br{\# I}^{-1}\sum_{i\in I}a_i}\).}
\begin{equation}\label{linwavemodelIntr}
    \Lambda\br{F,G}:=
        \mathop{\mathbb{E}}_{
        \frac{\sqrt{\lambda}}{3}
        \leq r\leq 3
        \sqrt{\lambda}
        }
    \sum_{\substack{p,q\\
        u,v,w
        }}
        \frac{
            F_{p-r,q-\frac{r^2}{\sqrt{\lambda}},u,v}
            G_{p,q,r,u,v,w}
        }{
            \ang{
                3w\frac{r^2}{\lambda}
                -2v\frac{r}{\sqrt{\lambda}}
                -u
            }^N
        }
\end{equation}
with  \(F,G\) representing a collection of Gabor coefficients at the linearizing scale $(\sqrt{\l})^{-1}$ given by
\begin{equation}\label{GabcoefIntr}
    F_{p,q,u,v}:=\abs{
        \ang{f,\Phi_{p,q,u,v}}
    },\quad
    G_{p,q,r,u,v,w}:=
    \abs{
        \ang{
            \overline{e\br{\frac{ar^3}{\lambda^{\frac{3}{2}}}}}\chi\br{\frac{a}{\sqrt{\lambda}}-w}\1_E,
            \Psi_{p,q,u,v}
        }
    }\,,
\end{equation}
and all the involved parameters being (positive) integers of size $\lesssim \sqrt{\lambda}$.

With these settled, we are now ready to state the \emph{philosophy} behind our approach: the curvature feature of the operator $C$ is encapsulated via the linearizing wave-packet decomposition into suitable time-frequency dependencies that can be now visualized in \eqref{linwavemodelIntr} in the \emph{non-linear entanglement} between the physical parameters \(p,q,r\) and the frequency parameters \(u,v,w\) as revealed by
\begin{itemize}
\item the structure of the Gabor coefficient
\begin{equation}\label{Gcoefcurv}
F_{p-r,q-\frac{r^2}{\sqrt{\lambda}},u,v}\,;
\end{equation}

\item the structure of the weight function
\begin{equation}\label{weightcurv}
 \ang{3w\frac{r^2}{\lambda}-2v\frac{r}{\sqrt{\lambda}}-u}^{-N}\,.
 \end{equation}
\end{itemize}

Understanding this entanglement is the key in the application of the LGC methodology. The examination of these entanglements is naturally leading to the size-analysis of the Gabor coefficients in \eqref{GabcoefIntr} that is further expressed as a \emph{sparse-uniform dichotomy}. The latter is a consequence of the following heuristic pertaining to the two possible extreme scenarios:
\begin{itemize}
\item the large-size Gabor coefficient scenario imposes the existence of only few such coefficients and hence this situation is referred to as the \emph{sparse} case;

\item the small-size Gabor coefficient scenario is controlled by the equally distributed/expected size Gabor coefficients scenario which is further referred to as the \emph{uniform} case.
\end{itemize}

In both of the above scenarios the presence of the weight function \eqref{weightcurv} is critical. In contrast with the latter, the structure of the Gabor coefficient \eqref{Gcoefcurv} can only be exploited in the sparse case scenario via the so called \emph{time-frequency correlation set} whose understanding requires a level set analysis.

Thus, in a nutshell, the two main ingredients in the resolution of our key estimate \eqref{keyIntr} will be given by
\begin{itemize}
\item a sparse-uniform dichotomy analysis regarding the relative size/cardinality of the Gabor coefficients \eqref{GabcoefIntr} that is performed in Section \ref{SpUnifDich};

\item a level set analysis concerning the time-frequency correlation set that is delivered in Section \ref{LevSet}.
\end{itemize}

\subsection*{Acknowledgments}
The first author received partial summer and travel support from the NSF grant DMS-2400238. The second author was supported by the NSF grant DMS-2400238 and by the Simons Travel grant MPS-TSM-00008072.

\section{Some notation and basic facts}\label{Not}

In what follows we will use \(e\br{z}:=e^{2\pi i z}\) for the standard character over $\R$ and $\ang{\cdot}$ for the Japanese bracket defined by $\ang{x}:=\br{1+\abs{x}^2}^{\frac{1}{2}},\quad x\in\R^n$.

Next,  for \(z,x,\xi\in\R\), \(\lambda,p >0\) and \(g\) a function on \(\R\) we introduce the standard translation, dilation and modulation operators:
    \begin{equation}\label{symm}
        \Tr_x g\br{z}:=g\br{z-x},\quad
        \Dil^p_\lambda g\br{z}:= \lambda^{-\frac{1}{p}}g\br{\frac{z}{\lambda}},\quad
        \Mod_\xi g\br{z}:=e\br{\xi z}g\br{z}.
    \end{equation}

We will also make frequent use of the Littlewood-Payley projections defined as follows:
for \(k\in\Z\), \(j\in\N\), and \(g\) a Schwartz function on \(\R\), we set
    \begin{equation}\label{projLP}
        \widehat{P_{k,j}g}:=
        \phi_{k,j}\widehat{g}
        \quad\textrm{with}\quad
        \phi_{k,j}:=
        \left\{
        \begin{aligned}
            \Dil^\infty_{2^k}\varphi, & \quad j=0\\
            \Dil^\infty_{2^{k+j}}\psi, & \quad j\in\N
        \end{aligned}
        \right.\,,
    \end{equation}
where in the above \(\varphi\in C^\infty_c\br{\R}\) is a smooth bump function that is even and satisfies \(\1_{\Br{-1,1}}\leq \varphi\leq \1_{\br{-2,2}}\) while \(\psi:=\varphi-\Dil^\infty_{1/2}\varphi\).

Finally, our reasonings will involve the following classical operators:
\begin{itemize}
\item the strong maximal operator
    \begin{equation*}
        M_\str f\br{x,y}:=
        \sup_{R,R'>0}
            \fint_{-R}^R
                \fint_{-R'}^{R'}
                \abs{f}\br{x-s,y-t}
            dsdt\,;
    \end{equation*}
\item the maximal operator along parabola
    \begin{equation*}
        M_\para f\br{x,y}:=
        \sup_{R>0}
            \fint_{-R}^R
                \abs{f}\br{x-t,y-t^2}
            dt\,;
    \end{equation*}

\item the maximally truncated Hilbert transform along parabola
    \begin{equation*}
        H^\ast_\para f\br{x,y}:=
        \sup_{R>r>0}
            \abs{
                \int_{r<\abs{t}\leq R}
                    f\br{x-t,y-t^2}
                \frac{dt}{t}
            }\,.
    \end{equation*}
\end{itemize}

Classical, well-known results state that all of the above three operators \(M_\str, M_\para, H^\ast_\para\) are of strong type \(\br{p,p}\) for every \(p\in\br{1,\infty}\).

\section{The decomposition $C=C_L\,+\,C_{H,NS}\,+\,C_{H,S}$. Reduction to Theorem \ref{thm_main_loc}}\label{Genconsid}

In this section, we cover the following steps:
\begin{itemize}
\item decomposition of our initial operator into three components: low frequency, high frequency non-stationary and high frequency stationary components;

\item treatment of the low frequency and  high frequency  non-stationary components;

\item provide general $L^p$ bounds, $1<p<\infty$, for the  high frequency stationary component  \emph{assuming} its single scale $L^2$-decay control.
\end{itemize}

Thus, in a nutshell, in this section we prove Theorem \ref{thm_main} assuming Theorem \ref{thm_main_loc}.

\subsection{Discretization of $C$}\label{Discr}

Recalling \eqref{projLP}, we start by introducing the following definitions\footnote{Unless otherwise stated, throughout this paper we work with the French school convention: the set of natural numbers contains the origin, \textit{i.e}, $\N:=\{0,1,2\ldots\}$.}: for $k\in\Z$ and $\vec{j}=(j_1,j_2,j_3)\in\N^3$ we let\footnote{By convention, if $j_3=0$, then in \eqref{Ckj} the preimage set is taken to be
$\1_{\abs{a}^{-1}[0,2^{3k+1}]}\br{x,y}$.}

\begin{equation}\label{Ckj}
C_{\vec{j},k}f\br{x,y}:=\1_{\abs{a}^{-1}\bR{2^{3k+j_3},2^{3k+j_3+1}}}\br{x,y}\,
            \int
                \br{P_{k,j_1}\otimes P_{2k,j_2}}f\br{x-t,y-t^2}
                e\br{a\br{x,y}t^3}
                \psi\br{2^k t}
            \frac{dt}{t}\,,
\end{equation}
\begin{equation}\label{Cj}
C_{\vec{j}}f:=\sum_{k\in\Z} C_{\vec{j},k}f\,,
\end{equation}
and notice that
\begin{equation*}
    Cf=\sum_{\vec{j}\in \N^3}C_{\vec{j}}f=\sum_{k\in\Z\atop{\vec{j}}\in \N^3}C_{\vec{j},k}f\,.
\end{equation*}

Once at this point, it is worth visualizing the multiplier form of the operator $ C_{\vec{j}}$ introduced above; indeed, focusing on the more subtle case $j_3\not=0$, we have
\begin{equation}\label{multform}
    C_{\vec{j}}f\br{x,y}:=
    \sum_{k\in\Z}
        \int
            m_{\vec{j}}\br{\frac{\xi}{2^k},\frac{\eta}{2^{2k}},\frac{a\br{x,y}}{2^{3k}}}\widehat{f}\br{\xi,\eta}
            e\br{x\xi+y\eta}
        d\xi d\eta\,,
\end{equation}
where the symbol is given by
\begin{equation}\label{mult}
    m_{\vec{j}}\br{\xi,\eta,a}:=
    \phi_{0,j_1}\br{\xi}\phi_{0,j_2}\br{\eta}\1_{\bR{2^{j_3},2^{j_3+1}}}\br{a}
    \int
        e\br{a t^3-\eta t^2-\xi t}
        \frac{\psi\br{t}}{t}
    dt.
\end{equation}

Dictated by the form of the expressions \eqref{multform} and \eqref{mult} in relation to the applicability of the (non)stationary phase principle, we partition the index-set $\N^3$ as follows:
\begin{equation*}
 \N^3=J_{0}\cup J_{ns}\cup J_s,\quad\textrm{where}
\end{equation*}
\begin{itemize}
\item the $0$-component is given by $J_0:=\N^2\times\{0\}$;

\item the non-stationary component is given by
\begin{equation*}
    J_{ns}:=\set{\vec{j}\in J\setminus J_0}{
        \,2^{\vert \vec{j}\vert_\infty}\lesssim 1+\hspace{-1.5em}
        \inf_{\substack{
            \frac{1}{2}\leq\abs{t}\leq 2\\
            \:\:\phi_{0,j_1}\br{\xi}\phi_{0,j_2}\br{\eta}\neq 0\\
            \:\:\:\:\:2^{j_3}\leq\abs{a}\leq 2^{j_3+1}
        }}\hspace{-1.5em}
            \abs{3at^2-2\eta t-\xi}
    },\quad
    \vert\vec{j}\vert_{\infty}:=\max\br{j_1,j_2,j_3}.
\end{equation*}

\item the stationary component is given by $J_s:=\N^3\setminus (J_{0}\cup J_{ns})$.
\end{itemize}
With these, according to the above index partitioning, we decompose our operator as follows:
\begin{equation*}
    C=C_L\,+\,C_{H,NS}\,+\,C_{H,S}\,,\quad\textrm{where}
\end{equation*}
\begin{itemize}
\item the \emph{low frequency} component
\begin{equation}\label{Low}
    C_L:=\sum_{\vec{j}\in J_0} C_{\vec{j}}\,;
\end{equation}
\item the \emph{high frequency  non-stationary} component
\begin{equation}\label{hns}
    C_{H,NS}:=\sum_{\vec{j}\in J_{ns}} C_{\vec{j}}\,;
\end{equation}
\item the \emph{high frequency stationary} component
\begin{equation}\label{hs}
    C_{H,S}:=\sum_{\vec{j}\in J_{s}} C_{\vec{j}}\,.
\end{equation}
\end{itemize}

\subsection{Treatment of the low frequency component $C_L$}\label{LF}

In this section we are proving the following

\begin{lemma}\label{lem_ptwise_est} Recalling the notation in Section \ref{Not}, we have the pointwise estimate
    \begin{equation*}
        \abs{C_L f}\lesssim H^\ast_\para f+M_\para f\,.
    \end{equation*}
Consequently, we deduce that
\begin{equation*}
    \nrm{C_L f}_{L^p}
    \underset{p}{\lesssim}
    \nrm{f}_{L^p},\quad \forall\, p\in\br{1,\infty}.
\end{equation*}
    \end{lemma}

\begin{proof}

Via a summation by part in \(k\in\Z\), we can rewrite the expression of $C_L$ as
\begin{align*}
    C_L f\br{x,y}= &
    \1_{\abs{a}^{-1}\BR{0}}\br{x,y}
    \cdot
    \int f\br{x-t,y-t^2}\frac{dt}{t}\\
    +\sum_{k\in\Z} &
        \1_{\abs{a}^{-1}\bR{2^{3k-2},2^{3k+1}}}\br{x,y}
        \cdot
        \int
            f\br{x-t,y-t^2}
            e\br{a\br{x,y}t^3}
            \varphi\br{2^k t}
        \frac{dt}{t}.
\end{align*}
In other words, for fixed \(\br{x,y}\in \R^2\), either \(a\br{x,y}=0\) and
\begin{equation*}
    \abs{C_Lf\br{x,y}}=
    \abs{\int f\br{x-t,y-t^2}\frac{dt}{t}}
    \leq H^\ast_\para f\br{x,y},
\end{equation*}
or there's a unique \(k\in\Z\) such that:
\begin{equation*}
    C_L f\br{x,y}=
    \int
        f\br{x-t,y-t^2}
        e\br{a\br{x,y}t^3}
        \varphi\br{2^k t}
    \frac{dt}{t}
\end{equation*}
with \(\abs{a\br{x,y}}\in \bR{2^{3k-2},2^{3k+1}}\). We now rewrite the expression:
\begin{align*}
    C_L f\br{x,y}= &
    \int
        f\br{x-t,y-t^2}
        \varphi\br{2^k t}
    \frac{dt}{t} & \leadsto I_1\\
    + &
    \int
        f\br{x-t,y-t^2}
        \frac{e\br{a\br{x,y}t^3}-1}{t}
        \varphi\br{2^k t}
    dt & \leadsto I_2.
\end{align*}
For \(I_1\), recall that \(\1_{\Br{-1,1}}\leq \varphi\leq \1_{\br{-2,2}}\) and thus \(0\leq \varphi-\1_{\Br{-1,1}}\leq \1_{\br{-2,2}\setminus \Br{-1,1}}\). As a result,
\begin{align}\label{I1}
    \abs{I_1}
    \leq &
    \abs{
        \int
            f\br{x-t,y-t^2}
            \1_{\Br{-1,1}}\br{2^k t}
        \frac{dt}{t}
    }
    +
    \int
        \abs{f}\br{x-t,y-t^2}
        \abs{
            \varphi-\1_{\Br{-1,1}}
        }
        \br{2^k t}
    \frac{dt}{\abs{t}}\nonumber\\
    \lesssim &
    H^\ast_\para f\br{x,y}+M_\para f\br{x,y}.
\end{align}
For \(I_2\), we use that \(\abs{e\br{z}-1}\lesssim \abs{z}\) and the assumption \(\abs{a\br{x,y}}\in \bR{2^{3k-2},2^{3k+1}}\) in order to deduce
\begin{align}\label{I2}
    \abs{I_2}\lesssim &
    \int
        \abs{f}\br{x-t,y-t^2}
        \abs{a\br{x,y} t^2}
        \varphi\br{2^k t}
    dt\nonumber\\
    \eqsim &
    \int
        \abs{f}\br{x-t,y-t^2}
        2^k
        \br{2^kt}^2\varphi\br{2^k t}
    dt
    \lesssim
    M_\para f\br{x,y}.
\end{align}
From \eqref{I1} and \eqref{I2} we obtain the desired conclusion.
\end{proof}

\subsection{Treatment of the high frequency non-stationary component $C_{H,NS}$}\label{HNS}

In this section we are proving the following

\begin{lemma}\label{lem_ptwise_improv} For \(\vec{j}\in J_{ns}\) we have the following pointwise estimate
    \begin{equation}\label{pns}
        \vert C_{\vec{j}}f\vert \underset{N}{\lesssim} 2^{-N\vert{j}\vert_\infty}M_\str f.
    \end{equation}
Consequently, we deduce that
\begin{equation}\label{ns}
    \nrm{C_{H, NS} f}_{L^p}
    \underset{p}{\lesssim}
    \nrm{f}_{L^p},\quad \forall p\in\br{1,\infty}.
\end{equation}
    \end{lemma}

\begin{proof}

We claim\footnote{Without loss of generality we may assume that $\vert{j}\vert_\infty\geq 100$ as otherwise there is nothing to prove.} that a simple inspection of \eqref{multform} and \eqref{mult} reveals that in the high-frequency non-stationary case, that is when \(\vec{j}\in J_{ns}\), we gain an extra-decaying factor via the application of the non-stationary phase principle. Indeed, we first notice that for a fixed \(\vec{j}\in J_{ns}\) and \(\br{x,y}\in \R^2\) there is at most one \(k\in\Z\) such that \(\abs{a\br{x,y}}\in \bR{2^{3k+j_3},2^{3k+j_3+1}}\) and thus
\begin{equation*}
    C_{\vec{j}}f\br{x,y}=
    \int
        m_{\vec{j}}\br{\frac{\xi}{2^k},\frac{\eta}{2^{2k}},\frac{a\br{x,y}}{2^{3k}}}\widehat{f}\br{\xi,\eta}
        e\br{x\xi+y\eta}
    d\xi d\eta.
\end{equation*}
Consequently, the non-stationary phase principle implies:
\begin{equation*}
    \abs{
        \frac{\partial^\alpha}{\partial \xi^\alpha}
        \frac{\partial^\beta}{\partial \eta^\beta}
        m_{\vec{j}}\br{\xi,\eta,a}
    }\underset{\alpha,\beta,N}{\lesssim} 2^{-N\vert \vec{j}\vert_\infty},\quad \forall \alpha,\beta,N\in \N\,,
\end{equation*}
and thus
\begin{equation*}
    \nrm{
        \ang{x}^{2N}
        \ang{y}^{2N}
        \br{m_{\vec{j}}\br{\cdot,\cdot,a}}^{\vee}\br{x,y}
    }_{L^\infty\br{dx dy}}
    \lesssim
    \nrm{
        \br{
            \id-\frac{\partial^2}{\partial \xi^2}
        }^N
        \br{
            \id-\frac{\partial^2}{\partial \eta^2}
        }^N
        m_{\vec{j}}\br{\xi,\eta,a}
    }_{L^1\br{d\xi d\eta}}
    \underset{N}{\lesssim} 2^{-N\vert\vec{j}\vert_\infty}.
\end{equation*}
Conclude now that
\begin{align*}
    \vert
    C_{\vec{j}} f\br{x,y}
    \vert =
    &
    \abs{
        \int
            m_{\vec{j}}\br{\frac{\xi}{2^k},\frac{\eta}{2^{2k}},\frac{a\br{x,y}}{2^{3k}}}\widehat{f}\br{\xi,\eta}
            e\br{x\xi+y\eta}
        d\xi d\eta
    }\nonumber\\
    \lesssim &
    2^{-N\vert \vec{j}\vert_\infty}
    \br{
        \Dil^1_{2^{-k}}\ang{\cdot}^{-2N}\otimes \Dil^1_{2^{-2k}}\ang{\cdot}^{-2N}
    }\ast \abs{f}\br{x,y}\nonumber\\
    \lesssim & 2^{-N\vert \vec{j}\vert_\infty}
    M_\str f\br{x,y}\,.
\end{align*}
\end{proof}

\subsection{Treatment of the high frequency stationary component $C_{H,S}$ assuming Theorem \ref{thm_main_loc}}\label{HS1}

The goal of this section is to prove the following

\begin{lemma} \label{lem_w22_jest}
For \(\vec{j}\in J_{s}\) the following hold:
\begin{itemize}
\item with the previous notations--see Section \ref{Not}, we have the pointwise estimate

 \begin{equation}\label{peh}
 \vert C_{\vec{j}} f\vert\lesssim M_\para M_\str f.
 \end{equation}

\item assuming that the single scale estimate in Theorem \ref{thm_main_loc} holds, we have
\begin{equation}\label{extl2}
\Vert C_{\vec{j}} f \Vert_{L^{2,\infty}} \lesssim 2^{-\sigma\vert\vec{j}\vert_\infty}\nrm{f}_{L^2}\,,
\end{equation}
where here \(\sigma\) is the same as in the statement of Theorem \ref{thm_main_loc}.
\end{itemize}
Deduce via real interpolation that for every \(p\in\br{1,\infty}\) there exists a constant \(\sigma_p>0\) such that:
\begin{equation}\label{s}
\Vert C_{\vec{j}} f \Vert_{L^p}\underset{p}{\lesssim}2^{-\sigma_p\vert\vec{j}\vert_\infty}\nrm{f}_{L^p} \quad\textrm{and hence}\quad \nrm{C_{H, S} f}_{L^p}
    \underset{p}{\lesssim}
    \nrm{f}_{L^p}\,.
\end{equation}
\end{lemma}

\begin{proof}

We start by discussing the pointwise estimate \eqref{peh}: for a fixed \(\br{x,y}\in \R^2\) we notice that there is at most one \(k\in\Z\) such that \(\abs{a\br{x,y}}\in \bR{2^{3k+j_3},2^{3k+j_3+1}}\) and hence
\begin{equation*}
    C_{\vec{j}}f\br{x,y}=
    \int
        \br{
            P_{k,j_1}\otimes
            P_{2k,j_2}
        }f\br{x-t,y-t^2}
        e\br{a\br{x,y}t^3}
        \psi\br{2^k t}
    \frac{dt}{t}.
\end{equation*}
Since \(P_{k,j_1}\otimes P_{2k,j_2}\) is a convolution operator with
$\br{
        P_{k,j_1}\otimes P_{2k,j_2}
    }
    f:=
    \br{
        \Dil^1_{2^{-k-j_1}}\widecheck{\phi_1}\otimes
        \Dil^1_{2^{-2k-j_2}}\widecheck{\phi_2}
    }\ast f$
and
\(
    \br{\phi_1,\phi_2}\in
    \BR{
        \br{\varphi,\psi},
        \br{\psi,\varphi},
        \br{\psi,\psi}
    }
\),
we have the following trivial bound:
\begin{equation*}
    \abs{
        \br{
            P_{k,j_1}\otimes
            P_{2k,j_2}
        }f
    }
    \lesssim
    M_\str f.
\end{equation*}
Consequently, estimate \eqref{peh} follows from
\begin{align*}
    &
    \abs{
        \int
            \br{
                P_{k,j_1}\otimes
                P_{2k,j_2}
            }f\br{x-t,y-t^2}
            e\br{a\br{x,y}t^3}
            \psi\br{2^k t}
        \frac{dt}{t}
    }\nonumber\\
    \lesssim &
    \int
        M_\str f\br{x-t,y-t^2}
        2^k\1_{\br{-2,2}}\br{2^k t}
    dt
    \lesssim
    M_\para M_\str f\br{x,y}\,.
\end{align*}

We shift now to \eqref{extl2}: as before, we fix \(\vec{j}\in J_{s}\) and consider the auxiliary measurable function
$$a_k:=\frac{\br{\Dil^\infty_{2^k}\otimes \Dil^\infty_{2^{2k}}}a}{2^{3k}}\,.$$
We now rewrite the expression:
\begin{equation*}
    C_{\vec{j}}f=
    \sum_{k\in\Z}\br{\Dil^2_{2^{-k}}\otimes \Dil^2_{2^{-2k}}}\Br{
            \1_{
                \abs{a_k}^{-1}
                \bR{2^{j_3},2^{j_3+1}}
            }
            \cdot
            \mathcal{C}^{\br{a_k}}
            \br{P_{0,j_1}\otimes P_{0,j_2}}
            \br{
                \Dil^2_{2^k}\otimes
                \Dil^2_{2^{2k}}
            }f
        }.
\end{equation*}
Take now a measurable set \(E\subset \R^2\) and dualize the above expression in order to obtain
\begin{equation*}
    \ang{
        C_{\vec{j}}f,\1_E
    }=\sum_{k\in\Z}
    \big\langle
        \mathcal{C}^{\br{a_k}}
        \br{P_{0,j_1}\otimes P_{0,j_2}}
            \br{
                \Dil^2_{2^k}\otimes
                \Dil^2_{2^{2k}}
            }f
    ,
        \1_{
                \abs{a_k}^{-1}
                \bR{2^{j_3},2^{j_3+1}}
            }
        \br{
            \Dil^2_{2^k}\otimes
            \Dil^2_{2^{2k}}
        }\1_E
    \big\rangle.
\end{equation*}
Applying now Theorem \ref{thm_main_loc} we deduce
\begin{equation*}
    \abs{
        \ang{
            C_{\vec{j}}f,\1_E
        }
    }
    \lesssim
    2^{-|\vec{j}|_\infty\sigma}
    \sum_{k\in\Z}
    \nrm{
        \br{P_{0,j_1}\otimes P_{0,j_2}}
        \br{
            \Dil^2_{2^k}\otimes
            \Dil^2_{2^{2k}}
        }f
    }_{L^2}
    \cdot
    \nrm{
        \1_{
                \abs{a_k}^{-1}
                \bR{2^{j_3},2^{j_3+1}}
            }
        \br{
            \Dil^2_{2^k}\otimes
            \Dil^2_{2^{2k}}
        }\1_E
    }_{L^2}.
\end{equation*}
It remains to notice that
\begin{equation}\nonumber
    \nrm{
        \br{P_{0,j_1}\otimes P_{0,j_2}}
        \br{
            \Dil^2_{2^k}\otimes
            \Dil^2_{2^{2k}}
        }f
    }_{L^2}
    =
    \nrm{
        \br{P_{k,j_1}\otimes P_{2k,j_2}} f
    }_{L^2}\,,
\end{equation}
and, similarly, that
\begin{equation}\nonumber
    \nrm{
        \1_{
                \abs{a_k}^{-1}
                \bR{2^{j_3},2^{j_3+1}}
            }
        \br{
            \Dil^2_{2^k}\otimes
            \Dil^2_{2^{2k}}
        }\1_E
    }_{L^2}
    =
    \abs{
        \abs{a}^{-1}\bR{2^{3k+j_3},2^{3k+j_3+1}}
        \cap E
    }^{\frac{1}{2}}.
\end{equation}
Now the condition \(\vec{j}\in J_{s}\) implies in particular that \(\br{j_1,j_2}\neq \br{0,0}\) which allows us to exploit an almost-orthogonality feature. Putting all these together and applying Cauchy-Schwarz, we conclude
\begin{equation*}
    \abs{
        \ang{
            C_{\vec{j}}f,\1_E
        }
    }
    \lesssim
    2^{-\sigma\vert\vec{j}\vert_\infty}
    \nrm{f}_{L^2}\abs{E}^{\frac{1}{2}}\,.
\end{equation*}

\end{proof}

\section{The proof of Theorem \ref{thm_main_loc} via the LGC-method}\label{mainsect}

In this section we prove Theorem \ref{thm_main_loc} by relying on the (Rank I) LGC-methodology introduced in \cite{LVunif} and further developed in \cite{LVBilHilbCarl}, \cite{HL(CTHT)}, \cite{GALVHybcurves}; as explained in these papers, this method consists of three main steps:
\begin{itemize}
\item \textsf{Step I.} Partition of the time-frequency space in regions that enforce a \textsf{linear} behavior of the argument of the input functions and of the oscillatory phase of the kernel/multiplier;

\item \textsf{Step II.} \textsf{Wave-packet/Gabor} frame discretization of the input functions at the scale that is consistent with the partition implemented at Step I;

\item \textsf{Step III.} Analysis of the resulting \textsf{time-frequency correlations} via an array of tools that include (i) a \emph{sparse-uniform dichotomy} addressing the distribution of the information encapsulated in the wave-packet decomposition carried at Step II and (ii) a \emph{level set analysis} addressing the maximally sparse situation and captured within a so-called time-frequency correlation set.
\end{itemize}

According to the above display we will subdivide our presentation in three components: in the first subsection we complete steps I and II above, while in the next two subsections we address Step III as follows: in the second subsection we implement a suitable sparse-uniform dichotomy analysis while in the third subsection we perform the level set analysis adapted to the time-frequency correlation set.

We close this introductory paragraph with the following comment: throughout the rest of the section, without loss of generality we will assume that \(\lambda>100\), \(\abs{E}<\infty\), and that the four-tuple \((\lambda,f,a,E)\) is fixed and satisfies assumption \eqref{eq_f_a_E_cond} in Theorem \ref{thm_main_loc}.

\subsection{The linearizing wave-packet decomposition of the expression $\ang{\mathcal{C}^{\br{a}}f,\1_E}$}\label{linsec}

$\:\:\quad$As mentioned above, in this subsection we complete Steps I and II of the LGC-method obtaining thus the control of our main expression $\ang{\mathcal{C}^{\br{a}}f,\1_E}$ by a wave-packet discretized model whose linearizing effect transforms the curvature feature of $\mathcal{C}^{\br{a}}$ into suitable time-frequency correlations--the latter can be visualized as a collection of dependencies (ultimately, a non-linear system of equations) between the spatial and frequency parameters involved in the discretized model.

Our presentation here will be further divided in two subsections: the first one encompasses the choice of the linearizing scale in our Gabor analysis and some basic properties of a wave-packet decomposition while the second one addresses the specifics of the adapted linearized wave-packet discretization of $\ang{\mathcal{C}^{\br{a}}f,\1_E}$.

\subsubsection{Preliminaries}

We first start by briefly touching the choice of the \emph{linearizing scale}: dictated by the hypothesis \eqref{eq_f_a_E_cond} on the size of the linearizing function $a$ (for the moment being one may think that $|a|\approx\l$) a simple inspection of \eqref{mult} gives, via a Taylor series argument, that the maximal physical scale for which the phase of the multiplier behaves linearly in $t$  is $(\sqrt{\l})^{-1}$. As a consequence, our subsequent Gabor analysis will involve wave-packets adapted to this
linearizing scale.

Next, we recall several basic facts about wave-packets and related concepts:
\begin{itemize}
\item a (frequency compactly supported) function \(\Phi\in L^2\br{\R^2}\) is time-frequency localized at unit scale if
    \begin{equation*}
        \abs{\Phi}\leq \ang{\cdot}^{-10N}\otimes \ang{\cdot}^{-10N},\quad
        |\widehat{\Phi}|\leq \1_{\Br{-1,1}^2}\quad \textrm{with}\quad N\in\N,\,N\geq 10\,;
    \end{equation*}
\item the wave-packets of interest in our present paper are constructed as follows: if \(\Phi\in L^2\br{\R^2}\) is as above, \(p,q\in\R\) and \(u,v\in\Z\) then, recalling \eqref{symm}, we set
    \begin{align*}
        \Phi_{p,q,u,v}:= &
        \br{
            \Dil^2_{1/\sqrt{\lambda}}\otimes
            \Dil^2_{1/\sqrt{\lambda}}
        }
        \br{
            \Tr_{\lfloor p\rfloor} \otimes
            \Tr_{\lfloor q\rfloor}
        }
        \br{
            \Mod_{u}\otimes
            \Mod_{v}
        }
        \Phi
        \,;
    \end{align*}

\item in order to better capture the local properties of the function \(f\in L^2\br{\R^2}\) encapsulated in the behavior of its Gabor coefficients
\(
    \BR{
        \ang{
            f,\Phi_{p,q,u,v}
        }
    }_{p,q,u,v\in\Z}
\) we introduce the following auxiliary measure: if \(S\subset \R^2\) is a Lebesgue measurable set and \(p,q\in\R\), we define
    \begin{equation*}
        \mu_{p,q}\br{S}:=
        \int_S
            \ang{\sqrt{\lambda}x-\lfloor p\rfloor}^{-N}
            \ang{\sqrt{\lambda}y-\lfloor q\rfloor}^{-N}
        dxdy\,,
    \end{equation*}
    and notice that
     \begin{equation}\label{meascontr}
     \mu_{p,q}\br{S}\leq \mu_{p,q}\br{\R^2}\approx\frac{1}{\l}\,.
     \end{equation}
    Also, remark that at the heuristic level \(\mu_{p,q}\br{S}:=\abs{S\cap I_{p,q}}\), with \(I_{p,q}:=\bR{\frac{\lfloor p\rfloor}{\sqrt{\lambda}},\frac{\lfloor p+1 \rfloor}{\sqrt{\lambda}}}\times\bR{\frac{\lfloor q\rfloor}{\sqrt{\lambda}},\frac{\lfloor q+1\rfloor}{\sqrt{\lambda}}}\).
    \end{itemize}

With these being said, we now state the following result whose easy proof is left to the reader:

\begin{lemma}\label{prop_Gabor_bds}
    For \(f\in L^2\br{\R^2}\) and a time-frequency localized \(\Phi\in L^2\br{\R^2}\), we have the following inequalities:
    \begin{align*}
        \abs{
            \ang{
                f,\Phi_{p,q,u,v}
            }
        }
        \leq
        \sqrt{\lambda}
        \nrm{f}_{L^1\br{d\mu_{p,q}}},\quad
        \br{
            \sum_{u,v}
                \abs{
                    \ang{
                        f,\Phi_{p,q,u,v}
                    }
                }^2
        }^{\frac{1}{2}}
        \lesssim
        \nrm{
            f
        }_{L^2\br{d\mu_{p,q}}}.
    \end{align*}
As a direct consequence,
\begin{equation*}
    \br{
        \sum_{
            p,q,u,v
        }
            \abs{
                \ang{
                    f,\Phi_{p,q,u,v}
                }
            }^2
    }^{\frac{1}{2}}
    \lesssim
    \nrm{f}_{L^2}
    .
\end{equation*}
\end{lemma}

\subsubsection{The wave-packet discretized model}\label{subsec_wp}
Our goal in this subsection is to prove the following

\begin{lemma}
    There are time-frequency localized \(\Phi,\Psi\in L^2\br{\R^2}\)
    and a function \(\chi\) on \(\R\) with \(\abs{\chi}\leq \1_{\Br{-1,1}}\) such that:
    \begin{equation}\label{eq_wp_rep}
        \abs{
            \ang{\mathcal{C}^{\br{a}}f,\1_E}
        }
        \lesssim
        \hspace{-1em}
        \sum_{\substack{p,q,r\\
        u,v,w:\\
        \frac{\sqrt{\lambda}}{3}\leq r\leq 3\sqrt{\lambda}
        }}\hspace{-1em}
        \frac{
            \abs{
                \ang{f,\Phi_{p-r,q-\frac{r^2}{\sqrt{\lambda}},u,v}}
                \ang{
                    \overline{e\br{\frac{ar^3}{\lambda^{\frac{3}{2}}}}}\chi\br{\frac{a}{\sqrt{\lambda}}-w}\1_E,
                    \!\Psi_{p,q,u,v}
                }
            }
        }{
            \sqrt{\lambda}
            \ang{
                3w\frac{r^2}{\lambda}
                -2v\frac{r}{\sqrt{\lambda}}
                -u
            }^N
        }.
    \end{equation}
\end{lemma}

\begin{proof}
  For convenience, we start by restating \eqref{multform}--\eqref{mult} in our specific setting:
    \begin{equation*}
        \ang{
            \mathcal{C}^{\br{a}}f,
            \1_E
        }=
        \int_E
            \widehat{f}\br{\xi,\eta}
                m\br{\xi,\eta,a\br{x,y}}
                e\br{x\xi+y\eta}
        d\xi d\eta
        dx dy,
    \end{equation*}
    where the symbol is defined by the following oscillatory integral:
    \begin{equation*}
        m\br{\xi,\eta,a}:=
        \int
            e\br{at^3-\eta t^2-\xi t}
            \rho\br{t}
        dt.
    \end{equation*}
    Next, we choose a smooth even bump function \(0\leq \phi \leq \1_{\Br{-1,1}}\) and generate the partition of unity
    \begin{equation}\label{eq_part_uni}
        1=\sum_{r,u,v,w}\Tr_r\phi^2\br{\sqrt{\lambda}t}\Tr_u\phi^2\br{\frac{\xi}{\sqrt{\lambda}}}\Tr_v\phi^2\br{\frac{\eta}{\sqrt{\lambda}}}\Tr_w\phi^2\br{\frac{a}{\sqrt{\lambda}}}
    \end{equation}
     that we attach it to the symbol. Via a change of variable \(\tau:=\sqrt{\lambda}t-r\) and some algebraic manipulations, we obtain the following expression:
    \begin{align*}
        m\br{\xi,\eta,a}
        =&
        \frac{1}{\sqrt{\lambda}}
        \sum_{r,u,v,w}
        e\br{
            v\BR{-\frac{r^2}{\sqrt{\lambda}}}+
            a\frac{r^3}{\lambda^{\frac{3}{2}}}
            +\frac{\eta}{\sqrt{\lambda}}\left\lfloor -\frac{r^2}{\sqrt{\lambda}}\right\rfloor
            -\xi\frac{r}{\sqrt{\lambda}}
        }\nonumber\\
            \cdot &
            \int
            e\br{
                \br{
                    3w\frac{r^2}{\lambda}
                    -2v\frac{r}{\sqrt{\lambda}}
                    -u
                }
                \tau
            }
            \mathcal{E}\br{
                \tau,\xi,\eta,a,r,u,v,w
            }
            \rho\br{\frac{r+\tau}{\sqrt{\lambda}}}
            \phi^2\br{\tau}
        d\tau\nonumber\\
        \cdot &
        \phi^2\br{\frac{\xi}{\sqrt{\lambda}}-u}
        \phi^2\br{\frac{\eta}{\sqrt{\lambda}}-v}
        \phi^2\br{\frac{a}{\sqrt{\lambda}}-w},
    \end{align*}
    where \(\mathcal{E}(\cdot)\), obtained via a Taylor decomposition, stands for the expression:
    \begin{align*}
        \mathcal{E}
        \br{\cdot}
        = &
        e\br{
            -
            \br{
                \frac{\xi}{\sqrt{\lambda}}-u
            }\tau
        }
        e\br{
            -2\br{
                \frac{\eta}{\sqrt{\lambda}}-v
            }\frac{r}{\sqrt{\lambda}}\cdot\tau
        }
        e\br{
            3\br{
                \frac{a}{\sqrt{\lambda}}-w
            }\frac{r^2}{\lambda}\cdot\tau
        }\nonumber\\
        \cdot &
        e\br{
            -\lambda^{-\frac{1}{2}}
                \br{
                    \frac{\eta}{\sqrt{\lambda}}-v
                }
            \cdot\tau^2
        }
        e\br{
            -\frac{
                v
            }{
                \sqrt{\lambda}
            }
            \cdot\tau^2
        }
        e\br{
            3\lambda^{-\frac{1}{2}}\br{\frac{a}{\sqrt{\lambda}}-w}\cdot\frac{r}{\sqrt{\lambda}}\cdot\tau^2
        }
        e\br{
            3\cdot\frac{w}{\sqrt{\lambda}}\cdot\frac{r}{\sqrt{\lambda}}\cdot\tau^2
        }\nonumber\\
        \cdot &
        e\br{
            \lambda^{-1}\cdot
            \br{
                \frac{a}{\sqrt{\lambda}}-w
            }\cdot\tau^3
        }
        e\br{
            \lambda^{-\frac{1}{2}}\cdot
            \frac{w}{\sqrt{\lambda}}\cdot\tau^3
        }
        e\br{
            \br{\frac{\eta}{\sqrt{\lambda}}-v}
            \BR{-\frac{r^2}{\sqrt{\lambda}}}
        }.\nonumber
    \end{align*}
    On the one hand, by design:
    \begin{equation*}
        \abs{\tau},\;
        \abs{
            \frac{\xi}{\sqrt{\lambda}}-u
        },\;
        \abs{
            \frac{\eta}{\sqrt{\lambda}}-v
        },\;
        \abs{
            \frac{a}{\sqrt{\lambda}}-w
        },\;
        \abs{
            \BR{-\frac{r^2}{\sqrt{\lambda}}}
        }\leq 1.\nonumber
    \end{equation*}
    On the other hand, since \(\abs{\xi},\abs{\eta},\abs{a}\leq 2\lambda\) and \(\frac{1}{2}\leq \frac{r+\tau}{\sqrt{\lambda}}\leq 2\), we further deduce:
    \begin{equation*}
        \abs{
            \frac{u}{\sqrt{\lambda}}
        },\;
        \abs{
            \frac{v}{\sqrt{\lambda}}
        },\;
        \abs{
            \frac{w}{\sqrt{\lambda}}
        }
        \leq 2+\frac{1}{\sqrt{\lambda}}\leq 3,\quad
        \frac{1}{3}
        \leq
        \frac{1}{2}-\frac{1}{\sqrt{\lambda}}
        \leq
        \frac{r}{\sqrt{\lambda}}
        \leq 2+\frac{1}{\sqrt{\lambda}}
        \leq 3.\nonumber
    \end{equation*}
    Using Taylor expansions for all the ten exponential terms appearing in $\mathcal{E}
        \br{\cdot}$ and making the notations $\vec{k}=(k_1,\ldots, k_{10})$ and \(\vec{k}!:=\prod_{j=1}^{10}\br{k_j!}\) we deduce
    \begin{align*}
        \mathcal{E}
        \br{\cdot}
        &=
        \sum_{\vec{k}}
            \frac{
                \br{-1}^{k_1+k_2+k_4+k_5}
                2^{k_2}
                3^{k_3+k_6+k_7}
                \lambda^{-\frac{k_4+k_6+2k_8+k_9}{2}}
                \br{2\pi i}^{\sum_{j=1}^{10}k_j}
            }{\vec{k}!}
            \BR{-\frac{r^2}{\sqrt{\lambda}}}^{k_{10}}
            \br{
                \frac{r}{\sqrt{\lambda}}
            }^{
                k_2+2k_3+k_6+k_7
            }
            \br{
                \frac{v}{\sqrt{\lambda}}
            }^{
                k_5
            } \nonumber\\
        \cdot &
            \br{
                \frac{w}{\sqrt{\lambda}}
            }^{
                k_7+k_9
            }
            \br{
            \frac{\xi}{\sqrt{\lambda}}-u}^{k_1}
        \br{
            \frac{\eta}{\sqrt{\lambda}}-v
        }^{k_2+k_4+k_{10}}
        \br{
            \frac{a}{\sqrt{\lambda}}-w
        }^{k_3+k_6+k_8}
        \tau^{k_1+k_2+k_3+2k_4+2k_5+2k_6+2k_7+3k_8+3k_9}.
    \end{align*}
  With these, setting
    \begin{align*}
        &C_{\vec{k}}\br{\lambda}:=
        \br{\vec{k}!}^{-\frac{1}{4}}
        \br{-1}^{k_1+k_2+k_4+k_5}
        2^{k_2}
        3^{k_3+k_6+k_7}
        \lambda^{-\frac{k_4+k_6+2k_8+k_9}{2}}\br{2\pi i}^{\sum_{j=1}^{10}k_j}\,,\nonumber\\
        &D_{\vec{k}}\br{r,v,w}:=
        \BR{-\frac{r^2}{\sqrt{\lambda}}}^{k_{10}}
        \br{\vec{k}!}^{-\frac{1}{4}}
        \br{
            \frac{r}{\sqrt{\lambda}}
        }^{
            k_2+2k_3+k_6+k_7
        }
        \br{
            \frac{v}{\sqrt{\lambda}}
        }^{
            k_5
        }
        \br{
            \frac{w}{\sqrt{\lambda}}
        }^{
            k_7+k_9}\,,
        \nonumber\\
        &\widehat{\Psi}^{\br{\vec{k}}}\br{z_1,z_2}:=
        \frac{z_1^{k_1}z_2^{k_2+k_4+k_{10}}}{\br{\vec{k}!}^\frac{1}{4}}\phi\br{z_1}\phi\br{z_2},\quad
        \chi^{\br{\vec{k}}}\br{z}:=z^{k_3+k_6+k_8}\phi^2\br{z},
        \nonumber\\
        &\phi^{\br{\vec{k}}}\br{z}:=\frac{z^{k_1+k_2+k_3+2k_4+2k_5+2k_6+2k_7+3k_8+3k_9}}{\br{\vec{k}!}^{\frac{1}{4}}}\phi^2\br{z}\,,\nonumber
    \end{align*}
    we deduce

    \begin{align*}
        \ang{
            \mathcal{C}^{\br{a}}f,\1_E
        }
        = &
        \sum_{{r,u,v,w}\atop{\vec{k}}}
        \frac{C_{\vec{k}}\br{\lambda}}{\sqrt{\lambda}}
        D_{\vec{k}}\br{r,v,w}
        e\br{v\BR{-\frac{r^2}{\sqrt{\lambda}}}}
        \int
            e\br{
                \br{
                    3w\frac{r^2}{\lambda}
                    -2v\frac{r}{\sqrt{\lambda}}
                    -u
                }
                \tau
            }
            \rho\br{\frac{r+\tau}{\sqrt{\lambda}}}
            \phi^{\br{\vec{k}}}\br{\tau}
        d\tau\nonumber\\
        \cdot &
        \int
            e\br{\frac{\eta}{\sqrt{\lambda}}
            \cdot
            \left\lfloor
                -\frac{r^2}{\sqrt{\lambda}}
            \right\rfloor
            -\frac{\xi}{\sqrt{\lambda}}\cdot r
            }
            \widehat{f}\br{\xi,\eta}
            \phi\br{\frac{\xi}{\sqrt{\lambda}}-u}
            \phi\br{\frac{\eta}{\sqrt{\lambda}}-v}\nonumber
            \\
            \cdot &
            \br{
                e\br{
                    \frac{ar^3}{\lambda^{\frac{3}{2}}}
                }
                \chi^{\br{\vec{k}}}\br{\frac{a}{\sqrt{\lambda}}-w}\1_E
            }^{\vee}
            \br{\xi,\eta}
            \widehat{\Psi}^{\br{\vec{k}}}\br{
                \frac{\xi}{\sqrt{\lambda}}-u,
                \frac{\eta}{\sqrt{\lambda}}-v
            }
        d\xi d\eta.\nonumber
    \end{align*}
    For the last integral on the compact interval
    $\bR{
            \sqrt{\lambda}\br{u-1},\sqrt{\lambda}\br{u+1}
        }\times
        \bR{
            \sqrt{\lambda}\br{v-1},\sqrt{\lambda}\br{v+1}
        }$
    we apply Parseval in order to relate the integral to the Fourier coefficients of the functions therein

    \begin{align}\label{discmod}
        \ang{
            \mathcal{C}^{\br{a}}f,\1_E
        }
        =
        \sum_{{r,u,v,w}\atop{\vec{k}}}
        \frac{C_{\vec{k}}\br{\lambda}}{\sqrt{\lambda}}
        D_{\vec{k}}\br{r,v,w}
        e\br{v\BR{-\frac{r^2}{\sqrt{\lambda}}}}
        \int
            e\br{
                \br{
                    3w\frac{r^2}{\lambda}
                    -2v\frac{r}{\sqrt{\lambda}}
                    -u
                }
                \tau
            }
            \rho\br{\frac{r+\tau}{\sqrt{\lambda}}}
            \phi^{\br{\vec{k}}}\br{\tau}
        d\tau\nonumber\\
        \cdot
        \sum_{\substack{
            p,q\\
            \sigma\in\BR{0,1}^2
        }}
        \ang{
            f,
            \br{\Tr_{\frac{\sigma_1}{2}}\widecheck{\phi}\otimes\Tr_{\frac{\sigma_2}{2}}\widecheck{\phi}}_{p-r,q-\frac{r^2}{\sqrt{\lambda}},u,v}
        }\cdot
        \overline{
            \ang{
                e\br{
                    -\frac{ar^3}{\lambda^{\frac{3}{2}}}
                }
                \chi^{\br{\vec{k}}}\br{\frac{a}{\sqrt{\lambda}}-w}\1_E,
                \br{\Tr_{\frac{\sigma_1}{2}}\otimes\Tr_{\frac{\sigma_2}{2}}\Psi^{\br{\vec{k}}}}_{p,q,u,v}
            }
        }.
    \end{align}
    Now repeated integration by parts gives the decay estimate
    \begin{equation}\label{decy}
        \abs{
            \int
                e\br{
                    \br{
                        3w\frac{r^2}{\lambda}
                        -2v\frac{r}{\sqrt{\lambda}}
                        -u
                    }
                    \tau
                }
                \rho\br{\frac{r+\tau}{\sqrt{\lambda}}}
                \phi^{\br{\vec{k}}}\br{\tau}
            d\tau
        }
        \lesssim
        \ang{
            3w\frac{r^2}{\lambda}
            -2v\frac{r}{\sqrt{\lambda}}
            -u
        }^{-N}
        \Vert\phi^{\br{\vec{k}}}\Vert_{C^N}\,,
    \end{equation}
     where \(\Vert\phi^{\br{\vec{k}}}\Vert_{C^N}\lesssim 1\) with the implicit constant only depending on $N$ and the function \(\rho\).

  Finally, using the $\vec{k}$-summability of the coefficients $C_{\vec{k}}$ and $D_{\vec{k}}$ in \eqref{discmod} together with the decaying estimate \eqref{decy} we conclude that \eqref{eq_wp_rep} holds.
\end{proof}

With these, via the above lemma, we reduced our problem to estimating the model form:
\begin{equation}\label{linwavemodel}
    \Lambda\br{F,G}:=
        \mathop{\mathbb{E}}_{
        \frac{\sqrt{\lambda}}{3}
        \leq r\leq 3
        \sqrt{\lambda}
        }
    \sum_{\substack{p,q\\
        u,v,w
        }}
        \frac{
            F_{p-r,q-\frac{r^2}{\sqrt{\lambda}},u,v}
            G_{p,q,r,u,v,w}
        }{
            \ang{
                3w\frac{r^2}{\lambda}
                -2v\frac{r}{\sqrt{\lambda}}
                -u
            }^N
        }
\end{equation}
where here \(F,G\) represent the collection of Gabor coefficients:
\begin{equation}\label{Gabcoef}
    F_{p,q,u,v}:=\abs{
        \ang{f,\Phi_{p,q,u,v}}
    },\quad
    G_{p,q,r,u,v,w}:=
    \abs{
        \ang{
            \overline{e\br{\frac{ar^3}{\lambda^{\frac{3}{2}}}}}\chi\br{\frac{a}{\sqrt{\lambda}}-w}\1_E,
            \Psi_{p,q,u,v}
        }
    }.
\end{equation}
Specifically, our Theorem \ref{thm_main_loc} reduces now to proving the following estimate:
\begin{equation}\label{key}
    \Lambda\br{F,G}\lesssim\lambda^{-\sigma}\nrm{f}_{L^2}\abs{E}^{\frac{1}{2}}.
\end{equation}

\subsection{A sparse--uniform dichotomy}\label{SpUnifDich}

In this section we analyze the main expression \eqref{linwavemodel} depending on the information carried by the Gabor coefficients defined in \eqref{Gabcoef}. We will split our section in four subsections:
\begin{itemize}
\item a preliminary subsection in which we introduce the main definitions and perform the sparse-uniform decomposition
$$\Lambda=: \Lambda_{SU}\,+\, \Lambda_{UU}\,+\,\Lambda_{SS}\,;$$
\item  in the second subsection we treat of the sparse-uniform component $\Lambda_{SU}$;
\item  in the third subsection we deal with the uniform-uniform component $\Lambda_{UU}$;
\item  in the last subsection  we handle the sparse-sparse component $\Lambda_{SS}$ assuming that the level set estimate--treated in Section \ref{LevSet}--holds.
\end{itemize}

\subsubsection{Preliminaries: the main decomposition}

Our reasonings in this subsection will focus on the information encapsulated within the Gabor coefficients in \eqref{Gabcoef}. We start with the following

\begin{remark} The lack of control of the smoothness degree for the function $a$ implies a rough oscillatory behavior of the term \(\overline{e\br{\frac{ar^3}{\lambda^{\frac{3}{2}}}}}\) in the definition of the Gabor coefficient $G_{p,q,r,u,v,w}$ in \eqref{Gabcoef}. As a consequence, one can not expect a good (pointwise) control in the $r-$parameter for the $G$ terms. Moreover, the $(u,v)$-frequency dependence of $G_{p,q,r,u,v,w}$ can only be exploited via the Parseval relation.
All of the above suggest an analysis of the terms $G_{p,q,r,u,v,w}$ that involve directly only the spatial parameters $p,q$ together with the frequency parameter $w$. To be more specific, for \(w\in\Z\), we define
    \begin{equation*}
        E_w:=\set{
            \br{x,y}\in E
        }{
            a\br{x,y}\in \sqrt{\lambda}[w-1,w+1]}.
    \end{equation*}

 Since by construction we have $\abs{\overline{e\br{\frac{ a r^3}{\lambda^{\frac{3}{2}}}}}\chi\br{\frac{a}{\sqrt{\lambda}}-w}\1_E}\leq \1_{E_w}$, Lemma \ref{prop_Gabor_bds} applied to \(G\) now reads
\begin{equation}\label{eq_G_bd_2_mu}
    G_{p,q,r,u,v,w}\leq \sqrt{\lambda}\,\mu_{p,q}\br{E_w},\quad
    \br{
        \sum_{u,v}
        G_{p,q,r,u,v,w}^2
    }^{\frac{1}{2}}
    \lesssim
    \mu_{p,q}^{\frac{1}{2}}\br{E_w}.
\end{equation}

\end{remark}

With these being said, we introduce now the following
\begin{definition}
    Let \(m,n\in \N\) and \(p,q\in\R\). We then introduce the following sets:
    \begin{align}%
        \F^{\br{m}}_{p,q}:= &
        \set{\br{u,v}\in\Z^2}{
            2^{m-1}
            \nrm{f}_{L^2\br{d\mu_{p,q}}}
            /\sqrt{\lambda}
            <
            F_{p,q,u,v}
            \leq
            2^m
            \nrm{f}_{L^2\br{d\mu_{p,q}}}
            /\sqrt{\lambda}
        }\label{eq_Fpq_set}\\
        \F^{\br{0}}_{p,q}:= &
        \set{\br{u,v}\in\Z^2}{
            F_{p,q,u,v}
            \leq
            \nrm{f}_{L^2\br{d\mu_{p,q}}}
            /\sqrt{\lambda}
        }
        \nonumber\\
        \mathcal{G}^{\br{n}}_{p,q}:= &
        \set{
            w\in\Z
        }{
            2^{n-1}
            \mu_{p,q}\br{E}/\sqrt{\lambda}<
            \mu_{p,q}\br{E_w}\leq 2^{n}
            \mu_{p,q}\br{E}/\sqrt{\lambda}
        }\label{eq_Gpq_set}\\
        \mathcal{G}^{\br{0}}_{p,q}:= &
        \set{
            w\in\Z
        }{
            \mu_{p,q}\br{E_w}\leq
            \mu_{p,q}\br{E}/\sqrt{\lambda}
        }
        .\nonumber
    \end{align}%
    Also, we truncate the Gabor coefficients accordingly:
    \begin{equation*}
        F^{\br{m}}_{p,q,u,v}:=
        \1_{\F^{\br{m}}_{p,q}}\br{u,v}F_{p,q,u,v},\quad
        G^{\br{n}}_{p,q,r,u,v,w}:=
        \1_{\mathcal{G}^{\br{n}}_{p,q}}\br{w}G_{p,q,r,u,v,w}.
    \end{equation*}
\end{definition}

From the definition above, we remark the following:

On the one hand, from \eqref{eq_f_a_E_cond}, we have that the natural range for our Gabor coefficients is
    \begin{equation*}
        F_{p,q,u,v}\neq 0\implies \abs{u},\abs{v}\leq 3\sqrt{\lambda},\quad
        G_{p,q,r,u,v,w}\neq 0\implies
        \abs{w}\leq 3\sqrt{\lambda}\,,
    \end{equation*}
which immediately implies the trivial bound
    \begin{equation*}
        \# \F^{\br{m}}_{p,q}\lesssim \lambda,\quad
        \# \mathcal{G}^{\br{n}}_{p,q}\lesssim
        \sqrt{\lambda}.
    \end{equation*}

On the other hand,  for \(m,n\in\N\), we can refine the above estimate through Chebyshev inequality:
    \begin{equation}\label{eq_cheby}
        \# \F^{\br{m}}_{p,q}\lesssim \frac{\lambda}{2^{2m}},\quad
        \# \mathcal{G}^{\br{n}}_{p,q}\lesssim
        \frac{\sqrt{\lambda}}{2^n}.
    \end{equation}
All of the above imply that the sets $\F^{\br{m}}_{p,q}, \mathcal{G}^{\br{n}}_{p,q}$ are nontrivial only when \(0\leq m,n\lesssim \log\lambda\).

With all these settled we are now ready for the key sparse-uniform decomposition of our form
\begin{equation*}
    \Lambda\br{F,G}=
    \sum_{\br{m,n}\in \J}
        \Lambda\br{F^{\br{m}},G^{\br{n}}}\,,
\end{equation*}
where here, for notational convenience, we set
\begin{equation*}
    \J:=\set{\br{m,n}\in \N^2}{m,n\lesssim \log\lambda}\,.
\end{equation*}

Indeed, the \emph{sparse-uniform decomposition} of $\Lambda$ is given by
\begin{equation*}
\Lambda=: \Lambda_{SU}\,+\, \Lambda_{UU}\,+\,\Lambda_{SS}\,,
\end{equation*}
where, for some small constant \(\epsilon>0\) (up to log losses one can take $\ep=\sigma$ with $\sigma$ appearing in \eqref{key}), we have
\begin{itemize}
\item the \emph{sparse-uniform component} $\Lambda_{SU}$ is given by
\begin{equation*}
    \Lambda_{SU}\br{F,G}=
    \sum_{\br{m,n}\in \J_{S,U}}
        \Lambda\br{F^{\br{m}},G^{\br{n}}}\,,
\end{equation*}
with $\J_{S,U}:=\set{\br{m,n}\in\J}{2^m\gtrsim 2^n\lambda^\epsilon}$;
\smallskip

\item the \emph{uniform-uniform component} $\Lambda_{UU}$ is given by
\begin{equation*}
    \Lambda_{UU}\br{F,G}=
    \sum_{\br{m,n}\in \J_{U,U}}
        \Lambda\br{F^{\br{m}},G^{\br{n}}}\,,
\end{equation*}
with $\J_{U,U}:=\set{\br{m,n}\in\J\setminus\J_{S,U}}{2^{2m}\lesssim
            2^n\lambda^{\frac{1}{2}-4\epsilon}}$;
\smallskip

\item the \emph{sparse-sparse component} $\Lambda_{SS}$ is given by
\begin{equation*}
    \Lambda_{SS}\br{F,G}=
    \sum_{\br{m,n}\in \J_{S,S}}
        \Lambda\br{F^{\br{m}},G^{\br{n}}}\,,
\end{equation*}
with $\J_{S,S}:=
    \set{\br{m,n}\in\J\setminus(\J_{S,U}\cup \J_{U,U})}{2^m,2^n \gtrsim \lambda^{\frac{1}{2}-6\epsilon}}$.
\end{itemize}

\subsubsection{Treatment of the sparse-uniform component $\Lambda_{SU}$} This is the easiest to treat component of $\Lambda$. At the heuristic level, the control on $\Lambda_{SU}$ is a consequence of the fact that the set of Gabor coefficients within $F$ are sparse (small cardinality) relative to the set of Gabor coefficients in $G$.

The concrete procedure follows the reasonings below: for a fixed \(\br{m,n}\in\J_{S,U}\), using \eqref{meascontr}, \eqref{eq_cheby} and the definitions given in \eqref{eq_Fpq_set} and \eqref{eq_Gpq_set}, we have
\begin{equation}\label{FGestimSU}
F^{\br{m}}_{p-r,q-\frac{r^2}{\sqrt{\lambda}},u,v}G^{\br{n}}_{p,q,r,u,v,w}
    \lesssim \1_{\F^{\br{m}}_{p-r,q-\frac{r^2}{\sqrt{\lambda}}}}\br{u,v}
    \cdot
    \frac{2^{m+n}}{\lambda}\nrm{f}_{L^2\br{d\mu_{p-r,q-\frac{r^2}{\sqrt{\lambda}}}}}\mu_{p,q}^{\frac{1}{2}}\br{E}.
\end{equation}
From \eqref{FGestimSU} we further deduce that the expression $\Lambda\br{F^{\br{m}},G^{\br{n}}}$ is bounded from above by
\begin{align*}
    &
    \mathop{\mathbb{E}}_{r\eqsim \sqrt{\lambda}}
    \sum_{p,q}  \left(\frac{2^{m+n}}{\lambda}
    \nrm{f}_{L^2\br{d\mu_{p-r,q-\frac{r^2}{\sqrt{\lambda}}}}}
    \mu_{p,q}^{\frac{1}{2}}\br{E}\right)
    \cdot \left(
    \sum_{\br{u,v}\in \F^{\br{m}}_{p-r,q-\frac{r^2}{\sqrt{\lambda}}}}
    \sum_{w}\ang{3w\frac{r^2}{\lambda}-2v\frac{r}{\sqrt{\lambda}}-v}^{-N}\right)\nonumber\\
    \lesssim &
    \mathop{\mathbb{E}}_{r\eqsim \sqrt{\lambda}}
    \sum_{p,q} 2^{n-m}
    \nrm{f}_{L^2\br{d\mu_{p-r,q-\frac{r^2}{\sqrt{\lambda}}}}}
    \mu_{p,q}^{\frac{1}{2}}\br{E}
\end{align*}
Finally, Cauchy-Schwarz inequality applied on \(p,q\) variables gives:
\begin{equation}\nonumber
    \Lambda\br{F^{\br{m}},G^{\br{n}}}
    \lesssim
    2^{n-m}\nrm{f}_{L^2}\abs{E}^{\frac{1}{2}}\lesssim \lambda^{-\epsilon}\nrm{f}_{L^2}\abs{E}^{\frac{1}{2}}\,,
\end{equation}
and hence
\begin{equation}\label{SU}
    \Lambda_{SU}\br{F,G}\lesssim \lambda^{-\epsilon}\log^2\lambda\nrm{f}_{L^2}\abs{E}^{\frac{1}{2}}\,.
\end{equation}

\subsubsection{Treatment of the uniform-uniform component $\Lambda_{UU}$}

Due to the uniform nature of Gabor coefficients, in this new context one has to rely in a key manner on the time-frequency correlation weight factor \eqref{weightcurv}. The following simple but very useful lemma---whose proof is left to the reader---will be used extensively to extract information from the correlation factor.
\begin{lemma}\label{lem_jap_mul_diff}
For any \(x,y\in\R^D\) and \(a,b\in\R\) with $(a,b)\not=(0,0)$, we have the following estimate:
    \begin{equation*}
        \ang{x}^{-1}\ang{y}^{-1}\leq \ang{\frac{ax+by}{\abs{a}+\abs{b}}}^{-1}\br{\ang{x}^{-1}+\ang{y}^{-1}}.
    \end{equation*}
\end{lemma}

Fixing \(\br{m,n}\in\J_{U,U}\), the first step consists of distributing via Cauchy-Schwarz the time-frequency correlation weight on each of the two Gabor coefficient terms:
\begin{equation}\label{eq_Gpqw_term}
    \Lambda\br{F^{\br{m}},G^{\br{n}}}\leq
    \sum_{p,q}
    \sum_{w\in\mathcal{G}_{p,q}^{\br{n}}}
    \br{
        \mathop{\mathbb{E}}_{r\eqsim\sqrt{\lambda}}
        \sum_{
            u,v
        }
        \frac{
            \br{F^{\br{m}}_{p-r,q-\frac{r^2}{\sqrt{\lambda}},u,v}}^2
        }{\ang{3w\frac{r^2}{\lambda}-2v\frac{r}{\sqrt{\lambda}}-u}^N}
}^{\frac{1}{2}}
    \cdot
    \br{
       \mathop{\mathbb{E}}_{r\eqsim\sqrt{\lambda}}
        \sum_{
            \substack{u,v:\\
            \abs{v}\leq 3\sqrt{\lambda}
        }}
        \frac{
            G_{p,q,r,u,v,w}^2
        }{\ang{3w\frac{r^2}{\lambda}-2v\frac{r}{\sqrt{\lambda}}-u}^N}
    }^{\frac{1}{2}}.
\end{equation}
We start our analysis by focussing first on the second term in \eqref{eq_Gpqw_term}. For a fixed \(w\) and \(\frac{\sqrt{\lambda}}{3}
\leq r \leq
3\sqrt{\lambda}\),  consider the following layer cake decomposition with respect to \(\ang{3w\frac{r^2}{\lambda}-2v\frac{r}{\sqrt{\lambda}}-u}^{-N}\):
\begin{equation}\label{eq_sum_k_1_G2}
    \sum_{
        \substack{u,v:\\
        \abs{v}\leq 3\sqrt{\lambda}
    }}
    \frac{
        G_{p,q,r,u,v,w}^2
    }{\ang{3w\frac{r^2}{\lambda}-2v\frac{r}{\sqrt{\lambda}}-u}^N}
    \lesssim
    \sum_{k}
        \ang{k}^{-N}
        \sum_{
            \substack{u,v:\\
            \abs{v}\leq 3\sqrt{\lambda}
        }}
        \1_{\Br{-1,1}}\br{
            3w\frac{r^2}{\lambda}-2v\frac{r}{\sqrt{\lambda}}-u-k
        }
            G_{p,q,r,u,v,w}^2.
\end{equation}
Next, recalling the notation in Section \ref{Not}, we introduce the frequency projection operator
\begin{equation*}
    \widehat{\pi_k f}\br{\xi,\eta}:=\varphi\br{\frac{\eta}{4\lambda}}\varphi\bigg(
        \frac{
            3w\frac{r^2}{\lambda}-2\frac{\eta}{\sqrt{\lambda}}\cdot\frac{r}{\sqrt{\lambda}}-\frac{\xi}{\sqrt{\lambda}}-k
        }{
            8
        }
    \bigg)\widehat{f}\br{\xi,\eta}.
\end{equation*}
By construction, we have:
\begin{equation*}
    G_{p,q,r,u,v,w}=
    \abs{
        \ang{
            \pi_k
            \br{\overline{e\br{\frac{ar^3}{\lambda^{\frac{3}{2}}}}}\chi\br{\frac{a}{\sqrt{\lambda}}-w}\1_E},
        \Psi_{p,q,u,v}
        }
    }=:G_{p,q,r,u,v,w,k}
\end{equation*}
for all \(u,v\) satisfying \(\abs{3w\frac{r^2}{\lambda}-2v\frac{r}{\sqrt{\lambda}}-u-k}\leq 1\). This allows us to rewrite \eqref{eq_sum_k_1_G2} as
\begin{equation*}
    \eqref{eq_sum_k_1_G2}= \sum_{k}
        \ang{k}^{-N}
        \sum_{\substack{
            u,v:\\
            \abs{v}\leq 3\sqrt{\lambda}
        }}
        \1_{\Br{-1,1}}\br{
            3w\frac{r^2}{\lambda}-2v\frac{r}{\sqrt{\lambda}}-u-k
        }G_{p,q,r,u,v,w,k}^2
\end{equation*}
and further dominate the above trivially by
\begin{equation*}
    \sum_{k}
        \ang{k}^{-N}
        \sum_{u,v}
        G_{p,q,r,u,v,w,k}^2.
\end{equation*}
At this point we apply Lemma \ref{prop_Gabor_bds} to further dominate the above expression by
\begin{equation}\label{eq_sum_k_proj}
    \sum_{k}
        \ang{k}^{-N}
        \nrm{
            \frac{
                \pi_k
                \br{\overline{e\br{\frac{ar^3}{\lambda^{\frac{3}{2}}}}}\chi\br{\frac{a}{\sqrt{\lambda}}-w}\1_E}\br{x,y}
            }{
                \ang{\sqrt{\lambda}x-p}^{9N}
                \ang{\sqrt{\lambda}y-q}^{9N}
            }
        }^2_{L^2\br{dxdy}}.
\end{equation}
Once here, we take a brief detour in our exposition in order to estimate \(\pi_k\) from the physical side: more precisely, we write \(\pi_k f= \mathcal{P}_k\ast f\) and compute the kernel
\begin{equation*}
    \mathcal{P}_k\br{x,y}:=
    \int
        \varphi\br{\frac{\eta}{4\lambda}}
        \varphi\bigg(
            \frac{
                3w\frac{r^2}{\lambda}-2\frac{\eta}{\sqrt{\lambda}}\cdot\frac{r}{\sqrt{\lambda}}-\frac{\xi}{\sqrt{\lambda}}-k
            }{
                8
            }
        \bigg)
        e\br{x\xi+y\eta}
    d\xi d\eta.
\end{equation*}
Direct calculation shows that:
\begin{equation}\nonumber
    \abs{
        \mathcal{P}_k\br{x,y}
    }
    \lesssim \lambda^{\frac{3}{2}}
    \abs{
        \widehat{\varphi}\br{
            8\sqrt{\lambda}x
        }
        \widehat{\varphi}\br{
            4\sqrt{\lambda}\br{2rx-\sqrt{\lambda}y}
        }
    }
    \lesssim
    \lambda^{\frac{3}{2}}
    \ang{
        \sqrt{\lambda}x
    }^{-3N}
    \ang{
        \sqrt{\lambda}\br{2rx-\sqrt{\lambda}y}
    }^{-3N}.
\end{equation}
Consequently,  uniform in \(k\in\Z\), the following estimate holds:
\begin{align}\label{projupb}
    &
    \abs{
        \pi_k
        \br{\overline{e\br{\frac{ar^3}{\lambda^{\frac{3}{2}}}}}\chi\br{\frac{a}{\sqrt{\lambda}}-w}\1_E}\br{x,y}
    }\nonumber\\
    &\qquad\qquad\lesssim
    \lambda^{\frac{3}{2}}
    \int_{E_w}
        \ang{
            \sqrt{\lambda}\br{x-x'}
        }^{-3N}
        \ang{
            \sqrt{\lambda}\br{2r\br{x-x'}-\sqrt{\lambda}\br{y-y'}}
        }^{-3N}
    dx'dy'.
\end{align}
Returning now to \eqref{eq_sum_k_proj}, we apply \eqref{projupb}, sum over \(k\in\Z\), and finally expand the \(\nrm{\cdot}_{L^2}^2\) in order to deduce
\begin{align*}
    \eqref{eq_sum_k_proj}\lesssim &
    \lambda^3
    \int_{\R^2}
        \int_{E_w^2}
            \prod_{j=0,1}
                \ang{
                    \sqrt{\lambda}\br{x-x_j}
                }^{-3N}
                \ang{
                    \sqrt{\lambda}\br{2r\br{x-x_j}-\sqrt{\lambda}\br{y-y_j}}
                }^{-3N}\\
    &\hspace{4.5em}
    dx_0dy_0dx_1dy_1
    \cdot
    \ang{
        \sqrt{\lambda}x-p
    }^{-18N}
    \ang{
        \sqrt{\lambda}y-q
    }^{-18N}
    dxdy.
\end{align*}
We now apply Lemma \ref{lem_jap_mul_diff} multiple times in order to
\begin{itemize}
    \item create the term \(\ang{\sqrt{\lambda}x_j-p}^{-N}\) via the estimate
    \begin{equation}\nonumber
            \ang{\sqrt{\lambda}\br{x-x_j}}^{-N}\ang{\sqrt{\lambda}x-p}^{-N}
        \lesssim
        \ang{\sqrt{\lambda}x_j-p}^{-N}.
    \end{equation}
    \item create the term \(\ang{\sqrt{\lambda}y_j-q}^{-N}\) via the estimate
    \begin{align*}
        &\ang{\sqrt{\lambda}\br{x-x_j}}^{-N}
        \ang{
            \sqrt{\lambda}\br{2r\br{x-x_j}-\sqrt{\lambda}\br{y-y_j}}
        }^{-N}
        \ang{\sqrt{\lambda}y-q}^{-N}\\
        \lesssim &
        \ang{
            \sqrt{\lambda}\br{y-y_j}
        }^{-N}
        \ang{\sqrt{\lambda}y-q}^{-N}
        \lesssim
        \ang{\sqrt{\lambda}y_j-q}^{-N}.
    \end{align*}
    \item create the term \(\ang{
        \sqrt{\lambda}\br{2r\br{x_1-x_0}-\sqrt{\lambda}\br{y_1-y_0}}
    }^{-N}\) via the estimate
    \begin{equation}\nonumber
        \prod_{j=0,1}
            \ang{
                \sqrt{\lambda}\br{2r\br{x-x_j}-\sqrt{\lambda}\br{y-y_j}}
            }^{-N}
        \lesssim
        \ang{
            \sqrt{\lambda}\br{2r\br{x_1-x_0}-\sqrt{\lambda}\br{y_1-y_0}}
        }^{-N}.
    \end{equation}
\end{itemize}
Combining all of the above, we derive
\begin{align*}
    \eqref{eq_sum_k_proj}
    \lesssim &
    \lambda^\frac{3}{2}
    \int_{E_w^2}
        \int_{\R^2}
            \sqrt{\lambda}
            \ang{
                \sqrt{\lambda}\br{x-x_0}
            }^{-N}
            \cdot
            \lambda
            \ang{
                \sqrt{\lambda}\br{2r\br{x-x_0}-\sqrt{\lambda}\br{y-y_0}}
            }^{-N}
        dxdy
        \nonumber\\
    \cdot &
        \ang{
            \sqrt{\lambda}\br{2r\br{x_1-x_0}-\sqrt{\lambda}\br{y_1-y_0}}
        }^{-N}
    d\mu_{p,q}\br{x_0,y_0}
    d\mu_{p,q}\br{x_1,y_1}\nonumber\\
    \eqsim &
    \lambda^{\frac{3}{2}}
    \!\!
    \int_{E_w^2}
    \!\!
        \ang{
            \sqrt{\lambda}\br{2r\br{x_1-x_0}-\sqrt{\lambda}\br{y_1-y_0}}
        }^{-N}
    \!\!
    d\mu_{p,q}\br{x_0,y_0}
    d\mu_{p,q}\br{x_1,y_1}.
\end{align*}

In conclusion, as a consequence of \eqref{eq_sum_k_proj}, we have now established the following estimate:
\begin{align}%
    &
    \mathop{\mathbb{E}}_{\frac{\sqrt{\lambda}}{3}\leq r \leq 3\sqrt{\lambda}}
        \sum_{
            \substack{u,v:\\
            \abs{v}\leq 3\sqrt{\lambda}
        }}
        \frac{
            G_{p,q,r,u,v,w}^2
        }{\ang{3w\frac{r^2}{\lambda}-2v\frac{r}{\sqrt{\lambda}}-u}^N}\nonumber
    \\
    \lesssim &
    \lambda^{\frac{3}{2}}
    \!\!
    \int_{E_w^2}
        \mathop{\mathbb{E}}_{r\eqsim \sqrt{\lambda}}
    \!\!
        \ang{
            \sqrt{\lambda}\br{2r\br{x_1-x_0}-\sqrt{\lambda}\br{y_1-y_0}}
        }^{-N}
    \!\!
    d\mu_{p,q}\br{x_0,y_0}
    d\mu_{p,q}\br{x_1,y_1}\label{eq_Lambda_UU_G_fact_int_bd}.
\end{align}%
With these we have achieved a simplified expression that encodes the curvature structure within the $r$-average. In order to exploit this structure and aim for the final desired decay we apply Cauchy-Schwarz inequality in the above integral and deduce
\begin{align}\label{csexplc}
   \eqref{eq_Lambda_UU_G_fact_int_bd} \lesssim &
    \lambda^{\frac{3}{2}}
    \br{\mu_{p,q}\otimes\mu_{p,q}}^{\frac{1}{2}}\br{E^2_w}
    \Bigg(
    \int_{E_w^2}
        \bigg(
            \mathop{\mathbb{E}}_{r\eqsim \sqrt{\lambda}}
                    \ang{
                        \sqrt{\lambda}
                        \br{
                            2r\br{x_1-x_0}-\sqrt{\lambda}\br{y_1-y_0}
                        }
                    }^{-N}
        \bigg)^2\nonumber\\
        &\hspace{13em}\cdot
    d\mu_{p,q}\br{x_0,y_0}
    d\mu_{p,q}\br{x_1,y_1}
    \Bigg)^{\frac{1}{2}}\nonumber\\
    \leq &
    \lambda\mu_{p,q}\br{E_w}
    \Bigg(
        \int_{E_w\times \R^2}
            \br{
                    \mathop{\mathbb{E}}_{r\eqsim \sqrt{\lambda}}
                        \ang{
                            \sqrt{\lambda}
                            \br{
                                2r x_\Delta-\sqrt{\lambda}y_\Delta
                            }
                        }^{-N}
            }^2
          d\mu_{p,q}\br{x_0,y_0}
        \sqrt{\lambda} dx_\Delta \sqrt{\lambda} dy_\Delta
    \Bigg)^{\frac{1}{2}}\nonumber\\
    =&
    \lambda\mu_{p,q}^{\frac{3}{2}}\br{E_w}
    \br{
        \int_{\R^2}
            \br{
                \mathop{\mathbb{E}}_{r\eqsim \sqrt{\lambda}}
                    \ang{
                            2r x-
                            \sqrt{\lambda}y
                    }^{-N}
            }^2
            dxdy
    }^{\frac{1}{2}}.
\end{align}
Expanding now the square and applying Lemma \ref{lem_jap_mul_diff} we obtain
\begin{align}\label{exp}
    \br{
            \mathop{\mathbb{E}}_{r\eqsim \sqrt{\lambda}}
                \ang{
                        2r x-
                        \sqrt{\lambda}y
                }^{-N}
    }^2
    \lesssim &
        \mathop{\mathbb{E}}_{r\eqsim \sqrt{\lambda}}
        \mathop{\mathbb{E}}_{s\eqsim \sqrt{\lambda}}
        \ang{
                2\br{r-s} x
        }^{-N}
        \ang{
            2s x-
            \sqrt{\lambda}y
        }^{-N}
    \nonumber\\
    \lesssim &
        \mathop{\mathbb{E}}_{\abs{r}\lesssim \sqrt{\lambda}}
        \ang{
            2r x
        }^{-N}
        \mathop{\mathbb{E}}_{s\eqsim \sqrt{\lambda}}
        \ang{
            2s x-
            \sqrt{\lambda}y
        }^{-N}.
\end{align}
Putting now together \eqref{eq_Lambda_UU_G_fact_int_bd}, \eqref{csexplc} and \eqref{exp}, we deduce
\begin{align}\label{Gterm}
   &
   \mathop{\mathbb{E}}_{\frac{\sqrt{\lambda}}{3}\leq r \leq 3\sqrt{\lambda}}
        \sum_{
            \substack{u,v:\\
            \abs{v}\leq 3\sqrt{\lambda}
        }}
        \frac{
            G_{p,q,r,u,v,w}^2
        }{\ang{3w\frac{r^2}{\lambda}-2v\frac{r}{\sqrt{\lambda}}-u}^N}
    \nonumber\\
    \lesssim &
    \lambda\mu_{p,q}^{\frac{3}{2}}\br{E_w}
    \br{
        \int
            \br{
                \mathop{\mathbb{E}}_{\abs{r}\lesssim \sqrt{\lambda}}
                    \ang{
                        2r x
                    }^{-N}
                \int
                    \mathop{\mathbb{E}}_{s\eqsim \sqrt{\lambda}}
                    \ang{
                        2s x-
                        \sqrt{\lambda}y
                    }^{-N}
                dy
            }
        dx
    }^{\frac{1}{2}}\nonumber\\
    \lesssim &
    \sqrt{\lambda}
    \log^\frac{1}{2}\lambda\,\mu_{p,q}^\frac{3}{2}\br{E_w}.
\end{align}
Thus, plugging \eqref{Gterm} into \eqref{eq_Gpqw_term} we conclude for this stage that
\begin{equation}\label{intL}
    \Lambda\br{F^{\br{m}},G^{\br{n}}}
    \lesssim
    \sum_{p,q}
    \sum_{w\in\mathcal{G}_{p,q}^{\br{n}}}
        \br{
            \mathop{\mathbb{E}}_{r\eqsim \sqrt{\lambda}}
            \sum_{u,v}
            \frac{
                \br{F^{\br{m}}_{p-r,q-\frac{r^2}{\sqrt{\lambda}},u,v}}^2
            }{
                \ang{3w\frac{r^2}{\lambda}-2v\frac{r}{\sqrt{\lambda}}-u}^N}
        }^{\frac{1}{2}}
    \lambda^{\frac{1}{4}}
    \log^\frac{1}{4}\lambda\mu_{p,q}^\frac{3}{4}\br{E_w}.
\end{equation}
We now turn our attention towards the $F$-term. Using now the definition of \(\mathcal{G}^{\br{n}}_{p,q}\), \eqref{eq_cheby}, Cauchy--Schwarz and H\"older, we have
\begin{equation}\label{intl1}
    \eqref{intL}\lesssim
    \frac{2^{\frac{3n}{4}}}{\lambda^{\frac{3}{8}}}\cdot
    \frac{\lambda^{\frac{3}{4}}\log^{\frac{1}{4}}\lambda}{2^n}
    \sum_{p,q}
    \sup_{w}
        \br{
            \mathop{\mathbb{E}}_{r\eqsim \sqrt{\lambda}}
            \sum_{u,v}
            \frac{
                \br{F^{\br{m}}_{p-r,q-\frac{r^2}{\sqrt{\lambda}},u,v}}^2
            }{
                \ang{3w\frac{r^2}{\lambda}-2v\frac{r}{\sqrt{\lambda}}-u}^N
            }
        }^{\frac{1}{2}}
    \mu_{p,q}^{\frac{3}{4}}\br{E}.
\end{equation}
Linearizing now the supremum in the above expression via a measurable function \(w\br{\cdot,\cdot}\) we further have
\begin{equation}\label{intl2}
    \eqref{intl1}\lesssim
        2^{-\frac{n}{4}}
        \lambda^{\frac{3}{8}}
        \log^{\frac{1}{4}}\lambda
    \sum_{p,q}
        \br{
            \mathop{\mathbb{E}}_{r\eqsim \sqrt{\lambda}}
            \sum_{u,v}
            \frac{
                \br{F^{\br{m}}_{p-r,q-\frac{r^2}{\sqrt{\lambda}},u,v}}^2
            }{
                \ang{3w\br{p,q}\frac{r^2}{\lambda}-2v\frac{r}{\sqrt{\lambda}}-u}^N}
        }^{\frac{1}{2}}
    \mu_{p,q}^{\frac{3}{4}}\br{E}.
\end{equation}
Using that \(\mu_{p,q}\br{E}^{\frac{3}{4}}\lesssim \lambda^{-\frac{1}{4}}\mu_{p,q}\br{E}^{\frac{1}{2}}\), we apply Cauchy-Schwarz inequality in \(p,q\) variables and get
\begin{equation}\label{eq_factor_E_F_squ_tf_E}
   \eqref{intl2}\lesssim
        2^{-\frac{n}{4}}
        \lambda^{\frac{1}{8}}
        \log^{\frac{1}{4}}\lambda
    \br{
        \mathop{\mathbb{E}}_{r\eqsim \sqrt{\lambda}}
        \sum_{p,q,u,v}
        \frac{
            \br{F^{\br{m}}_{p-r,q-\frac{r^2}{\sqrt{\lambda}},u,v}}^2
        }{
            \ang{3w\br{p,q}\frac{r^2}{\lambda}-2v\frac{r}{\sqrt{\lambda}}-u}^N}
    }^{\frac{1}{2}}
    \abs{E}^{\frac{1}{2}}.
\end{equation}
It now remains to estimate the expression
\begin{equation}\label{eq_Er_Spquv_F_tf}
    \mathop{\mathbb{E}}_{r\eqsim \sqrt{\lambda}}
        \sum_{p,q,u,v}
        \frac{
            \br{F^{\br{m}}_{p-r,q-\frac{r^2}{\sqrt{\lambda}},u,v}}^2
        }{
            \ang{3w\br{p,q}\frac{r^2}{\lambda}-2v\frac{r}{\sqrt{\lambda}}-u}^N}
    .
\end{equation}
We first perform a change of variables in order to decouple the \(r\) variable interaction between \(F^{\br{m}}\) and the time-frequency correlation term. Setting for notational simplicity \(w\br{p,q,r}:=3w\br{ p+r, q-\left\lfloor-\frac{r^2}{\sqrt{\lambda}}\right\rfloor}\frac{r^2}{\lambda}\), we have
\begin{equation*}
    \eqref{eq_Er_Spquv_F_tf}
    \lesssim
    \mathop{\mathbb{E}}_{r\eqsim \sqrt{\lambda}}
    \sum_{
        p,q,
        u,v
    }
        \br{F^{\br{m}}_{p,q,u,v}}^2
        \ang{
            w\br{ p,q,r}-2v\frac{r}{\sqrt{\lambda}}-u
        }^{-N}.
\end{equation*}

To extract cancellation out of the time-frequency correlation term, we move the expectation into the inner layer and perform another Cauchy-Schwarz inequality in the \(u,v\) variables. This gives the following estimate:
\begin{equation*}
    \lesssim
    \sum_{p,q}
    \br{
        \sum_{u,v}
            \br{F^{\br{m}}_{p,q,u,v}}^4
    }^{\frac{1}{2}}
    \br{
        \sum_{\substack{u,v:\\
        \abs{v}\lesssim \sqrt{\lambda}}}
        \br{
            \mathop{\mathbb{E}}_{
                r\eqsim \sqrt{\lambda}
            }
            \ang{
                w\br{p,q,r}-2v\frac{r}{\sqrt{\lambda}}-u
            }^{-N}
        }^2
    }^{\frac{1}{2}}.
\end{equation*}
On the one hand, from \eqref{eq_cheby} and the natural bound on \(F^{\br{m}}_{p,q,u,v}\),
\begin{equation}\label{eq_Lambda_UU_F1st_fac_bd}
    \br{
        \sum_{u,v}
            \br{F^{\br{m}}_{p,q,u,v}}^4
    }^{\frac{1}{2}}
    \lesssim \frac{\sqrt{\lambda}}{2^m}
    \frac{2^{2m}}{\lambda}\nrm{f}^2_{L^2\br{d\mu_{p,q}}}
    =\frac{2^{m}}{\sqrt{\lambda}}\nrm{f}^2_{L^2\br{\mu_{p,q}}}
    .
\end{equation}
On the other hand, we aim to obtain an appropriate control uniform in \(p,q\) on:
\begin{equation}\label{eq_Suv_E_sqed}
    \br{
        \sum_{\substack{u,v:\\
            \abs{v}\lesssim \sqrt{\lambda}}}
            \br{
                \mathop{\mathbb{E}}_{
                    r\eqsim \sqrt{\lambda}
                }
                \ang{
                    w\br{p,q,r}-2v\frac{r}{\sqrt{\lambda}}-u
                }^{-N}
            }^2
    }^{\frac{1}{2}}.
\end{equation}
For this we expand the square and apply Lemma \ref{lem_jap_mul_diff} in order to get
\begin{align}%
    &
    \br{
        \mathop{\mathbb{E}}_{
            r\eqsim \sqrt{\lambda}
        }
        \ang{
            w\br{p,q,r}-2v\frac{r}{\sqrt{\lambda}}-u
        }^{-N}
    }^2\nonumber\\
    \lesssim &
    \mathop{\mathbb{E}}_{
        s\eqsim \sqrt{\lambda}
    }
    \mathop{\mathbb{E}}_{
        \abs{r}\lesssim \sqrt{\lambda}}
    \ang{
        \left.w\br{p,q,\cdot}\right\vert^{s+r}_s-2v\frac{r}{\sqrt{\lambda}}
    }^{-N}
    \ang{
        w\br{p,q,s}-2v\frac{s}{\sqrt{\lambda}}-u
    }^{-N}\label{eq_EE_w2v}
\end{align}%
Putting the last two relations together we deduce
\begin{equation*}
    \eqref{eq_Suv_E_sqed}\lesssim
    \br{
        \sum_{\substack{
            v:\\
            \abs{v}\lesssim\sqrt{\lambda}
        }}
        \sum_{u}
            \eqref{eq_EE_w2v}
    }^{\frac{1}{2}}
    \lesssim
    \log^{\frac{1}{2}}\lambda\label{eq_Lambda_UU_F2nd_fac_bd}
    .
\end{equation*}
Returning now to \eqref{eq_Er_Spquv_F_tf} we further deduce that
\begin{equation}\label{sec}
    \sqrt{\eqref{eq_Er_Spquv_F_tf}}\lesssim
    \br{
        \sum_{p,q}
            \eqref{eq_Lambda_UU_F1st_fac_bd}
            \cdot
            \eqref{eq_Lambda_UU_F2nd_fac_bd}
    }^{\frac{1}{2}}
    \lesssim
    \frac{
        2^{\frac{m}{2}}\log^{\frac{1}{4}}\lambda
    }{\lambda^{\frac{1}{4}}}
    \nrm{f}_{L^2}.
\end{equation}
Finally, putting together \eqref{intL}--\eqref{eq_factor_E_F_squ_tf_E}, \eqref{eq_Er_Spquv_F_tf} and \eqref{sec} we have
\begin{equation*}
    \Lambda\br{F^{\br{m}},G^{\br{n}}}
    \lesssim
    \frac{
        2^{\frac{2m-n}{4}}\log^{\frac{1}{2}}\lambda
    }{
        \lambda^{\frac{1}{8}}
    }
    \nrm{f}_{L^2}\abs{E}^{\frac{1}{2}}\,,
\end{equation*}
which, recalling the definition of $\Lambda_{UU}\br{F,G}$, implies
\begin{equation*}
    \Lambda_{UU}\br{F,G}
    \lesssim
    \lambda^{-\epsilon}\log^{\frac{5}{2}}\lambda
    \nrm{f}_{L^2}\abs{E}^{\frac{1}{2}}.\qed
\end{equation*}

\subsubsection{Treatment of the sparse-sparse component $\Lambda_{SS}$ assuming control on the time-frequency correlation set}
Fix \(\br{m,n}\in\J_{S,S}\). From \eqref{eq_cheby}, for every \(p,q\in\Z\) the cardinals of the sets \(\F^{\br{m}}_{p,q}\) and \(\mathcal{G}^{\br{n}}_{p,q}\)
are small, that is
\begin{equation*}
    \#\F^{\br{m}}_{p,q}\lesssim \lambda^{12\epsilon},\quad
    \#\mathcal{G}^{\br{n}}_{p,q}\lesssim \lambda^{6\epsilon}.
\end{equation*}
Thus, via a maximality argument, we will be able to remove the sum in \(u,v,w\) from our discretized model form. This will produce eventually an expression involving only the spatial information of the functions \(f\) and \(\1_E\) with the frequency information being encoded into time-frequency correlation terms.

Indeed, in order to make precise the above description, we start by rearranging the sum:
\begin{equation}\label{ss1}
    \Lambda\br{F^{\br{m}},G^{\br{n}}}=
    \sum_{p,q}
    \sum_w
    \sum_{
        u,v
    }
    \mathop{\mathbb{E}}_{r\eqsim \sqrt{\lambda}}
    \frac{
        F^{\br{m}}_{p-r,q-\frac{r^2}{\sqrt{\lambda}},u,v}
        G^{\br{n}}_{p,q,r,u,v,w}
    }{
        \ang{
            3w\frac{r^2}{\lambda}
            -2v\frac{r}{\sqrt{\lambda}}
            -u
        }^N
    }.
\end{equation}
Accounting now for the smallness of the set \(\mathcal{G}^{\br{n}}_{p,q}\),we remove the $w$-parameter by introducing a measurable function \(w\br{\cdot,\cdot}\) with \(w\br{p,q}\in\mathcal{G}^{\br{n}}_{p,q}\) in order to deduce
\begin{equation}\label{ss2}
   \eqref{ss1} \leq
    \lambda^{6\epsilon}
    \sum_{p,q}
    \sum_{
        u,v
    }
    \mathop{\mathbb{E}}_{r\eqsim \sqrt{\lambda}}
    \frac{
        F^{\br{m}}_{p-r,q-\frac{r^2}{\sqrt{\lambda}},u,v}
        G^{\br{n}}_{p,q,r,u,v,w\br{p,q}}
    }{
        \ang{
            3w\br{p,q}\frac{r^2}{\lambda}
            -2v\frac{r}{\sqrt{\lambda}}
            -u
        }^N
    }.
\end{equation}
We further simplify the above expression via \eqref{eq_G_bd_2_mu}, \eqref{eq_Gpq_set}, and the fact that \(\mu_{p,q}\br{E}\lesssim\mu_{p,q}^{\frac{1}{2}}\br{E}/\sqrt{\lambda}\):
\begin{equation}\label{ss3}
    \eqref{ss2}\lesssim
    \lambda^{6\epsilon}
    \sum_{p,q}
    \sum_{
        u,v
    }
    \mathop{\mathbb{E}}_{r\eqsim \sqrt{\lambda}}
    \frac{
        F^{\br{m}}_{p-r,q-\frac{r^2}{\sqrt{\lambda}},u,v}\:\,
        \mu_{p,q}^{\frac{1}{2}}\br{E}
    }{
        \ang{
            3w\br{p,q}\frac{r^2}{\lambda}
            -2v\frac{r}{\sqrt{\lambda}}
            -u
        }^N
    }.
\end{equation}
We then perform a change of variable to decouple the \(r\)-variable interaction between \(F^{\br{m}}\) and the time-frequency correlation term:
\begin{equation}\label{ss4}
    \eqref{ss3}\lesssim
    \lambda^{6\epsilon}
    \sum_{p,q}
    \sum_{
        u,v
    }
    \mathop{\mathbb{E}}_{r\eqsim \sqrt{\lambda}}
    \frac{
        F^{\br{m}}_{p,q,u,v}\:\,
        \mu_{p+r,q-\left\lfloor-\frac{r^2}{\sqrt{\lambda}}\right\rfloor}^{\frac{1}{2}}\br{E}
    }{
        \ang{
            3w\br{p+r,q-\left\lfloor-\frac{r^2}{\sqrt{\lambda}}\right\rfloor}\frac{r^2}{\lambda}
            -2v\frac{r}{\sqrt{\lambda}}
            -u
        }^N
    }.
\end{equation}
Once again, two measurable functions \(u\br{\cdot,\cdot},v\br{\cdot,\cdot}\) satisfying \(\br{u\br{p,q},v\br{p,q}}\in\F^{\br{m}}_{p,q}\) are introduced in order to discard the summation in $u,v$ parameters and dominate the above expression with a loss accounting for the size of \(\F^{\br{m}}_{p,q}\):
\begin{equation*}
   \eqref{ss4} \leq
    \lambda^{18\epsilon}
    \sum_{p,q}
    \mathop{\mathbb{E}}_{r\eqsim \sqrt{\lambda}}
    \frac{
        F^{\br{m}}_{p,q,u\br{p,q},v\br{p,q}}\:\,
        \mu_{p+r,q-\left\lfloor-\frac{r^2}{\sqrt{\lambda}}\right\rfloor}^{\frac{1}{2}}\br{E}
    }{
        \ang{
            3w\br{p+r,q-\left\lfloor-\frac{r^2}{\sqrt{\lambda}}\right\rfloor}\frac{r^2}{\lambda}
            -2v\br{p,q}\frac{r}{\sqrt{\lambda}}
            -u\br{p,q}
        }^N
    }.
\end{equation*}
Furthermore, using that $\mu_{p+r,q-\left\lfloor-\frac{r^2}{\sqrt{\lambda}}\right\rfloor}\br{E}\eqsim \mu_{p+r,q+\frac{r^2}{\sqrt{\lambda}}}\br{E}$ condition \eqref{eq_Fpq_set} and the notation
\begin{equation}\nonumber
    \kappa\br{p,q,r}:=
    \ang{
        3w\br{p+r,q-\left\lfloor-\frac{r^2}{\sqrt{\lambda}}\right\rfloor}\frac{r^2}{\lambda}
        -2v\br{p,q}\frac{r}{\sqrt{\lambda}}
        -u\br{p,q}
    }^{-N}\,,
\end{equation}
we now deduce that
\begin{equation}\label{ssfin}
    \Lambda\br{F^{\br{m}},G^{\br{n}}}\lesssim
    \lambda^{18\epsilon}
    \sum_{p,q}
        \mathop{\mathbb{E}}_{
            r\eqsim \sqrt{\lambda}
        }
            \nrm{f}_{L^2\br{\mu_{p,q}}}
            \mu^{\frac{1}{2}}_{p+r,q+\frac{r^2}{\sqrt{\lambda}}}\br{E}
            \kappa\br{p,q,r}.
\end{equation}
The expression on the right is "fully physical." Namely, the frequency behavior of the functions \(f,\1_E\) plays no longer any role. Additionally, via almost orthogonality, it suffices to consider the local variant of the model form estimate. Namely, with the obvious notational correspondences, we may consider instead
    \begin{equation}\label{eq_phys_loc_mod}
        \Lambda_\ast\br{F,G}:=
        \sum_{\substack{p,q\\
        \abs{p},\abs{q}\lesssim\sqrt{\lambda} }}
        \mathop{\mathbb{E}}_{
            r\eqsim \sqrt{\lambda}
        }
            F_{p,q}
            G_{p+r,\left\lfloor q+\frac{r^2}{\sqrt{\lambda}}\right\rfloor}
            \kappa\br{p,q,r}.
    \end{equation}
    and aim to show the following local estimates:

\begin{proposition}\label{prop_phys_loc_mod_bd} With the previous notation we have
    \begin{equation}\label{stat}
        \Lambda_\ast\br{F,G}\lesssim
        \lambda^{-19\epsilon}
        \nrm{F_{p,q}}_{\ell^2\br{p,q}}
        \nrm{G_{p,q}}_{\ell^2\br{p,q}}
    \end{equation}
\end{proposition}
\medskip

Assuming for the moment that Proposition \ref{prop_phys_loc_mod_bd} holds then, by taking \(F_{p,q}:=\nrm{f}_{L^2\br{\mu_{p,q}}}\) and \(G_{p,q}:=\mu_{p,q}^{\frac{1}{2}}\br{E}\) and using \eqref{ssfin}, we deduce the full desired estimate
\begin{equation*}
    \Lambda_{SS}(F,G)\lesssim
    \lambda^{-\epsilon} \log^2\lambda
    \nrm{f}_{L^2}\abs{E}^{\frac{1}{2}}.\qed
\end{equation*}

\medskip

We now turn our attention towards Proposition \ref{prop_phys_loc_mod_bd}; its proof relies on the following classical result\footnote{For an alternative route that avoids entirely any interpolation theory and is based on purely $L^2$ methods and some basic combinatorics, we invite the reader to consult Sections \ref{Incimprovest} and  \ref{StrongL2}.}:
\begin{lemma}[\textsf{\(L^p\) improving of the averaging operator along parabola}, \cite{LW73}, \cite{CM98}]\label{thm_Lp_impro}
    \begin{equation}\label{Improv}
        \nrm{
            \int_{\frac{1}{2}}^2
                f\br{x+t,y+t^2}
            dt
        }_{L^3\br{dxdy}}
        \lesssim \nrm{f}_{L^{\frac{3}{2}}}
    \end{equation}
\end{lemma}
and on the following key statement:

\begin{lemma}[\textsf{Time-frequency correlation level set analysis}]\label{lem_tf_cor_lev_set_ana}
There is a universal constant \(\delta>0\). For \(\lambda\gg 1\) and measurable functions \(u,v,w:\R\to\R\) satisfying \(\abs{u}\vee\abs{v}\vee\abs{w}\eqsim \lambda\), we have the following decay estimate:
\begin{equation*}
    \nrm{
        \ang{
            w\br{x+t,y+t^2}t^2+v\br{x,y}t+u\br{x,y}
        }^{-N}
    }_{L^1\br{dxdydt,\Br{-1,1}^2\times\bR{\frac{1}{2},2}}}
    \lesssim \lambda^{-\delta}.
\end{equation*}
\end{lemma}

In what follows we will prove Proposition \ref{prop_phys_loc_mod_bd} assuming that Lemma \ref{lem_tf_cor_lev_set_ana} holds. The proof of the latter will be deferred to the next section.

\begin{proof}[\textsf{Proof of Proposition \ref{prop_phys_loc_mod_bd}}]
    Observe that for \(q\in\Z\) and \(\alpha\in\R\setminus\Z\), we have the following identity:
    \begin{equation*}
        q-\lfloor -\alpha \rfloor =q+\lceil \alpha \rceil=\lceil q+ \alpha \rceil,\quad \forall \alpha\in\R\setminus\Z.
    \end{equation*}
    This suggests that we extend \(u,v,w\) from \(\Z^2\) to \(\R^2\) via \(\lceil\cdot\rceil\):
    \begin{equation*}
        \widetilde{u}\br{p,q}:=
        -u\br{\lceil p\rceil ,\lceil q\rceil}
        ,\quad
        \widetilde{v}\br{p,q}:=
        -2v\br{\lceil p\rfloor,\lceil q\rceil},\quad
        \widetilde{w}\br{p,q}:=
        3w\br{\lceil p\rceil,\lceil q\rceil}.
    \end{equation*}
    We extend the functions \(F,G\) naturally from \(\Z^2\) to \(\R^2\) correspondingly and extend \(\kappa\) from \(\Z^3\) to \(\R^3\) by setting:
    \begin{equation*}
        \widetilde{\kappa}\br{p,q,r}:=
        \ang{
            \widetilde{w}\br{ p+r, q+\frac{r^2}{\sqrt{\lambda}}}\frac{r^2}{\lambda}
            +\widetilde{v}\br{p,q}
            \frac{r}{\sqrt{\lambda}}
            +\widetilde{u}\br{p,q}
        }^{-N}.
    \end{equation*}
    This allows us to reinterpret the sum in \eqref{eq_phys_loc_mod} as an integral:
    \begin{equation*}
        \Lambda_\ast\br{F,G}
        \lesssim
        I_\ast\br{f,g}:=
        \int_{\abs{x},\abs{y} \lesssim 1\eqsim t}
            f\br{x,y}
            g\br{x+t,y+t^2}
            K\br{x,y,t}
        dxdydt,
    \end{equation*}
    where we define:
    \begin{equation*}
        f\br{x,y}:=\sqrt{\lambda}F_{\sqrt{\lambda}\br{x,y}}
        ,\quad
        g\br{x,y}:=\sqrt{\lambda}G_{\sqrt{\lambda}\br{x,y}}
        ,\quad
        K\br{x,y,t}:=\widetilde{\kappa}\br{\sqrt{\lambda}\br{x,y,t}}.
    \end{equation*}
    Moreover, since
    \begin{equation*}
        \nrm{F_{p,q}}_{\ell^2\br{p,q}}\eqsim
        \nrm{f}_{L^2},\quad
        \nrm{G_{p,q}}_{\ell^2\br{p,q}}\eqsim
        \nrm{g}_{L^2},
    \end{equation*}
    it suffices to show:
    \begin{equation*}
        I_\ast\br{f,g}\lesssim \lambda^{-19\epsilon}\nrm{f}_{L^2}\nrm{g}_{L^2}.
    \end{equation*}
    On the one hand, using the trivial bound \(\kappa\br{x,y,t}\lesssim 1\) we deduce
    \begin{equation*}
        I_\ast\br{f,g}\lesssim
        \int_{\R^2}
            f\br{x,y}
            \int_{t\eqsim 1}
                g\br{x+t,y+t^2}
            dt
        dxdy,
    \end{equation*}
    which, after further applying Lemma \ref{thm_Lp_impro} and \(L^\frac{3}{2}\br{dxdy}\times L^3\br{dxdy}\) H\"{o}lder's inequality, gives
    \begin{equation}\label{eq_I_fg_L3/2}
        I_\ast\br{f,g}\lesssim
        \nrm{f}_{L^{\frac{3}{2}}}\nrm{g}_{L^{\frac{3}{2}}}.
    \end{equation}
    On the other hand, we may bound \(f,g\) trivially by their \(L^\infty\) norms:
    \begin{equation*}
        I_\ast\br{f,g}\lesssim \nrm{K\br{x,y,t}}_{L^1\br{dxdydt,\abs{x},\abs{y}\lesssim 1\eqsim t}}
        \nrm{f}_{L^\infty}\nrm{g}_{L^\infty}.
    \end{equation*}
    Notice that the condition given by \eqref{eq_f_a_E_cond} allows us to conclude:
    \begin{equation*}
        \abs{\widetilde{w}\br{\sqrt{\lambda}\br{x,y}}}\vee\abs{\widetilde{v}\br{\sqrt{\lambda}\br{x,y}}}
        \vee\abs{\widetilde{u}\br{\sqrt{\lambda}\br{x,y}}}\eqsim \sqrt{\lambda}.
    \end{equation*}
    Applying Lemma \ref{lem_tf_cor_lev_set_ana}, we obtain:
    \begin{equation}\label{eq_I_fg_Linfty}
        I_\ast\br{f,g}\lesssim \lambda^{-\frac{\delta}{2}}\nrm{f}_{L^\infty}\nrm{g}_{L^\infty}.
    \end{equation}
    Finally, by real interpolation\footnote{For a simple argument that removes any use of real interpolation please see the Final Remarks, Section \ref{StrongL2}.} between \eqref{eq_I_fg_L3/2} and \eqref{eq_I_fg_Linfty}, we obtain
    \begin{equation*}
        \Lambda_\ast\br{F,G}\lesssim I_\ast\br{f,g}\lesssim \lambda^{-\frac{\delta}{8}}\nrm{f}_{L^2}\nrm{g}_{L^2}\eqsim \lambda^{-\frac{\delta}{8}}\nrm{F_{p,q}}_{\ell^2\br{p,q}}\nrm{G_{p,q}}_{\ell^2\br{p,q}}\,,
    \end{equation*}
which gives us the desired conclusion in \eqref{stat} with \(\epsilon=\frac{\delta}{152}\).
\end{proof}

\subsection{Level set analysis}\label{LevSet}

In this section we provide the proof of \textbf{Lemma \ref{lem_tf_cor_lev_set_ana}}. This is subdivided in five subsections, as follows:

\subsubsection{Reduction to the critical case: \(\abs{u}\vee\abs{v}\eqsim \lambda\)}
For simplicity, we denote
\begin{equation*}
    K\br{x,y,t}:=\ang{
        w\br{x+t,y+t^2}t^2+
        v\br{x,y}t+
        u\br{x,y}
    }^{-N},
\end{equation*}
and we set\footnote{The implicit constant in the definition of \(\I,\A\) may change from line to line throughout the argument due to the change of variables that we will employ in several key steps.}
\begin{equation*}
    \I:=\set{
        z\in\R
    }{
        \abs{z}\lesssim 1
    }
    ,\quad
    \A:=\set{
        s\in\R
    }{
        s\eqsim 1
    }.
\end{equation*}
Our aim  is to show
\begin{equation*}
    \nrm{K\br{x,y,t}}_{L^1\br{dxdydt,\I^2\times\A}}\lesssim \lambda^{-\delta}.
\end{equation*}
Let \(\delta\in\br{0,1}\) be a constant to be chosen later and consider the following set:
\begin{equation*}
    \mathcal{E}:=\set{\br{x,y}\in \I^2}{\nrm{K\br{x,y,t}}_{L^1\br{dt,\A}}\lesssim \lambda^{-\delta}}.
\end{equation*}
By construction, we have the trivial bound \(\nrm{K}_{L^1\br{dxdydt,\mathcal{E}\times\A}}\lesssim \lambda^{-\delta}\).
As for \(\br{x,y}\in\I^2\setminus\mathcal{E}\), there exists \(t\in\A\) such that
\begin{equation*}
    \abs{w\br{x+t,y+t^2}t^2+
        v\br{x,y}t+
        u\br{x,y}}
        \lesssim \lambda^{\frac{\delta}{N}}\ll \lambda.
\end{equation*}
Thus, either
\begin{equation*}
    \lambda\eqsim \abs{w\br{x+t,y+t^2}t^2}\eqsim
    \abs{v\br{x,y}t+u\br{x,y}}\lesssim \abs{u\br{x,y}}\vee\abs{v\br{x,y}},
\end{equation*}
or the alternative
\begin{equation*}
    \lambda \gg \abs{w\br{x+t,y+t^2}t^2}\vee
    \abs{v\br{x,y}t+u\br{x,y}}\implies \abs{u\br{x,y}}\eqsim\abs{v\br{x,y}}\eqsim \lambda.
\end{equation*}
As a result, from now on we can assume without loss of generality that \(\abs{u}\vee\abs{v}\eqsim \lambda\) on \(\I^2\).

\subsubsection{Decoupling of smooth variables from rough variables}
We start with the following simple observation: modulo constants the $L^1$ norm is invariant under the parabolic shift\footnote{When the context is clear, we drop the domain \(\I^2\times \A\) dependency in the \(L^p\) norm.}
\begin{equation}\label{eq_K_less_K_shift_123}
    \nrm{K}_{L^1}\leq
    \nrm{K\br{x-t,y-t^2,t}}_{L^1\br{dxdydt}}.
\end{equation}
We then apply Cauchy-Schwarz inequality on the square of the above quantity
\begin{align*}
   \eqref{eq_K_less_K_shift_123}^2
    \lesssim
    \nrm{
        \nrm{
            K\br{x-t,y-t^2,t}
        }_{L^1\br{dt}}
    }_{L^2\br{dxdy}}^2
    =
    2
    \nrm{
        \prod_{j=0,1}K\br{x-t_j,y-t_j^2,t_j}
    }_{L^1\br{dxdydt_0dt_1,t_0<t_1}}
\end{align*}
and undo the shift to deduce
\begin{equation*}
    \nrm{K}_{L^1}^2\lesssim
    \nrm{
        K\br{x,y,t_+}K\br{x+t_+-t_-,y+t_+^2-t_-^2,t_-}
    }_{L^1\br{dxdydt_+dt_-,t_-<t_+}}.
\end{equation*}
We now notice that the expression
\begin{equation*}
    K\br{x+t_+-t_-,y+t^2_+-t^2_-,t_-}
    =\ang{
        w\br{x+t_+,y+t_+^2}t_-^2+
        \cdots
    }^{-N}.
\end{equation*}
shares the term \(w\br{x+t_+,y+t_+^2}\) with \(K\br{x,y,t_+}\).
This creates a bridge between measurable functions with \(\br{x,y}\) as input and those with \(\br{x+t_+-t_-,y+t^2_+-t^2_-}\) as input. For convenience, we set
\begin{equation*}
    X:=x+t_+-t_-,\quad
    Y:=y+t_+^2-t_-^2.
\end{equation*}
At this point, we are invited to apply Lemma \ref{lem_jap_mul_diff} in order to deduce
\begin{equation}\label{eq_KK_les_t2utv}
    K\br{x,y,t_+}K\br{X,Y,t_-}
    \lesssim
    \ang{
        \det
        \begin{pmatrix}
            t_+^2 & u\br{x,y}+t_+v\br{x,y}\\
            t_-^2 & u\br{X,Y}+t_-v\br{X,Y}
        \end{pmatrix}
    }^{-N}
    =\ang{
        \bm{t}^2\vec{1}\wedge
        \br{\vec{u}+\bm{t}\vec{v}}
    }^{-N},
\end{equation}
where we introduce the alternating tensor product \(\wedge\) and the matrix/vector notations:
\begin{equation*}
    \bm{t}:=
    \begin{pmatrix}
        t_+ & 0\\
        0 & t_-
    \end{pmatrix}
    ,\quad
    \vec{1}:=
    \begin{pmatrix}
        1\\
        1
    \end{pmatrix}
    ,\quad
    \vec{u}:=
    \begin{pmatrix}
        u\br{x,y}\\
        u\br{X,Y}
    \end{pmatrix}
    ,\quad
    \vec{v}:=
    \begin{pmatrix}
        v\br{x,y}\\
        v\br{X,Y}
    \end{pmatrix}
\end{equation*}
to express the determinant in a more succinct form. Bounding \eqref{eq_KK_les_t2utv} trivially by \(1\) when \(0<t_+-t_-\lesssim \lambda^{-2\delta}\) gives the following estimate:
\begin{equation*}
    \nrm{K}^2_{L^1}
    \lesssim \lambda^{-2\delta}+
    \nrm{
        \ang{
            \bm{t}^2\vec{1}\wedge
            \br{\vec{u}+\bm{t}\vec{v}}
        }^{-N}
    }_{L^1\br{dxdydt_+dt_-,\lambda^{-2\delta} \lesssim t_+-t_-\lesssim 1 }}.
\end{equation*}
Thus, to prove Lemma \ref{lem_tf_cor_lev_set_ana}, it suffices to show
\begin{equation}\label{eq_K_2_bdd}
    \nrm{
        \ang{
            \bm{t}^2\vec{1}\wedge
            \br{\vec{u}+\bm{t}\vec{v}}
        }^{-N}
    }_{L^1\br{dxdydt_+dt_-,\lambda^{-2\delta} \lesssim t_+-t_-\lesssim 1 }}\lesssim \lambda^{-2\delta}.
\end{equation}
To proceed, we observe that
\begin{equation}\nonumber
    t_\pm=\frac{Y-y}{2\br{X-x}}\pm\frac{X-x}{2},\quad dt_+dt_-=\frac{dXdY}{2\br{X-x}}.
\end{equation}
As a result, by setting
\(\bm{j}:=
    \begin{pmatrix}
        1 & 0\\
        0 & -1
    \end{pmatrix}\), we can write\footnote{For the simplicity of the exposition we allow the following notational abuse: in what follows if $\bm{A}$ is a matrix and $a\in\R$ then $\bm{A}+a:=\bm{A}+a \bm{I}$ where $\bm{I}$ stands for the identity matrix.} \(\bm{t}=\frac{Y-y}{2\br{X-x}}+\bm{j}\frac{X-x}{2}\)
and obtain
\begin{equation*}
    \ang{
        \bm{t}^2\vec{1}\wedge
        \br{\vec{u}+\bm{t}\vec{v}}
    }^{-N}
    =\ang{
        \br{
            \frac{Y-y}{2\br{X-x}}+
            \bm{j}\frac{X-x}{2}
        }^2\vec{1}
        \wedge
        \br{
            \vec{u}
            +
            \br{
                \frac{Y-y}{2\br{X-x}}+
                \bm{j}\frac{X-x}{2}
            }
            \vec{v}
        }
    }^{-N}.
\end{equation*}
Let \(K_2\br{x,y,X,Y}\) be the right-hand-side expression of the above equality; it then remains to show
\begin{equation*}
    \nrm{K_2\br{x,y,X,Y}}_{
        L^1\br{
            \frac{dxdX}{X-x}dydY,
            \lambda^{-2\delta} \lesssim X-x\lesssim 1
        }
    }
    \lesssim \lambda^{-2\delta}.
\end{equation*}
Once here, we start by eliminating the \(u\br{x,y}\) and \(v\br{x,y}\) terms in the \(K_2\br{x,y,X,Y}\) expression.
This is achieved by cubing the expression, performing a \(L^{\frac{3}{2}}\times L^3\) H\"{o}lder's inequality to triple the \(Y\) variable and then applying Lemma \ref{lem_jap_mul_diff}. Indeed, we first have
\begin{align*}
    &\nrm{K_2\br{x,y,X,Y}}_{
        L^1\br{
            \frac{dxdX}{X-x}dydY,
            \lambda^{-2\delta} \lesssim X-x\lesssim 1
        }
    }^3\\
    \lesssim &
    \nrm{
        \frac{
            \1_{\I^3}\br{x,y,X}
        }{
            X-x
        }
    }^3_{
        L^{\frac{3}{2}}
        \br{
            dxdXdy,
            \lambda^{-2\delta}\lesssim X-x\lesssim 1
        }
    }
    \cdot
    \nrm{
        \nrm{
            K_2\br{
                x,y,X,Y
            }
        }_{L^1\br{dY}}
    }^3_{L^3
        \br{
            dxdXdy
        }
    }\\
   \lesssim & \lambda^{2\delta}\cdot \nrm{
        \prod_{j=1}^3
            K_2\br{x,y,X,Y_j}
    }_{L^1\br{
        dxdXdydY_1dY_2dY_3
    }}.
\end{align*}
To implement Lemma \ref{lem_jap_mul_diff} on the inner expression \(\prod_{j=1}^3K_2\br{x,y,X,Y_j}\), we shall first describe the choice of the \(O\br{1}\) bounded coefficients we choose for our purpose.
Observe that
\begin{equation*}
    K_2\br{x,y,X,Y_j}=\ang{
        \cdots
        -
        \br{
            \frac{Y_j-y}{2\br{X-x}}
            -\frac{X-x}{2}
        }^2
        \br{
            u\br{x,y}
            +
            \br{
                \frac{Y_j-y}{2\br{X-x}}
                +\frac{X-x}{2}
            }
            v\br{x,y}
        }
    }^{-N}.
\end{equation*}
There exists a natural choice of coefficients $P_j$ given by
\begin{equation}\nonumber
    \begin{pmatrix}
        P_1\\
        P_2\\
        P_3
    \end{pmatrix}
    :=
    \br{
        \frac{\bm{Y}-y}{2\br{X-x}}
        -\frac{X-x}{2}
    }^2
    \vec{1}
    \times
    \br{
        \frac{\bm{Y}-y}{2\br{X-x}}
        -\frac{X-x}{2}
    }^2
    \br{
        \frac{\bm{Y}-y}{2\br{X-x}}
        +\frac{X-x}{2}
    }
    \vec{1},
\end{equation}
where \(\times\) stands for the cross product in \(\R^3\) and \(\bm{Y}:=
    \begin{pmatrix}
        Y_1 & 0 & 0\\
        0 & Y_2 & 0\\
        0 & 0 & Y_3
    \end{pmatrix}\),
    \(\vec{1}:=
    \begin{pmatrix}
        1\\
        1\\
        1
    \end{pmatrix}\).
By construction,
\begin{equation}\nonumber
    \sum_{j=1}^3
        P_j\cdot
        \br{
            \frac{Y_j-y}{2\br{X-x}}
            -\frac{X-x}{2}
        }^2
        \br{
            u\br{x,y}
            +
            \br{
                \frac{Y_j-y}{2\br{X-x}}
                +\frac{X-x}{2}
            }
            v\br{x,y}
        }
    =0.
\end{equation}
Therefore, applying Lemma \ref{lem_jap_mul_diff} for the choice above, we obtain:
\begin{align*}
    &
    \prod_{j=1}^3
        K_2\br{x,y,X,Y_j}\\
    \lesssim &
    \ang{
        \sum_{j=1}^3
            P_j \cdot
            \br{
                \frac{Y_j-y}{2\br{X-x}}
                +\frac{X-x}{2}
            }^2
            \br{
                u\br{X,Y_j}
                +
                \br{
                    \frac{Y_j-y}{2\br{X-x}}
                    -\frac{X-x}{2}
                }
                v\br{X,Y_j}
            }
    }^{-N}\nonumber\\
    = &
    \Bigg<
        \br{
            \frac{\bm{Y}-y}{2\br{X-x}}+\frac{X-x}{2}
        }^2
        \br{
            \vec{U}
            +
            \br{
                \frac{\bm{Y}-y}{2\br{X-x}}-\frac{X-x}{2}
            }
            \vec{V}
        }\\
        &\wedge
        \br{
            \frac{\bm{Y}-y}{2\br{X-x}}-\frac{X-x}{2}
        }^2
        \vec{1}
        \wedge
        \br{
            \frac{\bm{Y}-y}{2\br{X-x}}-\frac{X-x}{2}
        }^2
        \br{
            \frac{\bm{Y}-y}{2\br{X-x}}+\frac{X-x}{2}
        }\vec{1}
    \Bigg>^{-N}\\
    =&: K_6\br{x,y,X,\bm{Y}},
\end{align*}
where the two vectors
\begin{equation*}
    \vec{U}=\vec{U}\br{X,\bm{Y}}:=
    \begin{pmatrix}
        u\br{X,Y_1}\\
        u\br{X,Y_2}\\
        u\br{X,Y_3}
    \end{pmatrix}
    ,\quad
    \vec{V}=\vec{V}\br{X,\bm{Y}}:=
    \begin{pmatrix}
        v\br{X,Y_1}\\
        v\br{X,Y_2}\\
        v\br{X,Y_3}
    \end{pmatrix}
\end{equation*}
are introduced.
As a brief summary, showing \eqref{eq_K_2_bdd}
reduces now to showing:
\begin{equation*}
    \nrm{K_6\br{x,y,X,\bm{Y}}}_{L^1\br{
        dxdXdydY_1dY_2dY_3
    }}
    \lesssim \lambda^{-8\delta}.
\end{equation*}
To further simplify the expression, we perform another change of variables with the corresponding Jacobian \(\eqsim 1\) to simplify the expression. Namely, we consider:
\begin{align*}
    K_6\br{X-x,y-x^2,X,y+2Y}
    =&
    \ang{
        \br{
            \frac{\bm{Y}}{x}+x
        }^2
        \br{
            \vec{\mathcal{U}}
            +
            \frac{\bm{Y}}{x}
            \vec{\mathcal{V}}
        }
        \wedge
        \br{
            \frac{\bm{Y}}{x}
        }^2
        \vec{1}
        \wedge
        \br{
            \frac{\bm{Y}}{x}
        }^2
        \br{
            \frac{\bm{Y}}{x}+x
        }\vec{1}
    }^{-N}\\
    = &
    \ang{
        \br{
            \frac{\bm{Y}}{x}+x
        }^2
        \br{
            \vec{\mathcal{U}}
            +
            \frac{\bm{Y}}{x}
            \vec{\mathcal{V}}
        }
        \wedge
        \br{
            \frac{\bm{Y}}{x}
        }^2
        \vec{1}
        \wedge
        \br{
            \frac{\bm{Y}}{x}
        }^3\vec{1}
    }^{-N}\\
    \lesssim &
    \ang{
        \br{
            \bm{Y}+x^2
        }^2
        \br{
            x
            \vec{\mathcal{U}}
            +
            \bm{Y}
            \vec{\mathcal{V}}
        }
        \wedge
        \bm{Y}^2
        \vec{1}
        \wedge
        \bm{Y}^3
        \vec{1}
    }^{-N}=:\mathcal{K}_6\br{x,y,X,\bm{Y}},
\end{align*}
where we modify the two vectors
\(\vec{\mathcal{U}}=\vec{U}\br{X,y+2\bm{Y}}\) and \(\vec{\mathcal{V}}=\vec{V}\br{X,y+2\bm{Y}}\)
accordingly. With these modifications, we successfully create the \emph{smooth} variables \(x\). Indeed, the expression
\begin{equation}\label{polyn}
    \br{
        \bm{Y}+x^2
    }^2
    \br{
        x
        \vec{\mathcal{U}}
        +
        \bm{Y}
        \vec{\mathcal{V}}
    }
    \wedge
        \bm{Y}^2
        \vec{1}
    \wedge
        \bm{Y}^3
        \vec{1}
    =\sum_{j=0}^5
    C_j\br{X,y,\bm{Y}}x^j
\end{equation}
is now a polynomial in \(x\) with its degree not exceeding \(5\).
Thus, our goal becomes now the following estimate:
\begin{equation}\label{eq_det_smooth_-8bd}
    \nrm{
        \ang{
            \sum_{j=0}^5
            C_j\br{X,y,\bm{Y}}x^j
        }^{-N}
    }_{L^1\br{dxdXdydY_1dY_2dY_3,\I^6}}
    \lesssim \lambda^{-8\delta}.
\end{equation}

\subsubsection{A Van der Corput type estimate}
To prove \eqref{eq_det_smooth_-8bd}, we rely on the following
\begin{lemma}\label{lem_van_der_cor_type}
     Let \(n,d\in\N\). Then there exists a constant\footnote{The exponent  $\delta\br{d}$ may be taken $\frac{1}{2d+2}$.} \(\delta\br{d}\in\bR{0,1}\) such that for any real-valued multivariate polynomial \(P\) on \(\R^n\) given by
    \begin{equation*}
        P\br{x}:=\sum_{\alpha:
        \abs{\alpha}\leq d}c_\alpha x^\alpha,
    \end{equation*}
    we have the following Van der Corput type estimates:
    \begin{equation*}
        \nrm{
            \ang{P\br{x}}^{-1}
        }_{L^1\br{dx,\I^n}}
        \underset{n,d}{\lesssim} \ang{\nrm{c_\alpha}_{\ell^1\br{\alpha}}}^{-\delta\br{d}}.
    \end{equation*}
\end{lemma}
This can be proven easily by applying Lemma \ref{lem_jap_mul_diff} iteratively. Alternatively, observe that Lemma \ref{lem_van_der_cor_type} is the direct consequence of the Van der Corput type sub-level set estimate appearing in Proposition 2.2. in \cite{sw} were it not for the constant term \(c_0\). Below, we provide an argument to address this latter issue.
\begin{proof}
     We consider the following two cases\footnote{Throughout the proof, all the implicit constants are allowed to depend on \(n,d\).}:
    \begin{itemize}
        \item \(\abs{c_0}\gg \sum_{\alpha\neq 0}\abs{c_\alpha}\).
        Triangle inequality gives the following trivial bound:
        \begin{equation}\nonumber
            \inf_{x\in\I^n}\abs{P\br{x}}\gtrsim \abs{c_0}-\sum_{\alpha\neq 0}\abs{c_\alpha}\eqsim\nrm{c_\alpha}_{\ell^1\br{\alpha}}.
        \end{equation}
        As a direct consequence,
        \begin{equation*}
            \nrm{\ang{P\br{x}}^{-1}}_{L^1\br{dx,\I^n}}\lesssim
            \nrm{\ang{P\br{x}}^{-1}}_{L^\infty\br{dx,\I^n}}
            \lesssim \ang{\nrm{c_\alpha}_{\ell^1\br{\alpha}}}^{-1}.
        \end{equation*}

        \item \(\abs{c_0}\lesssim\sum_{\alpha\neq 0}\abs{c_\alpha}\). In this situation we first eliminate the constant term via Lemma \ref{lem_jap_mul_diff} and then  apply Proposition 2.2. in \cite{sw}:
        \begin{equation}\nonumber
            \nrm{
                \ang{P\br{x}}^{-1}
            }_{L^1\br{dx,\I^n}}^2\lesssim
            \nrm{
                \ang{Q\br{x,y}}^{-1}
            }_{L^1\br{dxdy,\I^{2n}}},\quad
            Q\br{x,y}:=P\br{x}-P\br{y}.
        \end{equation}
        The polynomial \(Q\br{x,y}\) has the same degrees as \(P\br{x}\) and no constant term. Moreover, the coefficients of \(Q\br{x,y}\) relates to those of \(P\br{x}\) in the following trivial way:
        \begin{equation}\nonumber
            Q\br{x,y}=\sum_{\alpha\neq 0}c_\alpha x^\alpha-c_\alpha y^\alpha=:\sum_{\beta}C_\beta\br{x,y}^\beta,\quad C_0=0.
        \end{equation}
        We can now apply Proposition 2.2. in \cite{sw} with \(\epsilon:=\nrm{C_\beta}_{\ell^1\br{\beta}}^{\frac{1}{d+1}}\). This gives the following estimate:
        \begin{equation}\nonumber
            \abs{
                \I^{2n}\cap \abs{Q}^{-1}\bR{0,\epsilon}
            }
            \lesssim
            \br{
                \frac{\epsilon}{\nrm{C_\beta}_{\ell^1\br{\beta}}}
            }^{\frac{1}{d}}
            =
            \nrm{
            C_\beta
            }_{
                \ell^1\br{\beta}
            }^{-\frac{1}{d+1}}
        \end{equation}
        and, as a direct consequence:
        \begin{equation}\label{vc1}
            \nrm{
                \ang{Q\br{x,y}}^{-1}
            }_{L^1\br{dxdy,\I^{2n}}}
            \lesssim
            \abs{
                \I^{2n}\cap \abs{Q}^{-1}\bR{0,\epsilon}
            }
            +
            \ang{\epsilon}^{-1}
            \lesssim
            \ang{
                \nrm{
                C_\beta
                }_{
                    \ell^1\br{\beta}
                }
            }^{-\frac{1}{d+1}}.
        \end{equation}
        Lastly, using
        \begin{equation}\nonumber
            \nrm{
            C_\beta
            }_{
                \ell^1\br{\beta}
            }
            =2\sum_{\alpha\neq 0}\abs{c_\alpha}\eqsim \nrm{c_\alpha}_{\ell^1\br{\alpha}},
        \end{equation}
        we conclude from \eqref{vc1} that
        \begin{equation*}
            \nrm{
                \ang{P\br{x}}^{-1}
            }_{L^1\br{dx,\I^n}}
            =
            \nrm{
                \ang{Q\br{x,y}}^{-1}
            }_{L^1\br{dxdy,\I^{2n}}}^{\frac{1}{2}}
            \lesssim
            \ang{
                \nrm{
                    c_\alpha
                }_{\ell^1\br{\alpha}}
            }^{-\frac{1}{2d+2}}.
        \end{equation*}
    \end{itemize}
\end{proof}

\subsubsection{Control over the coefficients of the polynomial}
Recall the definition of \(\mathcal{K}_6\br{x,y,X,\bm{Y}}\). By the multilinearity of the alternating tensor product, we may rewrite the inner expression of \(\mathcal{K}_6\br{x,y,X,\bm{Y}}\) in terms of matrix multiplications:
\begin{equation}\label{eq_wedges_mats_ini_123}
    \br{
        \bm{Y}^2+\bm{Y}\cdot 2x^2+x^4
    }
    \br{
        x
        \vec{\mathcal{U}}
        +
        \bm{Y}
        \vec{\mathcal{V}}
    }
    \wedge
        \bm{Y}^2
        \vec{1}
    \wedge
        \bm{Y}^3
        \vec{1}
   =
    \begin{pmatrix}
        1 & x^2 & x^4
    \end{pmatrix}
    \bm{M}_{\bm{Y}}
    \br{
        x
        \vec{\mathcal{U}}
        +
        \bm{Y}
        \vec{\mathcal{V}}
    },
\end{equation}
where \(\bm{M}_{\bm{Y}}\) is a \(3\times 3\) matrix consisting of \(O\br{1}\) bounded polynomial entries \(M_{i,j}\br{Y_1,Y_2,Y_3}\in\R\Br{Y_1,Y_2,Y_3}\). The above expression allows us to relate the coefficients in \eqref{polyn} with the vectors \(\vec{\mathcal{U}}\) and \(\vec{\mathcal{V}}\):
\begin{equation*}
    \vec{C}\br{X,y,\bm{Y}}:=
    \begin{pmatrix}
       C_1\br{X,y,\bm{Y}}\\
       C_3\br{X,y,\bm{Y}}\\
       C_5\br{X,y,\bm{Y}}\\
       C_0\br{X,y,\bm{Y}}\\
       C_2\br{X,y,\bm{Y}}\\
       C_4\br{X,y,\bm{Y}}
    \end{pmatrix}
    =
    \begin{pmatrix}
        \bm{M}_{\bm{Y}} & \bm{O}_{3\times 3}\\
        \bm{O}_{3\times 3} & \bm{M}_{\bm{Y}}\bm{Y}
    \end{pmatrix}
    \begin{pmatrix}
        \vec{\mathcal{U}}\\
        \vec{\mathcal{V}}
    \end{pmatrix}
    .
\end{equation*}
Now, in order to utilize \textbf{Lemma \ref{lem_van_der_cor_type}},
it suffices to derive lower bound on \(\nrm{\vec{C}\br{X,y,\bm{Y}}}\) since
\begin{equation}\nonumber
    \nrm{\vec{C}\br{X,y,\bm{Y}}}\eqsim
    \nrm{C_j\br{X,y,\bm{Y}}}_{\ell^1\br{j}}.
\end{equation}
This is achieved by showing that \(\bm{M}_{\bm{Y}}\) is non-singular ``for most" of the values of \(\br{Y_1,Y_2,Y_3}\in \I^3\), and thus the order of magnitude of the vector \(\begin{pmatrix}
        \vec{\mathcal{U}}\\
        \vec{\mathcal{V}}
    \end{pmatrix}\) is preserved under the linear transformation
    \(
    \begin{pmatrix}
        \bm{M}_{\bm{Y}} & \bm{O}_{3\times 3}\\
        \bm{O}_{3\times 3} & \bm{M}_{\bm{Y}}\bm{Y}
    \end{pmatrix}
    \).
To understand the structure of the matrix \(\bm{M}_{\bm{Y}}\), we observe that \eqref{eq_wedges_mats_ini_123} implies the following relation:
\begin{equation*}
    \br{
        \bm{Y}^2+\bm{Y}\cdot 2x^2+x^4
    }
    \vec{v}
    \wedge
        \bm{Y}^2
        \vec{1}
    \wedge
        \bm{Y}^3
        \vec{1}
   =
    \begin{pmatrix}
        1 & x^2 & x^4
    \end{pmatrix}
    \bm{M}_{\bm{Y}}
    \vec{v},\quad \vec{v}\in\R^3.
\end{equation*}
By taking \(\vec{v}=\vec{1},\;\bm{Y}\vec{1},\;\bm{Y}^2\vec{1}\), we can calculate explicitly the following matrix
\begin{equation*}
    \bm{M}_{\bm{Y}}
    \cdot
    \begin{pmatrix}
        \vec{1} & \bm{Y}\vec{1} & \bm{Y}^2\vec{1}
    \end{pmatrix}
    =
    \begin{pmatrix}
        0 & 0 & \bm{Y}^{4,2,3}\\
        2\bm{Y}^{1,2,3} & 0 & 0\\
        \bm{Y}^{0,2,3} & \bm{Y}^{1,2,3} & 0
    \end{pmatrix},
\end{equation*}
where we introduce the notation
\begin{equation*}
    \bm{Y}^{a,b,c}:=\bm{Y}^a\vec{1}\wedge \bm{Y}^b\vec{1}\wedge\bm{Y}^c\vec{1}=
    \det\begin{pmatrix}
        Y_1^a & Y_1^b & Y_1^c\\
        Y_2^a & Y_2^b & Y_2^c\\
        Y_3^a & Y_3^b & Y_3^c
    \end{pmatrix}
    .
\end{equation*}
Taking the determinant on both sides
\begin{equation*}
    \det \bm{M}_{\bm{Y}}\cdot \bm{Y}^{0,1,2}=2\bm{Y}^{4,2,3}\br{\bm{Y}^{1,2,3}}^2,
\end{equation*}
we conclude that:
\begin{equation*}
    \det \bm{M}_{\bm{Y}}=2\det\bm{Y}^4\cdot \br{\bm{Y}^{0,1,2}}^2=
    2\br{Y_1Y_2Y_3}^4 \br{Y_1-Y_2}^2\br{Y_1-Y_3}^2\br{Y_2-Y_3}^2,
\end{equation*}
which is a nontrivial polynomial expression. As a direct consequence, we can calculate the product of all the singular values of the large matrix \(\begin{pmatrix}
        \bm{M}_{\bm{Y}} & \bm{O}_{3\times 3}\\
        \bm{O}_{3\times 3} & \bm{M}_{\bm{Y}}\bm{Y}
    \end{pmatrix}\):
\begin{equation*}
    \prod_{k=1}^6\mu_k\br{\bm{Y}}=
    \abs{
        \det\bm{M}_{\bm{Y}}^2\cdot \det\bm{Y}
    }
    =
    4
    \abs{Y_1 Y_2 Y_3}^9
    \br{Y_1-Y_2}^4\br{Y_1-Y_3}^4\br{Y_2-Y_3}^4.
\end{equation*}
On the one hand, we have trivial bounds on all singular values:
\begin{equation*}
    \max_k\nrm{\mu_k\br{\bm{Y}}}_{L^\infty\br{\I^3}}
    \lesssim
    \max_{i,j}\nrm{M_{i,j}\br{Y_1,Y_2,Y_3}}_{L^\infty\br{\I^3}}
    \lesssim 1.
\end{equation*}
On the other hand, if we have a lower bound on their product:
\begin{equation*}
    \prod_{k=1}^6\mu_k\br{\bm{Y}}=
    \abs{
        \det\bm{M}_{\bm{Y}}^2\cdot \det\bm{Y}
    }\gtrsim\lambda^{-\epsilon},
\end{equation*}
we can deduce directly that
\begin{equation*}
    \lambda^{-\epsilon}
    \lesssim
        \min_k\mu_k\br{\bm{Y}}
    \lesssim
        \max_k\mu_k\br{\bm{Y}}
    \lesssim 1.
\end{equation*}
As a direct consequence, we have the operator norm estimate:
\begin{equation*}
    \nrm{\begin{pmatrix}
        \bm{M}_{\bm{Y}} & \bm{O}_{3\times 3}\\
        \bm{O}_{3\times 3} & \bm{M}_{\bm{Y}}\bm{Y}
    \end{pmatrix}^{-1}}\lesssim \lambda^\epsilon.
\end{equation*}
This transfers to a lower bound on the norm of the coefficient vector:
\begin{equation*}
    \lambda
    \eqsim
    \nrm{
        \begin{pmatrix}
            \vec{\mathcal{U}}\\
            \vec{\mathcal{V}}
        \end{pmatrix}
    }
    \leq
     \nrm{\begin{pmatrix}
        \bm{M}_{\bm{Y}} & \bm{O}_{3\times 3}\\
        \bm{O}_{3\times 3} & \bm{M}_{\bm{Y}}\bm{Y}
    \end{pmatrix}^{-1}}
    \cdot
    \nrm{\vec{C}\br{X,y,\bm{Y}}}\lesssim
    \lambda^\epsilon
    \nrm{C_j\br{X,y,\bm{Y}}}_{\ell^1\br{j}}
    .
\end{equation*}
We can thus conclude:
\begin{equation}\label{eq_detM_R_dich}
    \lambda^{-\epsilon}\lesssim
    \abs{
        \det\bm{M}_{\bm{Y}}^2\cdot \det\bm{Y}
    }
    \implies
    \lambda^{1-\epsilon}\lesssim
    \nrm{C_j\br{X,y,\bm{Y}}}_{\ell^1\br{j}}
    .
\end{equation}

\subsubsection{Concluding our lemma: the proof of \eqref{eq_det_smooth_-8bd}}
Let \(\epsilon>0\) be a constant to be chosen later and consider the set
\begin{equation*}
    \mathcal{S}_3:=\set{\br{Y_1,Y_2,Y_3}\in\I^3}{\lambda^\epsilon
        \det\bm{M}_{\bm{Y}}^2\cdot \det\bm{Y}
    \in \I}.
\end{equation*}
Naturally, the analysis of the left-hand-side of \eqref{eq_det_smooth_-8bd} appeals to Lemma \ref{lem_van_der_cor_type} using the decomposition
\begin{align}%
    &
    \nrm{
        \ang{
            \sum_{j=0}^5
            C_j\br{X,y,\bm{Y}}x^j
        }^{-N}
    }_{L^1\br{dxdXdydY_1dY_2dY_3,\I^6}}\nonumber\\
    \leq &
    \nrm{
        \ang{
            \sum_{j=0}^5
            C_j\br{X,y,\bm{Y}}x^j
        }^{-N}
    }_{L^1\br{dxdXdy,\I^3}\otimes L^1\br{dY_1dY_2dY_3,\mathcal{S}_3}}\label{eq_R_est_Sing}\\
    + &
    \nrm{
        \ang{
            \sum_{j=0}^5
            C_j\br{X,y,\bm{Y}}x^j
        }^{-N}
    }_{L^1\br{dxdXdy,\I^3}\otimes L^1\br{dY_1dY_2dY_3,\I^3\setminus\mathcal{S}_3}}\label{eq_R_est_Poly}.
\end{align}%

For \eqref{eq_R_est_Sing}, we use the trivial bound and dominate the sublevel set:
\begin{equation*}
    \eqref{eq_R_est_Sing}\leq \abs{\mathcal{S}_3}  \leq
    \nrm{
        \ang{
            \lambda^\epsilon
            \det\bm{M}_{\bm{Y}}^2\cdot \det\bm{Y}
        }^{-1}
    }_{L^1\br{dY_1dY_2dY_3,\I^3}}
    \lesssim \lambda^{ -\epsilon \delta\br{39} }.
\end{equation*}
As for \eqref{eq_R_est_Poly}, we use the fact that
\begin{equation*}
    \br{Y_1,Y_2,Y_3}\in \I^3\setminus \mathcal{S}_3\implies \lambda^{-\epsilon}\lesssim \abs{\det\bm{M}_{\bm{Y}}^2\cdot \det\bm{Y}},
\end{equation*}
the implication \eqref{eq_detM_R_dich}, and Lemma \ref{lem_van_der_cor_type} to obtain:
\begin{equation*}
    \eqref{eq_R_est_Poly}\leq
    \nrm{
        \nrm{
            \ang{
                \sum_{j=0}^5C_j
                    \br{X,y,\bm{Y}}x^j
            }^{-1}
        }_{L^1\br{dx,\I}}
    }_{L^\infty\br{dXdydY_1dY_2dY_3,\I\times \br{\I^3\setminus \mathcal{S}_3}}}\hspace{-5em}
    \lesssim
    \lambda^{\br{\epsilon-1} \delta\br{5}}.
\end{equation*}
Optimizing the parameter gives \(\epsilon=\frac{\delta\br{5}}{\delta\br{5}+\delta\br{39}}\) and that
\begin{equation*}
    \nrm{
        \ang{
            \sum_{j=0}^5C_j\br{X,y,\bm{Y}}x^j
        }^{-N}
    }_{L^1\br{dxdydXdY_1dY_2dY_3,\I^6}}
    \lesssim \lambda^{-\frac{\delta\br{5}\delta\br{39}}{\delta\br{5}+\delta\br{39}}}
    =:\lambda^{-8\delta}.
\end{equation*}
Finally, we conclude:
\begin{equation*}
    \nrm{K}_{L^1}\lesssim\lambda^{-\frac{\delta\br{5}\delta\br{39}}{8\delta\br{5}+8\delta\br{39}}}.
\end{equation*}
This completes the proof of Lemma \ref{lem_tf_cor_lev_set_ana} with \(\delta=\frac{\delta\br{5}\delta\br{39}}{8\delta\br{5}+8\delta\br{39}}\).

\section{Final Remarks}

\subsection{Incidence geometry (I): $L^p$ improving estimates for the averaging operator along parabola}\label{Incimprovest}

Our goal in this section is to provide a simple, short proof of a weaker version of Lemma \ref{thm_Lp_impro}, that serves though as a good enough substitute for obtaining the full strength of our Main Theorem \ref{mainresult}.

We first notice that after a dualization argument, \eqref{Improv} is equivalent with
    \begin{equation}\label{Limp}
    \L_\p (f,g):=\int_{\R^2} \int_{[1,2]} f(x-t, y-t^2)\,g(x,y)\,dt\,dx\,dy\lesssim \nrm{f}_{L^{\frac{3}{2}}(\R^2)}\,\nrm{g}_{L^{\frac{3}{2}}(\R^2)}\,,
    \end{equation}
with $f$ and $g$ positive Lebesgue measurable functions.

Next, via an almost orthogonality reasoning followed by a density argument, \eqref{Limp} is essentially equivalent with showing that for any
$F,\,G\subseteq [1,2]$ measurable sets one has
   \begin{equation}\label{Limp1}
    \L_\p (\chi_F,\chi_G)\lesssim |F|^{\frac{2}{3}}\,|G|^{\frac{2}{3}}\,.
    \end{equation}
As mentioned earlier, in order to provide a simple and concise argument, we will only aim for a weaker version of \eqref{Limp1} given by
   \begin{equation}\label{Limp1simple}
    \L_\p (\chi_F,\chi_G)\lesssim |F|^{\frac{5}{8}}\,|G|^{\frac{5}{8}}\,.
    \end{equation}
Indeed, in view of \eqref{eq_I_fg_Linfty}, any bound of the form  $|F|^{\a}\,|G|^{\a}$ with $\a>\frac{1}{2}$ is enough for our purposes. The approach to \eqref{Limp1simple} appeals to some basic incidence geometry arguments. With some non-trivial effort that relies on the full strength of the Szemer\'edi-Trotter point-line incidence result, \cite{SzT83}, the reasonings below may be upgraded in order to achieve \eqref{Limp1}. \footnote{Indeed, Szemer\'edi-Trotter corresponds to the precise numerology in the RHS of \eqref{Limp1}.}

With these being said, we pass to the proof of \eqref{Limp1simple}. We first notice that via another density argument, wlog, we can assume that both $F$ and $G$ consist of a union of disjoint intervals of the form $I_{p,q}:=I_p\times I_q:=\left[\frac{p-1}{2^m}, \frac{p}{2^m}\right)\times \left[\frac{q-1}{2^m}, \frac{q}{2^m}\right)$ where here $m\in\N$ is a fixed large parameter and $p,\,q\in\{1,\ldots, 2^m\}$. That is
  \begin{equation}\label{FGsets}
  F:=\bigcup_{(p,q)\in\mathcal{A}_l} I_{p,q}\quad\textrm{and}\quad G:=\bigcup_{(p,q)\in\mathcal{B}_s} I_{p,q}\,,
  \end{equation}
with $l,s\in \N$, $0\leq l,\,s\leq 2m$, such that
  \begin{equation}\label{FGsets1}
  |F|=2^{-l},\quad \# \mathcal{A}_l=2^{2m-l}\qquad\textrm{and}\qquad |G|=2^{-s},\quad \# \mathcal{B}_s=2^{2m-s}\,.
  \end{equation}
With these notations we notice that
  \begin{equation}\label{est}
\L_\p (\chi_F,\chi_G)\approx 2^{-3m}\,\sum_{(p,r)\in \mathcal{A}_l}\: \sum_{q\sim 2^m} \mathds{1}_{\mathcal{B}_s}\left(p+q, r+\frac{q^2}{2^m}\right)=: 2^{-3m}\,\mathcal{I}_{l,s}\,.
  \end{equation}
Now, on the one hand, a trivial double counting argument in \eqref{est} gives
\begin{equation}\label{est1}
\L_\p (\chi_F,\chi_G) \lesssim  2^{-3m}\,\min\left\{(\#\mathcal{A}_l) (\#\mathcal{B}_s),\: 2^m\,(\#\mathcal{A}_l),\: 2^m\,(\#\mathcal{B}_s)\right\}
 \approx\min\left\{ 2^{m-l-s},\:2^{-l},\:2^{-s}\right\}\,.
\end{equation}
On the other hand, we can visualize $\mathcal{I}_{l,s}$ in \eqref{est} as representing the incidences between
\begin{itemize}
\item the set of discretized\footnote{By convention if $Q=(p,q)$ then $I_Q:=I_{p,q}$.} ``points" $\mathcal{R}_{\B_s}:=\{I_Q\}_{Q\in \mathcal{B}_s}$;
\medskip
\item the set of discretized ``parabolas"  $\p_{\mathcal{A}_l}:=\{\p_Q\}_{Q\in \mathcal{A}_l}$ where here
$$\p_Q=\bigcup_{q\sim 2^m} I_{Q+(q,\frac{q^2}{2^m})}\,.$$
\end{itemize}

For $Q\in \mathcal{R}_{\B_s}$ we set
\begin{equation}\label{inc1}
\mathfrak{n}_Q:=\#\{R\in \mathcal{R}_{\mathcal{A}_l} \,|\, \p_R\:\textrm{is incident to}\:Q\}\,.
\end{equation}
Using now \eqref{est1} and \eqref{inc1}, via Cauchy--Schwarz, we have
\begin{equation}\label{i1}
\mathcal{I}_{l,s}=\sum_{Q\in \mathcal{R}_{\B_s}} \mathfrak{n}_Q= \# \B_s\,+\, \sum_{{Q\in \mathcal{R}_{\B_s}}\atop{\mathfrak{n}_Q\geq 2}} \mathfrak{n}_Q
\lesssim \# \B_s\,+\, \left(\sum_{{Q\in \mathcal{R}_{\B_s}}\atop{\mathfrak{n}_Q\geq 2}} \mathfrak{n}_Q^2\right)^{\frac{1}{2}}\,(\# \B_s)^{\frac{1}{2}}\,.
\end{equation}
Next, for $Q\in \mathcal{R}_{\B_s}$ and $P,\,R\in \mathcal{R}_{\mathcal{A}_l}$, we set
$\mathfrak{n}_Q(P,R)=1$ if $\p_P$ and  $\p_R$ are incident to $Q$ and $\mathfrak{n}_Q(P,R)=0$ otherwise. It is then not hard to see that
\begin{equation}\label{i2}
\sum_{{Q\in \mathcal{R}_{\B_s}}\atop{\mathfrak{n}_Q\geq 2}} \mathfrak{n}_Q^2\leq
\sum_{P,R\in \mathcal{R}_{\mathcal{A}_l}} \sum_{Q\in \mathcal{R}_{\B_s}} \mathfrak{n}_Q(P,R)\,.
\end{equation}
Once at this point,  we appeal to the simple but key geometric observation
\begin{equation}\label{i3}
\sum_{Q\in \mathcal{R}_{\B_s}} \mathfrak{n}_Q(P,R)\lesssim \min\left\{2^{2m-s},\,\frac{2^m}{|P-R|+1}\right\}\,.
\end{equation}
Combining \eqref{i1}--\eqref{i3} and using the symmetry in the sets $\mathcal{A}_l$ and $\mathcal{B}_s$ we conclude that
\begin{equation}\label{eq_key}
\mathcal{I}_{l,s}\lesssim \min\left\{(\#\mathcal{A}_l)(\#\mathcal{B}_s),\: 2^{\frac{m}{2}}\,(\#\mathcal{A}_l)^{\frac{3}{4}}\,(\#\mathcal{B}_s)^{\frac{1}{2}},\: 2^{\frac{m}{2}}\,(\#\mathcal{B}_s)^{\frac{3}{4}}\,(\#\mathcal{A}_l)^{\frac{1}{2}}\right\}\,.
\end{equation}
Putting together \eqref{est} and \eqref{eq_key}, via a geometric mean argument, we conclude
  \begin{equation}\label{estf}
\L_\p (\chi_F,\chi_G)\lesssim 2^{-3m}\,2^{\frac{m}{2}}\,(\#\mathcal{A}_l)^{\frac{5}{8}}\,(\#\mathcal{B}_s)^{\frac{5}{8}}\approx |F|^{\frac{5}{8}}\,|G|^{\frac{5}{8}}\,,
  \end{equation}
thus proving our desired estimate \eqref{Limp1simple}.

\subsection{Incidence geometry (II): strong $L^2$-decay bounds}\label{StrongL2}

In this section, we provide a very brief outline on how, under hypothesis \eqref{eq_f_a_E_cond}, we can obtain the direct (with no use of interpolation) $L^2$-strong decay bound
    \begin{equation}\label{glstrong}
        \abs{\Lambda(f,g)}
        \lesssim
        \lambda^{-\sigma}
        \nrm{f}_{L^2}\,\nrm{g}_{L^2}\,,
    \end{equation}
where here $\Lambda(f,g)=\ang{\mathcal{C}^{\br{a}}f,g}$.

Indeed, by taking $\lambda\approx 2^m$, it is not hard to see that under requirement \eqref{eq_f_a_E_cond}, one may reduce the analysis of  $\Lambda(f,g)$  to the situation when $\nrm{f}_{2}=\nrm{g}_{2}=1$ and, essentially, the moral spatial support of $f$ and $g$ is contained in the unit cube and the frequency support satisfies $\textrm{supp}\,\hat{f},\,\textrm{supp}\,\hat{g}\subseteq [2^m,\,2^{m+1}]^2$. This latter condition further implies that both $f$ and $g$ are morally constant on cubes of the form $I_{p,r}$ with $p,\,r\in\{1,\ldots, 2^m\}$. Hence, we further deduce that  for any $p,r$ and $x\in I_{p,r}$ we essentially have
$$|f(x)|^2\approx \frac{1}{|I_{p,r}|}\int_{I_{p,r}}|f|^2\,.$$
As a consequence,  for $l,s\geq 0$, we are naturally led to consider the level sets
\begin{align}\label{defAB}
    \mathcal{A}_l:=\left\{(p,r)\,\Big|\,\frac{1}{|I_{p,r}|}\int_{I_{p,r}}|f|^2\approx 2^l\, \|f\|_2^2\right\}\,,\nonumber\\
    \mathcal{B}_s:=\left\{(p,r)\,\Big|\,\frac{1}{|I_{p,r}|}\int_{I_{p,r}}|g|^2\approx 2^s\, \|g\|_2^2\right\}\,,
\end{align}%
and, following a similar line of thought with that in \cite{HL(CTHT)} (see also \cite{LVBilHilbCarl}), one can involve a \emph{spatial sparse-uniform dichotomy}, and, for some properly chosen $\ep>0$, define
\begin{itemize}
\item the spatial uniform component of $f$ by
\begin{equation}\label{funifU}
   f^U:=\sum_{l\leq \ep m} \sum_{(p,r)\in \mathcal{A}_l} f \chi_{I_{p,r}}\,;
\end{equation}%

\item the spatial sparse component of $f$ by
\begin{equation}\label{funifS}
   f^S:=\sum_{l> \ep m} \sum_{(p,r)\in \mathcal{A}_l} f \chi_{I_{p,r}}\,,
\end{equation}%
\end{itemize}
withe the obvious analogue for $g.$

With these, we let
\begin{equation}\label{dec}
   \Lambda(f,g)=\Lambda_U(f,g)\,+\,\Lambda_S(f,g)\,,
\end{equation}%
where
\begin{itemize}
\item the uniform component $\Lambda_U$ is given by $\Lambda_U(f,g):=\Lambda(f^U,g^U)$;

\item the sparse component $\Lambda_S$ is given by $\Lambda_S(f,g):=\Lambda(f^U,g^S)\,+\,\Lambda(f^S,g^U)\,+\,\Lambda(f^S,g^S)$.
\end{itemize}
Once at this point, we observe that the treatment of the uniform component $\Lambda_U(f,g)$ can be substituted to the treatment of the original $\Lambda(f,g)$ in the earlier sections but under more favorable conditions since we now have an $L^{\infty}$ control of the input functions. Indeed, we may pretend that both input functions in $\Lambda_U(f,g)$ are essentially in $L^{\infty}$  since $\|f^U\|_{L^{\infty}}\lesssim 2^{\ep m} \|f\|_2$ (and similarly for $g$) since the factor $2^{\ep m}$ is harmless.\footnote{Indeed, for $\ep>0$ small properly chosen, the $2^{\ep m}$ factor can be absorbed into the decaying factor appearing in \eqref{glstrong}.}

For the remaining sparse component, we use \eqref{FGsets} in order to further define--with the obvious correspondences--the sets $F_l$ and $G_s$, and notice that based on \eqref{Limp1simple}

\begin{equation}\label{lamsp}
    |\Lambda_S(f,g)|\lesssim \sum_{l+s\geq \ep m} 2^{\frac{l}{2}}\,2^{\frac{s}{2}}\,\Lambda_\p(\chi_{F_l},\chi_{G_s})\lesssim
    \sum_{l+s\geq \ep m} 2^{\frac{l}{2}}\,2^{\frac{s}{2}} |F_l|^{\frac{5}{8}}\, |G_s|^{\frac{5}{8}}\lesssim \sum_{l+s\geq \ep m} 2^{-\frac{l+s}{8}}\lesssim 2^{-\frac{\ep m}{10}}\,.
\end{equation}

\subsection{Treatment of the non-resonant Carleson-Radon transform $C_{[\vec{\a}]}$ for general $\vec{\a}$}\label{genalfa}

For the general $\vec{\a}$ setting, we first notice that the corresponding variant of \(M_\para, H^\ast_\para\) are of strong type \(\br{p,p}\), (\cite{sw1}), and thus one can perform the same reduction steps as in Section \ref{Genconsid} in order to reduce matters to the analogue of Theorem \ref{thm_main_loc}. Focussing thus on the latter, via a simple change of variable one reaches to
\begin{equation*}
    \mathcal{C}^{\br{a}}f\br{x,y}:=
    \int
        f\br{x-t,y-t^\alpha}e\br{a\br{x,y}t^\beta}\rho\br{t}
    dt
\end{equation*}
with \(\alpha:=\alpha_2/\alpha_1\neq 1\) and \(\beta:=\alpha_3/\alpha_1\in \R\setminus\BR{1,\alpha}\). Due to the symmetry between the \(x,y\) variables, one may further assume \(\alpha>1\).

Next, we notice that in order to derive the wave-packet discretized model form in the absence of polynomial structures, one can employ an almost identical argument with the one in Section \ref{subsec_wp}: indeed, due to the choice of the linearizing scale $(\sqrt{\l})^{-1}$,  the Taylor expansion is expressed as the linear term plus a second-order error term in the integral form. With some straightforward modifications, one then obtains
    \begin{equation*}
        \sum_{\substack{p,q,r\\
        u,v,w:\\
        \frac{\sqrt{\lambda}}{3}\leq r\leq 3\sqrt{\lambda}
        }}\hspace{-1em}
        \frac{
            \abs{
                \ang{\widetilde{f},\Phi_{p-r,q-r\br{\frac{r}{\sqrt{\lambda}}}^{\alpha-1},u,v}}
                \ang{
                    \overline{e\br{a\br{\frac{r}{\sqrt{\lambda}}}^{\beta}}}\chi\br{\frac{a}{\sqrt{\lambda}}-w}\chi_E,
                    \!\Psi_{p,q,u,v}
                }
            }
        }{
            \sqrt{\lambda}
            \ang{
                \beta w\br{\frac{r}{\sqrt{\lambda}}}^{\beta-1}
                -\alpha v\br{\frac{r}{\sqrt{\lambda}}}^{\alpha-1}
                -u
            }^N
        }
    \end{equation*}
where the two functions satisfy \(\big\vert\widehat{\widetilde{f}}\big\vert\lesssim \big\vert\widehat{f}\big\vert\) and \(\abs{\chi_E}\lesssim\abs{\1_E}\).

Assuming now that the analogue of Lemma \ref{lem_tf_cor_lev_set_ana} holds, the rest requires little adjustment. Indeed, to treat the above model, the only difference relative to the monomial case is the occasional use of the mean value theorem in various estimates, for example, in the analogues of \eqref{exp} and \eqref{eq_EE_w2v}. Moreover, we remark that the \(L^p\) improving estimates for averaging operators along planar curves are well-understood and in particular that Lemma \ref{thm_Lp_impro} holds in our general settings. Alternatively, one can also easily adapt the argument in Section \ref{Incimprovest} in order to obtain a weaker---albeit strong enough for our purposes---\(L^p\) improving estimate.

With these clarifications, it only remains to address the general version of Lemma \ref{lem_tf_cor_lev_set_ana}:

\begin{equation*}
  \nrm{K\br{x,y,t}}_{L^1\br{dxdydt,\I^2\times\A}} := \nrm{
        \ang{
            w\br{x+t,y+t^\alpha}t^{\beta-1}+
            v\br{x,y}t^{\alpha-1}+
            u\br{x,y}
        }^{-N}
    }_{L^1\br{dxdydt,\I^2\times\A}}\lesssim\lambda^{-\delta},
\end{equation*}
where, as before, we make the same assumption that \(\abs{u}\vee\abs{v}\vee\abs{w}\eqsim\lambda\).

Via the same reduction step, we may assume without loss of generality that \(\abs{u}\vee\abs{v}\eqsim \lambda\). Since \(\alpha,\beta\) may not be integers, we make a few adjustments while decoupling the smooth variables from the rough variables. We start by introducing an average in the \(t\) variables within a \(\lambda^{-\frac{1}{3}}\)--scale neighbourhood.
\begin{equation}\label{eq_K_les_K_avg}
    \nrm{K\br{x,y,t}}_{L^1\br{dxdydt,\I^2\times \A}}\lesssim \nrm{K\br{x,y,t+\frac{\tau}{\lambda^{\frac{1}{3}}}}}_{L^1\br{dxdyd\tau dt,\I^3\times \A}}.
\end{equation}
We then perform similar shifting, doubling, and then unshifting to obtain the following estimate:
\begin{align*}
    \eqref{eq_K_les_K_avg}^2
    \lesssim &
    \nrm{
        \prod_{j=0,1}
            K\br{x-\br{t+\frac{\tau_j}{\lambda^{\frac{1}{3}}}},y-\br{t+\frac{\tau_j}{\lambda^{\frac{1}{3}}}}^\alpha,t+\frac{\tau_j}{\lambda^{\frac{1}{3}}}}
    }_{L^1\br{dxdyd\tau_0 d\tau_1 dt}}\nonumber\\
    \lesssim &
    \lambda^{-2\delta}
    +
    \nrm{
        \prod_{j=0,1}
            K\br{x-t_j,y-t_j^\alpha,t_j}
    }_{
        L^1\br{
            \lambda^{\frac{1}{3}}dxdydt_0dt_1,
            \lambda^{-\frac{1}{3}-2\delta}\lesssim t_0-t_1 \lesssim \lambda^{-\frac{1}{3}}
        }
    }\nonumber\\
    \lesssim &
    \lambda^{-2\delta}
    +
    \nrm{
        K\br{x,y,t_+}
        K\br{
            x+t_+-t_-,
            y
            +t_+^\alpha
            -t_-^\alpha,
            t_-
        }
    }_{
        L^1\br{
            \lambda^{\frac{1}{3}}dxdydt_+dt_-,
            \lambda^{-\frac{1}{3}-2\delta}\lesssim t_+-t_- \lesssim \lambda^{-\frac{1}{3}}
        }
    }.
\end{align*}
For notation clarity, we again write
\begin{equation*}
    X:=x+t_+-t_-,\quad
    Y:=y
    +t_+^\alpha
    -t_-^\alpha.
\end{equation*}
We apply Lemma \ref{lem_jap_mul_diff} to further dominate the inner expression
\begin{equation*}
    K\br{x,y,t_+}
    K\br{
        X,Y,t_-
    }
    \lesssim
    \ang{
        \bm{t}^{\beta-1} \vec{1}
        \wedge
        \br{
            \vec{u}+
            \bm{t}^{\alpha-1}\vec{v}
        }
    }^{-N},
\end{equation*}
As a brief summary, we have obtained the estimate
\begin{equation*}
    \nrm{K}^2_{L^1}\lesssim
    \lambda^{-2\delta}
    +
    \nrm{
        \ang{
            \bm{t}^{\beta-1} \vec{1}
            \wedge
            \br{
                \vec{u}+
                \bm{t}^{\alpha-1}\vec{v}
            }
        }^{-N}
    }_{
    L^1\br{
            \lambda^{\frac{1}{3}} dxdydt_+ dt_-,
            \lambda^{-\frac{1}{3}-2\delta}\lesssim t_+-t_- \lesssim \lambda^{-\frac{1}{3}}
        }
    }.
\end{equation*}
Following the same line of logic as in the monomial case, we perform a change of variables, expressing \(\bm{t}\) in terms of \(x,y,X,Y\).
On the one hand, a direct calculation on the Jacobian shows that
\begin{equation*}
    \abs{
        \frac{
            \partial \br{X,Y}
        }{
            \partial \br{t_+,t_-}
        }
    }
    =\abs{
        \det
        \begin{pmatrix}
            1 & \alpha t_+^{\alpha-1}\\
            -1 & \alpha t_-^{\alpha-1}
        \end{pmatrix}
    }
    \eqsim \alpha\br{\alpha-1}\abs{t_+-t_-}\eqsim X- x.
\end{equation*}
Hence, we have
\begin{equation*}
    \lambda^{\frac{1}{3}}dxdydt_+dt_-
    \eqsim
    \lambda^{\frac{1}{3}} \frac{dxdX}{X-x}dydY
\end{equation*}
within the following regions:
\begin{equation*}
    \lambda^{-\frac{1}{3}-2\delta}
    \lesssim X-x \eqsim Y-y
    \lesssim \lambda^{-\frac{1}{3}}.
\end{equation*}
On the other hand, one cannot expect a closed form expressing \(\bm{t}\) precisely. However, due to the fact that
\begin{equation*}
    \abs{z-z'}\lesssim 1\implies \ang{z}^{-N}\eqsim \ang{z'}^{-N}
\end{equation*}
and that \(\abs{u}\vee\abs{v}\eqsim \lambda\), an approximation of \(\bm{t}\) precise up to an \(O\br{1/\lambda}\) error is sufficient for our purpose. To achieve this, we consider the average \(T:=\frac{t_++t_-}{2}\) and rewrite the expression via Taylor expansion
\begin{align*}
    Y-y= &
    \br{T+\frac{X-x}{2}}^\alpha-\br{T-\frac{X-x}{2}}^\alpha\\
    =&
    \alpha T^{\alpha-1} \br{X-x}
    +\frac{\alpha\br{\alpha-1}\br{\alpha-2}}{24}T^{\alpha-3}\br{X-x}^3
    +O\br{
        \begin{pmatrix}
            \alpha\\
            5
        \end{pmatrix}
    \br{X-x}/\lambda}.
\end{align*}
Divide the above expression by \(\alpha \br{X-x}\) gives the following:
\begin{align*}
    \frac{Y-y}{\alpha \br{X-x}}= & T^{\alpha-1}+ \frac{\br{\alpha-1}\br{\alpha-2}}{24}T^{\alpha-3}\br{X-x}^2 +
    O\br{
    \begin{pmatrix}
        \alpha-1\\
        4
    \end{pmatrix}
    /\lambda
    }\\
    = & \br{
        T^2+\frac{\alpha-2}{12}\br{X-x}^2
    }^{\frac{\alpha-1}{2}}+
    O\br{
    \begin{pmatrix}
        \alpha-1\\
        3
    \end{pmatrix}
    /\lambda
    }.
\end{align*}
Solving for \(T\) yields the following:
\begin{equation*}
    T=\br{\frac{Y-y}{\alpha \br{X-x}}}^{\frac{1}{\alpha-1}}-\frac{\alpha-2}{24}\br{\frac{Y-y}{\alpha \br{X-x}}}^{-\frac{1}{\alpha-1}}\br{X-x}^2 +
    O\br{
    \begin{pmatrix}
        \alpha-1\\
        2
    \end{pmatrix}
    /\lambda
    }.
\end{equation*}
Setting
\begin{equation*}
    T\br{x,y}:=\br{\frac{y}{\alpha x}}^{\frac{1}{\alpha-1}}-\frac{\alpha-2}{24}\br{\frac{y}{\alpha x}}^{-\frac{1}{\alpha-1}}x^2
    ,
\end{equation*}
we can now write
\begin{equation*}
    \bm{t}=
    T\br{X-x,Y-y}+\frac{\bm{j}\br{X-x}}{2}
    +
    O\br{
    \begin{pmatrix}
        \alpha-1\\
        2
    \end{pmatrix}
    /\lambda
    }.
\end{equation*}
Using the approximation above, we obtain
\begin{align*}
    \ang{
        \bm{t}^{\beta-1} \vec{1}
        \wedge
        \br{
            \vec{u}+
            \bm{t}^{\alpha-1}\vec{v}
        }
    }^{-N}
    \eqsim
    \Bigg<
        &
        \br{T\br{X-x,Y-y}+\frac{\bm{j}\br{X-x}}{2}}^{\beta-1}\vec{1}\\
        \wedge &
        \br{
            \vec{u}+
            \br{T\br{X-x,Y-y}+\frac{\bm{j}\br{X-x}}{2}}^{\alpha-1}\vec{v}
        }
    \Bigg>^{-N}=:K_2\br{x,y,X,Y}.
    \end{align*}
As a brief summary, it suffices to show the following:
\begin{equation*}
    \nrm{
        K_2\br{x,y,X,Y}
    }_{L^1\br{
        \lambda^{\frac{1}{3}}
        \frac{dxdX}{X-x}dy dY,
        \lambda^{-\frac{1}{3}-2\delta}\lesssim X-x\eqsim Y-y\lesssim \lambda^{-\frac{1}{3}}
    }}
    \lesssim \lambda^{-2\delta}.
\end{equation*}
We can mimic the argument for \(\br{\alpha,\beta}=\br{2,3}\) case. By tripling the \(Y\) variables, the above estimate becomes a consequence of the following inequality:
\begin{equation*}
    \nrm{
        \prod_{j=1}^3
        K_2\br{x,y,X,Y_j}
    }_{
        L^1\br{
            \lambda^{\frac{4}{3}} dxdXdy dY_1dY_2dY_3,
            \lambda^{-\frac{1}{3}-2\delta}\lesssim X-x\eqsim Y_j-y\lesssim \lambda^{-\frac{1}{3}}
        }
    }
    \lesssim \lambda^{-8\delta}
    .
\end{equation*}
Moreover, by Lemma \ref{lem_jap_mul_diff}, we further dominate the inner expression
\begin{align*}
    &
    \prod_{j=1}^3
    K_2\br{x,y,X,Y_j}\\
    \lesssim
    \Bigg<
        &
        \br{
            T\br{X-x,\bm{Y}-y}+\frac{X-x}{2}
        }^{\beta-1}
        \br{
            \vec{U}
            +
            \br{
                T\br{X-x,\bm{Y}-y}-\frac{X-x}{2}
            }^{\alpha-1}
            \vec{V}
        }\\
        \wedge &
        \br{
            T\br{X-x,\bm{Y}-y}-\frac{X-x}{2}
        }^{\beta-1}
        \vec{1}\\
        \wedge &
        \br{
            T\br{X-x,\bm{Y}-y}-\frac{X-x}{2}
        }^{\beta-1}
        \br{
            T\br{X-x,\bm{Y}-y}+\frac{X-x}{2}
        }^{\alpha-1}\vec{1}
    \Bigg>^{-N}=:K_6\br{x,y,X,\bm{Y}}.
\end{align*}
Thus, it remains to show
\begin{equation*}
    \nrm{
        K_6\br{x,y,X,\bm{Y}}
    }_{
        L^1\br{
            \lambda^{\frac{4}{3}} dxdXdy dY_1dY_2dY_3,
            \lambda^{-\frac{1}{3}-2\delta}\lesssim X-x\eqsim Y_j-y\lesssim \lambda^{-\frac{1}{3}}
        }
    }
    \lesssim \lambda^{-8\delta}
    .
\end{equation*}
We consider two scenarios.
\subsubsection{Parabola case}
If \(\alpha=2\), we may utilize the alternating nature of \(\wedge\) to simplify the expression
\begin{align*}
    K_6\br{x,y,X,\bm{Y}}
    =
    \Bigg<
        &
        \br{
            \frac{\bm{Y}-y}{2\br{X-x}}+\frac{X-x}{2}
        }^{\beta-1}
        \br{
            \vec{U}
            +
            \br{
                \frac{\bm{Y}-y}{2\br{X-x}}-\frac{X-x}{2}
            }
            \vec{V}
        }\\
        \wedge &
        \br{
            \frac{\bm{Y}-y}{2\br{X-x}}-\frac{X-x}{2}
        }^{\beta-1}
        \vec{1}
        \wedge
        \br{
            \frac{\bm{Y}-y}{2\br{X-x}}-\frac{X-x}{2}
        }^{\beta}\vec{1}
    \Bigg>^{-N}.
\end{align*}
After a change of variable, we have conditions \(X,y\in\I\) and \(\lambda^{-2\delta}\lesssim x \eqsim Y_j\lesssim 1\) in the following expression:
\begin{align*}
    &
    K_6\br{X-\frac{x}{\lambda^{\frac{1}{3}}},y-\frac{x^2}{\lambda^{\frac{2}{3}}},X,y+\frac{2\bm{Y}}{\lambda^{\frac{1}{3}}}}
    \\
    = &
    \ang{
        \br{
            \frac{\bm{Y}}{x}+
            \frac{x}{\lambda^{\frac{1}{3}}}
        }^{\beta-1}
        \br{
            \vec{\mathcal{U}}
            +
            \frac{\bm{Y}}{x}\cdot
            \vec{\mathcal{V}}
        }
        \wedge
        \br{
            \frac{\bm{Y}}{x}
        }^{\beta-1}
        \vec{1}
        \wedge
        \br{
            \frac{\bm{Y}}{x}
        }^{\beta}\vec{1}
    }^{-N}\\
    \eqsim &
    \ang{
        \br{
            \frac{\bm{Y}}{x}+
            \frac{x}{\lambda^{\frac{1}{3}}}
        }^{\beta-1}
        \cdot
        \br{
            \frac{\bm{Y}}{x}
        }^{2-\beta}
        \br{
            \vec{\mathcal{U}}
            +
            \frac{\bm{Y}}{x}\cdot
            \vec{\mathcal{V}}
        }
        \wedge
        \br{
            \frac{\bm{Y}}{x}
        }^2
        \vec{1}
        \wedge
        \br{
            \frac{\bm{Y}}{x}
        }^3\vec{1}
    }^{-N}\\
    \lesssim &
    \ang{
        \br{
            \bm{Y}^2+\bm{Y}
            \frac{\beta-1}{\lambda^{\frac{1}{3}}}x^2
            +\frac{\br{\beta-1}\br{\beta-2}}{2\lambda^{\frac{2}{3}}}x^4
        }
        \br{
            x\vec{\mathcal{U}}
            +
            \bm{Y}
            \vec{\mathcal{V}}
        }
        \wedge
        \bm{Y}^{2}
        \vec{1}
        \wedge
        \bm{Y}^3\vec{1}
    }^{-N}
    =:\mathcal{K}_6\br{x,y,X,\bm{Y}}\,,
\end{align*}
where two vectors are modified accordingly
\begin{equation*}
    \vec{\mathcal{U}}:=
    \vec{U}\br{X,y+\frac{2\bm{Y}}{\lambda^{\frac{1}{3}}}}
    ,\quad
    \vec{\mathcal{V}}:=\vec{V}\br{X,y+\frac{2\bm{Y}}{\lambda^{\frac{1}{3}}}}.
\end{equation*}
It remains to prove the normalized estimate
\begin{equation*}
    \nrm{
        \mathcal{K}_6\br{x,y,X,\bm{Y}}
    }_{L^1\br{
        dxdXdydY_1dY_2dY_3,
        \lambda^{-2\delta}\lesssim x\eqsim Y_j\lesssim 1
    }}
    \lesssim \lambda^{-8\delta},
\end{equation*}
whose argument is almost identical to the \(\br{\alpha,\beta}=\br{2,3}\) case.
\subsubsection{Non-parabola case}
If \(\alpha\neq 2\),
we shall perform Taylor expansion on the following expression:
\begin{align*}
    \br{
        T\br{X-x,\bm{Y}-y}\pm\frac{X-x}{2}
    }^{\gamma-1}
    = &
    \br{\frac{\bm{Y}-y}{\alpha\br{X-x}}}^{\frac{\gamma -1}{\alpha-1}}
    \pm
    \frac{\gamma-1}{2}
    \br{\frac{\bm{Y}-y}{\alpha\br{X-x}}}^{\frac{\gamma -2}{\alpha-1}}
    \br{X-x}\\
    + &\frac{\br{\gamma-1}\br{3\gamma-\alpha-4}}{24}
    \br{\frac{\bm{Y}-y}{\alpha\br{X-x}}}^{\frac{\gamma -3}{\alpha-1}}
    \br{X-x}^2+
    O\br{
        1/\lambda
    }.
\end{align*}
We again mimic what we have done for the \(\alpha=2\) case
\begin{align*}
    &
    K_6\br{X-\frac{x}{\lambda^{\frac{1}{3}}},y,X,y+\frac{\alpha \bm{Y}}{\lambda^{\frac{1}{3}}}}\\
    \lesssim &
    \Bigg<
        \Bigg\{
            \br{
                \bm{Y}^{\frac{2}{\alpha-1}}
                +
                \frac{\beta-1}{2}\cdot
                \bm{Y}^{\frac{1}{\alpha-1}}
                \cdot
                \frac{x^{\frac{\alpha}{\alpha-1}}}{\lambda^{\frac{1}{3}}}
                +\frac{\br{\beta-1}\br{3\beta-\alpha-4}}{24}
                \cdot
                \frac{x^{\frac{2\alpha}{\alpha-1}}}{\lambda^{\frac{2}{3}}}
            }
        x\vec{\mathcal{U}}\\
        &\phantom{<}+
        \Bigg(
            \bm{Y}^{\frac{2}{\alpha-1}}
            +\frac{\beta-\alpha}{2}\cdot \bm{Y}^{\frac{1}{\alpha-1}}\cdot
            \frac{x^{\frac{\alpha}{\alpha-1}}}{\lambda^{\frac{1}{3}}}+\Br{
                \frac{\br{\beta-1}\br{3\beta-\alpha-4}}{24}-\frac{\br{\alpha-1}\br{3\beta-\alpha-1}}{12}
            }
            \cdot\frac{x^{\frac{2\alpha}{\alpha-1}}}{\lambda^{\frac{2}{3}}}
        \Bigg)\bm{Y}\vec{\mathcal{V}}\Bigg\}\\
        &\wedge
        \br{
            \bm{Y}^{\frac{2}{\alpha-1}}
            -
            \frac{\beta-1}{2}\cdot
            \bm{Y}^{\frac{1}{\alpha-1}}
            \cdot
            \frac{x^{\frac{\alpha}{\alpha-1}}}{\lambda^{\frac{1}{3}}}
            +\frac{\br{\beta-1}\br{3\beta-\alpha-4}}{24}
            \cdot
            \frac{x^{\frac{2\alpha}{\alpha-1}}}{\lambda^{\frac{2}{3}}}
        }\vec{1}\\
        &\wedge
        \Bigg(
            \bm{Y}^{\frac{2}{\alpha-1}}-\frac{\beta-\alpha}{2}\cdot \bm{Y}^{\frac{1}{\alpha-1}}\cdot
            \frac{x^{\frac{\alpha}{\alpha-1}}}{\lambda^{\frac{1}{3}}}
            +\Br{
                \frac{\br{\beta-1}\br{3\beta-\alpha-4}}{24}-\frac{\br{\alpha-1}\br{3\beta-\alpha-1}}{12}
            }
            \cdot\frac{x^{\frac{2\alpha}{\alpha-1}}}{\lambda^{\frac{2}{3}}}
        \Bigg)\bm{Y}\vec{1}\Bigg>^{-N}\\
        =&: \mathcal{K}_6\br{x,y,X,\bm{Y}},
\end{align*}
where the two vectors are modified accordingly
\begin{equation*}
    \vec{\mathcal{U}}:=
    \vec{U}\br{X,y+\frac{\alpha\bm{Y}}{\lambda^{\frac{1}{3}}}}
    ,\quad
    \vec{\mathcal{V}}:=\vec{V}\br{X,y+\frac{\alpha\bm{Y}}{\lambda^{\frac{1}{3}}}}.
\end{equation*}
It remains to prove the normalized estimate
\begin{equation*}
    \nrm{
        \mathcal{K}_6\br{x,y,X,\bm{Y}}
    }_{
        L^1\br{
            dxdydXdY_1dY_2dY_3,
            \lambda^{-2\delta}\lesssim x\eqsim Y_j\lesssim 1
        }
    }
    \lesssim \lambda^{-8\delta}.
\end{equation*}
Exploiting the multilinearity of the alternating tensor products (determinant) gives the following expression:
\begin{equation}\label{eq_K6_rel_to_MN}
    \mathcal{K}_6\br{x,y,X,\bm{Y}}=
    \ang{
        \begin{pmatrix}
            1 &
            \frac{x^{\frac{\alpha}{\alpha-1}}}{\lambda^{\frac{1}{3}}} &
            \frac{x^{\frac{2\alpha}{\alpha-1}}}{\lambda^{\frac{2}{3}}}
        \end{pmatrix}
        \br{
            \bm{M}_{\bm{Y}}
            x\vec{\mathcal{U}}
            +\bm{N}_{\bm{Y}}\bm{Y}\vec{\mathcal{V}}
        }
        +O\br{1}
    }^{-N},
\end{equation}
where \(\bm{M}_{\bm{Y}},\bm{N}_{\bm{Y}}\in M_{3\times 3}\br{\R}\) are matrices with \(O\br{1}\) bounded generalized polynomials of \(Y_j\) as entries. We aim to show that neither of \(\det\bm{M}_{\bm{Y}}\) and \(\det\bm{N}_{\bm{Y}}\) are identically zero. We use the following fact: for \(O\br{1}\) bounded  vectors \(\vec{u},\vec{v}\in \R^3\),
\begin{align*}
    &
    \begin{pmatrix}
        1 &
        \frac{x^{\frac{\alpha}{\alpha-1}}}{\lambda^{\frac{1}{3}}} &
        \frac{x^{\frac{2\alpha}{\alpha-1}}}{\lambda^{\frac{2}{3}}}
    \end{pmatrix}
    \br{
        \bm{M}_{\bm{Y}}
        \vec{u}
        +\bm{N}_{\bm{Y}}\vec{v}
    }+O\br{1/\lambda}\\
    =&
    \Bigg\{
        \br{
            \bm{Y}^{\frac{2}{\alpha-1}}
            +
            \frac{\beta-1}{2}\cdot
            \bm{Y}^{\frac{1}{\alpha-1}}
            \cdot
            \frac{x^{\frac{\alpha}{\alpha-1}}}{\lambda^{\frac{1}{3}}}
            +\frac{\br{\beta-1}\br{3\beta-\alpha-4}}{24}
            \cdot
            \frac{x^{\frac{2\alpha}{\alpha-1}}}{\lambda^{\frac{2}{3}}}
        }
    \vec{u}\\
    &\phantom{<}+
    \Bigg(
        \bm{Y}^{\frac{2}{\alpha-1}}
        +\frac{\beta-\alpha}{2}\cdot \bm{Y}^{\frac{1}{\alpha-1}}\cdot
        \frac{x^{\frac{\alpha}{\alpha-1}}}{\lambda^{\frac{1}{3}}}+\Br{
            \frac{\br{\beta-1}\br{3\beta-\alpha-4}}{24}-\frac{\br{\alpha-1}\br{3\beta-\alpha-1}}{12}
        }
        \cdot\frac{x^{\frac{2\alpha}{\alpha-1}}}{\lambda^{\frac{2}{3}}}
    \Bigg)\vec{v}\Bigg\}\\
    &\wedge
    \br{
        \bm{Y}^{\frac{2}{\alpha-1}}
        -
        \frac{\beta-1}{2}\cdot
        \bm{Y}^{\frac{1}{\alpha-1}}
        \cdot
        \frac{x^{\frac{\alpha}{\alpha-1}}}{\lambda^{\frac{1}{3}}}
        +\frac{\br{\beta-1}\br{3\beta-\alpha-4}}{24}
        \cdot
        \frac{x^{\frac{2\alpha}{\alpha-1}}}{\lambda^{\frac{2}{3}}}
    }\vec{1}\\
    &\wedge
    \Bigg(
        \bm{Y}^{\frac{2}{\alpha-1}}-\frac{\beta-\alpha}{2}\cdot \bm{Y}^{\frac{1}{\alpha-1}}\cdot
        \frac{x^{\frac{\alpha}{\alpha-1}}}{\lambda^{\frac{1}{3}}}
        +\Br{
            \frac{\br{\beta-1}\br{3\beta-\alpha-4}}{24}-\frac{\br{\alpha-1}\br{3\beta-\alpha-1}}{12}
        }
        \cdot\frac{x^{\frac{2\alpha}{\alpha-1}}}{\lambda^{\frac{2}{3}}}
    \Bigg)\bm{Y}\vec{1}
\end{align*}
to compute the following matrices
\begin{equation*}
    \bm{M}_{\bm{Y}}
    \cdot
    \begin{pmatrix}
        \vec{1} & \bm{Y}\vec{1} & \bm{Y}^2\vec{1}
    \end{pmatrix}
    ,\quad
    \bm{N}_{\bm{Y}}
    \cdot
    \begin{pmatrix}
        \vec{1} & \bm{Y}\vec{1} & \bm{Y}^2\vec{1}
    \end{pmatrix}
\end{equation*}
as we did in \(\br{\alpha,\beta}=\br{2,3}\) case to derive the expression of \(\det \bm{M}_{\bm{Y}}\) and \(\det \bm{N}_{\bm{Y}}\) respectively. The alternating nature of \(\wedge\) simplifies the computation. Eventually, we obtain
\begin{align*}
    \det\bm{M}_{\bm{Y}}
    = &
    \det\br{
        \bm{Y}
    }^{\frac{2}{\alpha-1}}\cdot
    \begin{pmatrix}
        \br{\beta-1}\cdot\bm{Y}^{
            \frac{1}{\alpha-1},
            \frac{2}{\alpha-1},
            1+\frac{2}{\alpha-1}
        }\\
        -\frac{\br{\beta-1}\br{\beta-\alpha}}{2}\cdot
        \bm{Y}^{
            \frac{1}{\alpha-1},
            \frac{2}{\alpha-1},
            1+\frac{1}{\alpha-1}
        }
    \end{pmatrix}\\
    \wedge &
    \begin{pmatrix}
        \frac{2\beta-\alpha-1}{2}\cdot\bm{Y}^{
            1+\frac{1}{\alpha-1},
            \frac{2}{\alpha-1},
            1+\frac{2}{\alpha-1}
        }\\
        - \frac{\br{\beta-1}\br{2\beta-\alpha-1}}{4}\cdot
        \bm{Y}^{
            1+\frac{1}{\alpha-1},
            \frac{1}{\alpha-1},
            1+\frac{2}{\alpha-1}
        }
        +\frac{\br{\alpha-1}\br{3\beta-\alpha-1}}{12}
        \cdot
        \bm{Y}^{
            1,
            \frac{2}{\alpha-1},
            1+\frac{2}{\alpha-1}
        }
    \end{pmatrix}\\
    \det \bm{N}_{\bm{Y}}
    =&
    \begin{pmatrix}
        \frac{2\beta-\alpha-1}{2}\cdot
        \bm{Y}^{
            \frac{1}{\alpha-1},
            \frac{2}{\alpha-1},
            1+\frac{2}{\alpha-1}
        }\\
        - \frac{\br{\beta-\alpha}\br{2\beta-\alpha-1}}{4}\cdot
        \bm{Y}^{
            \frac{1}{\alpha-1},
            \frac{2}{\alpha-1},
            1+\frac{1}{\alpha-1}
        }
        -
        \frac{\br{\alpha-1}\br{3\beta-\alpha-1}}{12}
        \cdot
        \bm{Y}^{
            0,
            \frac{2}{\alpha-1},
            1+\frac{2}{\alpha-1}
        }
    \end{pmatrix}\\
    \wedge &
    \det\br{
        \bm{Y}
    }^{\frac{2}{\alpha-1}}\cdot
    \begin{pmatrix}
        \br{\beta-\alpha}\cdot\bm{Y}^{
            1+\frac{1}{\alpha-1},
            \frac{2}{\alpha-1},
            1+\frac{2}{\alpha-1}
        }\\
        -\frac{\br{\beta-\alpha}\br{\beta-1}}{2}\cdot
        \bm{Y}^{
            1+\frac{1}{\alpha-1},
            \frac{1}{\alpha-1},
            1+\frac{2}{\alpha-1}
        }
    \end{pmatrix}
    .
\end{align*}
One may now verify that the above two expressions are non-trivial generalized polynomials of \(\br{Y_1,Y_2,Y_3}\in \left[0,\infty\right)^3\) by considering the specific value of \(\bm{Y}=\begin{pmatrix}
    z & 0 & 0\\
    0 & 1 & 0\\
    0 & 0 & \epsilon
\end{pmatrix}=:\bm{Y}_\epsilon\) and evaluate the limits
\begin{equation*}
    \lim_{\epsilon \searrow 0} \det\br{\bm{Y}_\epsilon}^{\gamma_1}
    \det\bm{M}_{\bm{Y}_\epsilon},\quad
    \lim_{\epsilon \searrow 0} \det\br{\bm{Y}_\epsilon}^{\gamma_2} \det\bm{N}_{\bm{Y}_\epsilon}
\end{equation*}
for suitable choice of \(\gamma_1,\gamma_2\) depending on \(\alpha,\beta\). The rest follows from a more refined application of Van der Corput--type estimates addressing the generalized polynomial setting.


\bibliographystyle{plain}

\end{document}